\documentclass[10pt]{amsart}
\usepackage{graphicx} 
\usepackage{nag}
\usepackage[T1]{fontenc}
\usepackage{fouriernc}

\usepackage{amsmath}
\usepackage{amsthm}
\usepackage{amssymb}
\usepackage{amsfonts}
\usepackage{latexsym}
\usepackage{array}
\usepackage{verbatim}

\usepackage{color}
\usepackage{mysects}  
\usepackage{fancyhdr} 
\usepackage[labelformat=simple]{caption}
\usepackage{setspace}
\usepackage{bm}
\usepackage{enumitem}
\usepackage{microtype}

\theoremstyle{plain}
\newtheorem{theorem}[subsection]{Theorem}

\newtheorem*{theorem*}{Theorem}
\newtheorem{proposition}[subsection]{Proposition}
\newtheorem*{proposition*}{Proposition}
\newtheorem{lemma}[subsection]{Lemma}
\newtheorem*{lemma*}{Lemma}

\theoremstyle{definition}

\newtheorem{remark}[subsection]{Remark}

\newtheorem{rem}{Remark}

\newcommand{\0}{{(0)}}

\newcommand{\Hom}{\operatorname{Hom}}

\newcommand{\GL}{\operatorname{GL}}

\newcommand{\PSL}{\operatorname{PSL}}

\newcommand{\PSO}{\operatorname{PSO}}

\newcommand{\meas}{\operatorname{\meas}}
\newcommand{\mF}{\mathcal{F}}
\newcommand{\tmF}{\tilde{\mathcal{F}}}
\newcommand{\half}{\frac{1}{2}}

\newcommand{\nc}{\newcommand}
\nc{\cal}{\mathcal} 
\nc{\la}{\langle} \nc{\ra}{\rangle}
 \nc{\CA}{\cal A}
 \nc{\CBB}{\cal B}
\nc{\CDD}{\cal D}
\nc{\CE}{\cal E}
\nc{\CF}{\cal F} \nc{\CG}{\cal
G} \nc{\CH}{\cal H} \nc{\CI}{\cal I} \nc{\CJ}{\cal J}
\nc{\CK}{\cal K} \nc{\CL}{\cal L} \nc{\CM}{\cal M} \nc{\CN}{\cal
N} \nc{\CO}{\cal O} \nc{\CP}{\cal P} \nc{\CQ}{\cal Q}
\nc{\CR}{\cal R} \nc{\CS}{\cal S} \nc{\CT}{\cal T} \nc{\CU}{\cal
U} \nc{\CV}{\cal V} \nc{\CW}{\cal W} \nc{\CZ}{\cal Z}

\nc{\fa}{\mathfrak a} \nc{\fg}{\mathfrak g} \nc{\fk}{\mathfrak k}
\nc{\fh}{\mathfrak h} \nc{\fm}{\mathfrak m} \nc{\fn}{\mathfrak n}
\nc{\fA}{\mathfrak A} \nc{\fC}{\mathfrak C} \nc{\fI}{\mathfrak I}
\nc{\fL}{\mathfrak L} \nc{\fS}{\mathfrak S}
\nc{\fz}{\mathfrak z}

\nc{\nen}{\newenvironment} \nc{\ol}{\overline}
\nc{\ul}{\underline} \nc{\lra}{\longrightarrow}
\nc{\lla}{\longleftarrow} \nc{\Lra}{\Longrightarrow}
\nc{\Lla}{\Longleftarrow} \nc{\Llra}{\Longleftrightarrow}
\nc{\hra}{\hookrightarrow} \nc{\iso}{\overset{\sim}{\lra}}

\makeatletter
\@addtoreset{equation}{section}
\makeatother
\numberwithin{equation}{subsection}
 \nc{\ba}{\mathbb A}
 \nc{\bG}{\mathbb G}
 \nc{\bq}{\mathbb Q}
 \nc{\br}{\mathbb R}
 \nc{\bz}{\mathbb Z}
 \nc{\bc}{\mathbb C}
 \nc{\bn}{\mathbb N}
 \nc{\bX}{\mathbb{X}}
 \nc{\ck}{\mathcal{K}}
 \nc{\G}{\Gamma}
 \nc{\sm}{\setminus}
 \nc{\sub}{\subset}
 \nc{\lm}{\lambda}
  \nc{\Lm}{\Lambda}
 \nc{\al}{\alpha}
 \nc{\bt}{\beta}
 \nc{\om}{\omega}
 \nc{\dl}{\delta}
 \nc{\g}{\gamma}
 \nc{\Dl}{\Delta}
 \nc{\Om}{\Omega}
 \nc{\s}{\sigma}
 \nc{\ro}{\rho}
 \nc{\te}{\theta}
 \nc{\SLR}{\operatorname{SL}_2(\br)}
 \nc{\GLR}{\operatorname{GL}_2(\br)}
 \nc{\PGLR}{\operatorname{PGL}_2(\br)}
 \nc{\PSLR}{\operatorname{PSL}_2(\br)}
 \nc{\PSLZ}{\operatorname{PSL}_2(\bz)}
 \nc{\SLC}{\operatorname{SL}_2(\bc)}
 \nc{\uH}{\mathbb H}
 \nc{\fD}{\mathcal{D}}
 \nc{\fE}{\mathcal{E}}
 \nc{\fO}{\mathcal{O}}
 \nc{\haf}{\frac{1}{2}}
 \nc{\qtr}{\frac{1}{4}}
 \nc{\shaf}{{\scriptstyle\frac{1}{2}}}
 \nc{\hlm}{{\scriptstyle\frac{\lambda}{2}}}

 \nc{\8}{\infty}
 \nc{\7}{{-\infty}}
 \nc{\inv}{^{-1}}
 \nc{\eps}{\varepsilon}
 \nc{\aG}{\mathbf{G}}
 \nc{\spn}{\operatorname{Span}}
 \nc{\Cm}{\operatorname{CM}}

\begin{document}

\title{ Nodal domains of Maass forms II}
\author{Amit Ghosh, Andre Reznikov and Peter Sarnak}
\date{}
\setcounter{section}{0}
\setcounter{subsection}{0}
\setcounter{footnote}{0}
\setcounter{page}{1}

\begin{abstract}
In Part I we gave a polynomial growth lower-bound for the number of nodal domains of a Hecke-Maass cuspform in a compact part of the modular surface, assuming a Lindel\"{o}f hypothesis. That was a consequence of a topological argument and known subconvexity estimates,  together with new sharp lower-bound restriction theorems for the Maass forms. This paper deals with the same question for general (compact or not) arithmetic  surfaces which have a reflective symmetry. The topological argument is extended and representation theoretic methods are needed for the restriction theorems, together with results of Waldspurger. Various explicit examples are given and studied.
\end{abstract}

\maketitle
\thispagestyle{empty}

\pagestyle{fancy}
\fancyhead{}
\fancyhead[LE]{\thepage}
\fancyhead[RO]{\thepage}
\fancyhead[CO]{\small Nodal domains of Maass forms II}
\fancyhead[CE]{\small  Amit Ghosh, Andre Reznikov and Peter Sarnak}
\renewcommand{\headrulewidth}{0pt}
\headsep = 1.0cm
\fancyfoot[C]{}
\ \vskip 1cm
\renewcommand{\thefootnote}{ } 
\renewcommand{\footnoterule}{{\hrule}\vspace{3.5pt}} 


\renewcommand{\thefootnote}{\arabic{footnote}\quad } 
\setcounter{footnote}{0} 

\setcounter{section}{0}\setcounter{equation}{0}
\section{Introduction}
This paper is a continuation of our \cite{GRS1}, which will be referred to as I. The main result there gives a lower bound for the number of nodal domains for Maass forms on a compact part of the modular surface. We indicated in I that the techniques there can be modified and extended to deal with a general arithmetic surface. Our aim here is to carry this out. The treatment in I relies heavily on the Fourier expansions of the Maass forms in the cusp, and for a compact surface these are not available. Our primary focus here will be on the compact cases. For the most part we use the same notation and terminology as in I.

Let $\mathbb{X}=\Gamma\backslash\mathbb{H}$ be a finite area hyperbolic surface.  Let $\Delta$ be the Laplacian on functions of $\mathbb{X}$, and $\psi$ (or $\phi$) will denote $L^{2}$-eigenfunctions with eigenvalue $\lambda_{\psi}= \frac{1}{4} + t_{\psi}^{2} >0$\ (functions $\psi$ are also called  Maass forms \cite{Maass}). We let $\sigma :  \mathbb{X} \to \mathbb{X}$ be an isometric involution on $\mathbb{X}$ which is induced from a hyperbolic reflection on $\mathbb{H}$. Next, we let $\Sigma = \mathrm{Fix}(\sigma)$ denote the points of $\mathbb{X}$ fixed by $\sigma$; it consists of a union (possibly disconnected) of piecewise geodesics in $\mathbb{X}$ (see Section 2). Our main interest is when $\Sigma$ is compact unlike the case in I where $\Sigma$ ran into the cusps of $\mathbb{X}$. By $\phi$ we will always mean a $\sigma$-even eigenfunction of $\Delta$ (we could just as well deal with $\sigma$-odd as indicated in the Appendix of I).  The nodal domains $\Omega$  of $\phi$ are inert or split according as $\sigma(\Omega)\cap\Omega$ is equal to $\Omega$ or is empty. The method we use in I to produce inert nodal domains involved three parts. The first part is a topological reduction, the second an investigation of lower bounds of $L^{2}$-restrictions of $\phi$ to $\Sigma$. This last appeals to QUE (Quantum Unique Ergodicity \cite{Li}) but involves no arithmetic input. It is in the third step where we need $\phi$ to be a Hecke eigenfunction as well, that we assume $\mathbb{X}$ is arithmetic. This allows us to use proven QUE in these cases (\cite{Li, So10}),  subconvexity $L^{\infty}$-bounds for eigenfunctions (\cite{IS95}) and the theory of automorphic $L$-functions. All of these go into the proof of our main Theorem 1.4 below.

We turn to the precise statements of our results and an outline of the paper. In Section 2 we give an extension of Theorem 2.1 of I to a general surface $\mathbb{X}$ of genus $g$ and for a general $\Sigma$, which may be non-separating (for another proof of this extension in the separating case, see \cite{JZ}, Section 7);

\begin{theorem} Let $N_\mathrm{ in}(\phi)$ be the number of inert nodal domains of $\phi$  and let $n_{\phi}$ denote the number of sign-changes of $\phi$ along $\Sigma$ . Then
\[ N_\mathrm{ in}(\phi)\geq \frac{1}{2}n_{\phi} -g +1.
\]

\end{theorem}

In order to study the sign-changes of $\psi$'s along curves $\mathcal{C}$ in $\mathbb{X}$, we examine the restriction of $\psi$ and its unit normal derivative $\partial_{n}\psi$ to $\mathcal{C}$. By Green's theorem, separating variables and some elementary analysis we derive lower bounds for these restrictions for simple closed curves $\mathcal{C}$ of constant geodesic curvature $\kappa\neq 0$ on a hyperbolic surface $\mathbb{Y}$, which is either $\mathbb{H}$ or a cylinder. In each case, the curve divides $\mathbb{Y}$ into inside and  outside regions, $S^{+}$ and $S^{-}$. We have the following type of lower bound for $\kappa = \pm 1$ (Propositions 3.2 and 3.4), and \eqref{CH1} for other $\kappa$'s for which an extra term is needed.

\begin{theorem} Let $\mathbb{Y}$, $\mathcal{C}$ and $S^{+}$ be as above. Then, for any eigenfunction $\psi$ which is in $L^{2}\mathrm{ (}S^{+}\mathrm{ )}$, we have
\[
\int\limits_{\mathcal{C}}\left(|\psi(s)|^{2} + \frac{1}{\lambda}|\partial_{n}\psi(s)|^{2}\right) \mathrm ds \geq \iint\limits_{S^{+}}W(x)|\psi(x)|^{2}\ \mathrm d\mu,
\]
with respective invariant measures. The function $W(x)$ is a positive smooth weight function depending only on $\mathbb{Y}$ and $\mathcal{C}$, and not on $\psi$.
\end{theorem}

For our main application we need a lower bound as above for the case that $\mathcal{C}$ is a closed geodesic (i.e. $\kappa=0$). In this case an inequality as in Theorem 1.2 involving only a positive weight function is not valid (see Section 3.5). However as in I one can infer a similar lower bound if one inputs the equidistribution that comes from QUE. In I, this was achieved using some arithmetical input as well (and for the "infinite" closed geodesic). In Section 4 we use some representation theory and the full micro-local Wigner measures to prove;

\begin{theorem}\label{closed-g-intro} Let $\mathbb{X}$ be a hyperbolic surface, $\mathcal{C}$ a closed geodesic on $\mathbb{X}$ and $\psi_{\lambda}$ a sequence of eigenfunctions on $\mathbb{X}$ satisfying QUE (that is $\|\psi\|_{_2}=1$ and $\langle\mathrm{ Op(a)}\psi_{\lambda},\psi_{\lambda}\rangle \to \int\limits_{\mathrm{ T}^{*}_{1}(\mathbb{X})} \mathrm{ a}(\xi)\mathrm d\mu(\xi)$ as $\lambda \to \infty$ for any zeroth order  p.d.o with symbol a). Then for $\lambda_{\psi}$ large enough we have 
\[
\int\limits_{\mathcal{C}}\left(|\psi(s)|^{2} + \frac{1}{\lambda}|\partial_{n}\psi(s)|^{2}\right) \mathrm ds \gg 1,
\]
here the implied constant depending only on $\mathcal{C}$ and $\mathbb{X}$.
\end{theorem}

Starting with Section 5 and for the rest of the paper, we assume that the surface $\mathbb{X}$ is arithmetic (in fact a congruence surface), which carries an involution, and that $\phi$ is a $\sigma$-even and a Hecke eigenfunction. This allows us to bring in various arithmetical tools as in I. The most significant new one being Waldspurger's formula \cite{Wa} which allows one to express $\int_{\Sigma} \phi(s) \mathrm ds$ in terms of the central values of $L$-functions of  automorphic forms on $GL_{2}/K$, where $K$ is a quadratic extension associated to the closed geodesic $\Sigma$ in $\mathbb{X}$. With this and following the technique in I, we obtain our main result;

\begin{theorem}\label{nodal-domains-thm-intro} Let $\mathbb{X}$, $\sigma$ and $\phi_{\lambda}$ be as above. Then assuming the Lindel\"{o}f Hypothesis for the $L$-functions of the GL$_{2}$ cusp forms, we have that as $\lambda_{\phi}\to \infty$
\[
N_{\mathrm{ in}}(\phi_{\lambda}) \gg_{_ \eps} \lambda_{\phi}^{\frac{1}{27}- \eps},
\]
for any $\eps >0$. The implied constant depends only on $\eps$, $\mathbb{X}$ and $\sigma$.
\end{theorem}

\begin{remark}\label{sup-norm-remark}\ 

\noindent(a). As noted in I, this lower bound is no doubt far from the truth. From \cite{TZ09} it follows that $N_{\mathrm{ in}}(\phi_{\lambda}) \ll \lambda_{\phi}^{\frac{1}{2}}$ and this upper bound is probably sharp. As far as split nodal domains we do not know how to produce any of them in spite of the fact that they are probably the vast majority. The expectation is that $N_{\mathrm{ split}}(\phi_{\lambda}) \sim c\frac{\mathrm{ Area}(\mathbb{X})}{4\pi}\lambda_{\phi}$\ , where $c>0$ is the Nazarov-Sodin constant \cite{NS09}

\noindent(b). If $\mathbb{X}$ comes from a quaternion algebra defined over $\mathbb{Q}$ (rather than a more general totally real field $k$), then the exponent $\frac{1}{27}$ can be replaced by $\frac{1}{24}$, see Section \ref{waldspurger-sect} .
\end{remark}

\vspace{10pt}
In Section 6 we examine in some detail a variety of examples of arithmetic surfaces $\mathbb{X}$ with involutions $\sigma$ and the corresponding set $\Sigma$, to which our results apply. Examples are given, for non-compact $\mathbb{X}$ for which $\Sigma$ has  (i) disjoint components, (ii) is separating and (iii) is non-separating. The compact examples include a description of $\Sigma$ which is made of closed geodesic arcs. A number of these examples were investigated (including detailed figures) in Fricke and Klein \cite{FrickeKlein}.

To end, we point to a significant advance in the recent preprint \cite{JJ}. Using results in \cite{CTZ} they first give a localised (i.e. to a piece of geodesic) version of Theorem 1.3 which is valid more generally (i.e. is not restricted to hyperbolic surfaces). They then go on to introduce a novel method of studying weak limits of renormalizations of the restrictions of the $\phi_{\lambda}$'s in Theorem 1.4 to $\Sigma$, and in particular the positive definiteness of their Fourier transforms. Using this they prove without any hypothesis that in the setting of Theorem 1.4, $N_{\mathrm{ in}}(\phi_{\lambda})$ (and so afortiori $N(\phi_{\lambda})$ ) goes to infinity with $\lambda$!

\vskip 0.2in
\noindent{\small {\bf Acknowledgments.}}\\
AG thanks the Institute for Advanced Study (IAS) and Princeton University for their hospitality. He also gratefully acknowledges additional financial support from IAS, the College of A\&S and the Department of Mathematics of his home university, and the Simons Foundation for a Collaboration Grant.\\
AR was partially supported by the Veblen Fund at IAS, the ERC grant  291612  and by the ISF   grant 533/14.  AR and PS were also supported by a  BSF grant.\\
AR thanks A. Venkatesh for enlightening discussions of QUE which greatly simplified the proof of Theorem \ref{thm-geodesic-low-bound}.\\
PS was supported by the NSF.\\
The software Mathematica$^{\copyright}$ was used for computations and  some of the images.
\vskip 0.2in

\section{Geometric preliminaries}
\subsection{The fixed point set.}\quad \label{geom-pre}

Let $G = \PSLR$ and $K$ be the maximal compact subgroup of $G$. Let $\sigma$ denote the orientation reversing isometry on $G\slash K = \mathbb{H}$ given by $\sigma(z)=-\overline{z}$. Let $\Gamma$ be a discrete subgroup of $G$ satisfying the condition $\sigma\Gamma\sigma=\Gamma$, so that if $\left( \begin{smallmatrix} a&b\\ c&d \end{smallmatrix}\right)$ is in $\Gamma$, then so is  $\left( \begin{smallmatrix} a&-b\\ -c&d \end{smallmatrix}\right)$. For $\alpha$ and $\beta$ in $G$, we will write $\alpha \sim \beta \mod \Gamma$ to mean $\alpha = \gamma \beta$ for some $\gamma \in \Gamma$ (if the group is clear, we shall drop the ``mod $\Gamma$'').
 We denote by $\hat{\alpha}$ the element $\sigma\alpha\sigma$, so that for the groups under consideration,  $\hat{\alpha}\in \Gamma$  if $\alpha \in \Gamma$. It is clear that $\widehat{\alpha\beta}=\hat{\alpha}\hat{\beta}$, so that if $\alpha \sim \beta$, then $\hat{\alpha} \sim \hat{\beta}$.
We note that $\hat{k} = k^{-1}$ for all $k\in K$, and if $A$ denotes the diagonal matrices in $G$ then $\hat{a}=a$ for all $a\in A$.  Also for $S=\left( \begin{smallmatrix} 0&-1\\ 1&0 \end{smallmatrix}\right)$ one has $\hat{S}=S$ in $G$. 

Let $\mathbb{X} = \Gamma\backslash\mathbb{H} = \Gamma\backslash G \slash K$. Then $\sigma$ lifts to $\sigma^{*}$  on $\mathbb{X}$, and is given by $\sigma^{*}(\Gamma z)=\Gamma\sigma z$. Since $\sigma\Gamma\sigma=\Gamma$, this is well-defined as $\Gamma z = \Gamma w$ implies $\Gamma\sigma z=\Gamma\sigma w$.

For any map $f : M \to M$, let $\mathrm{ Fix}(f;M) = \{z\in M : f(z)=z\}$. Then if $\mathcal{I}$ denotes the imaginary axis in $\mathbb{H}$, we have $\mathrm{ Fix}(\sigma;\mathbb{H})= \mathcal{I}$. We now prove the following

\begin{proposition}\label{p1} Suppose $\Gamma$ is a cocompact, torsion free discrete subgroup of $G$ satisfying $\sigma \Gamma\sigma = \Gamma$.  Then there exist a finite set of elements $h_{1}$, $h_{2}$, $\ldots$, $h_{R}$ in $G$ such that 
\[ \mathrm{ Fix}(\sigma^{*},\mathbb{X}) = \bigcup_{i=1}^{R} \Gamma h_{i}\mathcal{I}\ ,
\]
with the latter a disjoint union, with $\hat{h}_{i}h_{i}^{-1}\in\Gamma$ for all $i$.
\end{proposition} 

We require the following 

\begin{lemma} \label{lem1}For an arbitrary $\Gamma < G$, $\Gamma w\in \mathrm{ Fix}(\sigma^{*},\mathbb{X})$ if and only if there exists a $g\in G$ with $w=gi$ and $\hat{g}g^{-1}\in\Gamma$. Moreover, if $\Gamma$ is torsion-free, there are exactly two elements $g$ and $gS$ satisfying the conditions.

\begin{proof} If $g$ exists, then writing $\hat{g}=\gamma g$, we have $\sigma^{*}\Gamma w = \Gamma \sigma gi = \Gamma\hat{g}i = \Gamma\gamma gi = \Gamma w$ so that $\Gamma w\in \mathrm{ Fix}(\sigma^{*},\mathbb{X})$. 

Conversely, $\Gamma w\in \mathrm{ Fix}(\sigma^{*},\mathbb{X})$ implies $\Gamma w=\sigma^{*}\Gamma w = \Gamma \sigma w$, so that $\Gamma hi = \Gamma \hat{h}i$ for any $h\in G$ with $w=hi$. Then there exists $k\in K$ so that $\hat{h}kh^{-1}\in \Gamma$. Writing $k=\eta^{-2}$ with $\eta\in K$ and putting $g=h\eta$ gives us the conclusion for the first part. 

For the second part, suppose $\Gamma w\in \mathrm{ Fix}(\sigma^{*},\mathbb{X})$ with $w=g_{1}i=g_{2}i$. Putting $h= g_{2}^{-1}g_{1}$, we see that $h\in K$ and $\hat{h}h^{-1}\in g_{2}^{-1}\Gamma g_{2}$, so that $h^{-2}\in g_{2}^{-1}\Gamma g_{2} \cap K$, a finite group. Since $\Gamma$ is torsion-free, so is $g_{2}^{-1}\Gamma g_{2}$ and so  $h^{2} = I$. Hence $g_{1}=g_{2}$ or $g_{1}=g_{2}S$. 

\end{proof}

\end{lemma}
\noindent {\bf Proof of  Prop.\ref{p1}.}\  Let $g\in G$ satisfy the conditions in the lemma for a given $w$, so that if we identify $g$ with $gS$,  the former is determined uniquely by $w$.  For each $a\in A$, let $w_{a}= gag^{-1}w = gai$. Then since any element in $A$ commutes with $S$, one can identify $w_{a}$ with $ga$. Moreover, $\widehat{ga}(ga)^{-1}\in\Gamma$ so that $w_{a}\in  \mathrm{ Fix}(\sigma^{*},\mathbb{X})$. Thus, $\Gamma w_{a}$ with $a\in A$ determines a curve  $\Gamma w_{A}$ in $\mathbb{X}$ consisting of fixed points of $\sigma^{*}$ for any $\Gamma w\in \mathrm{ Fix}(\sigma^{*},\mathbb{X})$. 

These curves are disjoint as follows: suppose $\Gamma z\in \mathrm{ Fix}(\sigma^{*},\mathbb{X})$ with associated curve $\Gamma z_{A}$ and suppose $\Gamma w_{A}$ and $\Gamma z_{A}$ have a point in common. Then there exist $\gamma_{j}\in \Gamma$, $a_{j}\in A$ and $g_{j}\in G$ satisfying $\gamma_{1}g_{1}a_{1}i= \gamma_{2}g_{2}a_{2}i$, with $\hat{g}_{j}g_{j}^{-1}\in\Gamma$ for $j=1$ and $2$, and with $g_{1}i=w$ and $g_{2}i=z$. It follows that if $h_{j}= g_{j}a_{j}$, then there is a $\gamma \in \Gamma$  with $h_{2}^{-1}\gamma h_{1}\in K$. Thus $\gamma h_{1}i = h_{2}i$, with $\widehat{\gamma h}_{1}(\gamma h_{1})^{-1}\in \Gamma$ and $\hat{h}_{2}h_{2}^{-1}\in\Gamma$. Since $\Gamma h_{2}i = \Gamma h_{1}i \in \mathrm{ Fix}(\sigma^{*},\mathbb{X})$, it follows from Lemma \ref{lem1} that $\gamma h_{1} = h_{2}$ or $\gamma h_{1} = h_{2}S$. Since $S$ commutes with $A$, we have either $\Gamma g_{1}A = \Gamma g_{2}A$ or $\Gamma g_{1}A = \Gamma g_{2}AS$, so that $\Gamma z_{A}=\Gamma w_{A}$.

Let $A^{*} = A \cup AS$, the subgroup of $G$ mapping $\mathcal{I}$ to itself. Then from the discussion above, we have 
\[ \mathrm{ Fix}(\sigma^{*},\mathbb{X}) = \bigcup_{h\in \Gamma\backslash G\slash A^{*}} \Gamma h\mathcal{I},
\]
where in the (disjoint) union, $h$ satisfies $\hat{h}h^{-1}\in \Gamma$. The set $\mathcal{S}$ of representatives $h$ forms a discrete set, since if $ h\in \mathcal{S}$ and $h_{i}\to h$ in $\mathcal{S}$, then putting $\gamma_{i}= \hat{h}_{i}h_{i}^{-1}$ and $\gamma = \hat{h}h^{-1}$ in $\Gamma$ gives us $\gamma_{i}\to\gamma$. Discreteness of $\Gamma$ implies  for all $i$ and $j$ large enough, $\gamma_{i}=\gamma_{j}$, that is $\widehat{h_{j}^{-1}h_{i}}=h_{j}^{-1}h_{i}$, so that $h_{i}=h_{j}a$ for some $a\in A$.  Then $\Gamma h_{i}A^{*}=\Gamma h_{j}A^{*}$. Since $\Gamma\backslash \mathbb{H}$ is compact, so is $\mathcal{S}$ and finiteness follows.

The proposition above shows that $\mathrm{ Fix}(\sigma^{*},\mathbb{X}) $ is a finite union of connected components, each an image of a geodesic in $\mathbb{H}$. We now show

\begin{theorem}\label{thm1} For  $\Gamma$ is a cocompact, torsion free discrete subgroup of $G$ satisfying $\sigma \Gamma\sigma = \Gamma$, the connected components of $\mathrm{ Fix}(\sigma^{*},\mathbb{X}) $ are closed geodesics in $\mathbb{X}$.
\end{theorem}
\begin{rem} If \ $\Gamma$ is not torsion free, then each connected component is the image of a geodesic arc (explicit examples are given in Section 6). If there is any advantage to having closed geodesics as components of the fixed point set,  as considered in  Sections 4 and 5, one may do so as follows. By the Malcev-Selberg lemma, there is a subgroup $\Gamma ' < \Gamma$ of finite index that is torsion free. However we need $\Gamma '$ to satisfy the normalizing condition $\sigma\Gamma '\sigma = \Gamma '$.  Since $\sigma\Gamma '\sigma$ is also a finite index subgroup of $\Gamma$, the set $\Gamma^{*} = \Gamma ' \cap \sigma\Gamma '\sigma$ is a finite indexed torsion free subgroup of $\Gamma$ satisfying $\sigma \Gamma^{*}\sigma = \Gamma^{*}$.  Thus, one may descend to such a subgroup if necessary.
\end{rem}
\begin{proof} Let $\eps_{\mathbb{X}} >0$ be the injectivity radius of $\mathbb{X}$.  Choose a connected component $\mathcal{C}$ in $\mathrm{ Fix}(\sigma^{*},\mathbb{X}) $, with say the point $\Gamma z$ on it. As in Prop. \ref{p1},  let $h\in \Gamma\backslash G\slash A^{*}$ with $hi=z$, so that $\mathcal{C}=\Gamma h \mathcal{I}$.   Let $0<\eps<\frac{1}{2}\eps_{\mathbb{X}}$ and subdivide $\mathcal{I}$ into subintervals $\mathcal{I}_{j}$ of length $\eps$ with $i$ as an endpoint. The subintervals determine discs $\mathcal{D}_{j}$  in $\mathbb{X}$ covering $\mathcal{C}$ with $1\leq j\leq R$,  with each disc of radius $\frac{1}{2}\eps$ mapped locally homeomorphically to a disc in $\mathbb{H}$ , each containing a segment of $\mathcal{I}$.   By choosing $\eps$ small enough, we can ensure that each $\mathcal{D}_{j}$ contains no other  arcs of components of $\mathrm{ Fix}(\sigma^{*},\mathbb{X}) $ except $\mathcal{C}$. 

Let $l>1$ be a parameter chosen sufficiently large later on and let $\mathcal{I}_{l}$ consist of the subinterval in $\mathcal{I}$ with endpoints at $i$ and $il$.  For each $j$, let $n_{j}$ denote the number of images of the subintervals of $\mathcal{I}_{l}$ in $\mathcal{A}_{j}$, and of each such image, let $l_{jk}$ denote the length with $1\leq k\leq n_{j}$. Then
\[
\log{l} = \sum_{j=1}^{R}\sum_{k=1}^{n_j}l_{jk} \leq C\sum_{j=1}^{R}n_{j},
\]
for some constant $C>0$ independent of $l$. Moreover, 
\[
\sum_{j=1}^{R}\mathrm{ Area}(\mathcal{A}_{j}) \ll \mathrm{ Area}(\mathbb{X}),
\]
so that there is a constant $D_{\eps}>0$ with  $R\leq D_{\eps} \mathrm{ Area}(\mathbb{X})$. Choosing $l$ large enough ensures that $\sum_{j=1}^{R}n_{j} \geq R+1$, so that there is a disc $A_{0}$ containing at least two images of subintervals of $\mathcal{I}$.  Since locally each neighbourhood of $\mathcal{I}$ contains only one segment of $\mathcal{I}$, by the homeomorphism $h$ the images in $A_{0}$ must be identical, so that $\mathcal{C}$ is closed.
\end{proof}

\subsection{The nodal count.}\label{nodal-count-sect}\quad 

Let $\mathbb{Y}$ be any compact surface of genus $g\geq 0$, possibly with boundary. Let $\mathfrak{C}$ be a finite set of ovals (i.e. closed simple curves) on $\mathbb{Y}$ satisfying (i) two distinct ovals have no common points and (ii) ovals never have boundary points. Let $N(\mathbb{Y},\mathfrak{C})$ denote the number of connected components of $\mathbb{Y}\backslash\mathfrak{C}$. Then we have

\begin{proposition}\label{p2} $N(\mathbb{Y},\mathfrak{C})\geq |\mathfrak{C}| -g +1$.

\end{proposition}
\begin{proof}
We proceed by induction on $g$. If $g=0$, then $N=|\mathfrak{C}|+1$ as required. Assume the result for any such surface with genus less than $g\geq 1$. Let $\gamma \in \mathfrak{C}$. There are two cases depending on whether $\gamma$ is separating or not. 

If $\gamma$ does not separate, so that $\mathbb{Y}\backslash \gamma$ is connected, we cut along $\gamma$ to obtain a new surface $\mathbb{Y}_{1}$ with genus $g-1$ with oval set $\mathfrak{C}_{1}=\mathfrak{C}\backslash\gamma$. There are now $|\mathfrak{C}|-1$ ovals on $\mathbb{Y}_{1}$, and they do not intersect the boundary due to the assumptions made for $\mathfrak{C}$. Then, by induction, 
\[
N(\mathbb{Y},\mathfrak{C})= N(\mathbb{Y}_{1},\mathfrak{C}_{1})\geq |\mathfrak{C}_{1}|-(g-1)+1 = |\mathfrak{C}|-1-(g-1)+1=|\mathfrak{C}|-g+1.
\]
Next, suppose $\gamma$ separates $\mathbb{Y}$ so that $\mathbb{Y}=\mathbb{Y}_{1} \oplus \mathbb{Y}_{2}$ with corresponding genus $g_{i}$ satisfying $g=g_{1}+g_{2}$. The ovals remaining in $\mathfrak{C}\backslash \gamma$ separate into two subsets so that $\mathfrak{C}\backslash \gamma = \mathfrak{C}_{1}\cup \mathfrak{C}_{2}$ with $|\mathfrak{C}_{1}|+ |\mathfrak{C}_{2}|=|\mathfrak{C}|-1$. If both $g_{1}$ and $g_{2}$ are not zero, we can proceed inductively to get
\[
N(\mathbb{Y},\mathfrak{C})= N(\mathbb{Y}_{1},\mathfrak{C}_{1})+N(\mathbb{Y}_{2},\mathfrak{C}_{2})\geq (|\mathfrak{C}_{1}|-g_{1}+1)+(|\mathfrak{C}_{1}|-g_{2}+1)=|\mathfrak{C}|-g+1.
\]
If either $g_{1}=0$ or $g_{2}=0$ (and not both), then $\mathbb{Y}_{1}$ or $\mathbb{Y}_{2}$ determines a connected component of $\mathbb{Y}\backslash\mathfrak{C}$. Thus we conclude that either $\gamma$ determines a connected component of $\mathbb{Y}\backslash\mathfrak{C}$ or if not then necessarily $N(\mathbb{Y},\mathfrak{C})\geq |\mathfrak{C}| -g +1$. We now repeat this argument with another choice for $\gamma$. Repeating at least $|\mathfrak{C}| -g +1$ times gives us the result.
\end{proof}

Suppose $\mathbb{Y}$ has an involution $\omega$.

\begin{lemma}\label{lem2} A simple $\omega$-invariant, closed curve on $\mathbb{Y}$ can intersect $\mathrm{ Fix}(\omega,\mathbb{Y}) $ in at most two points.
\end{lemma}

\begin{proof}Let $\mathcal{C}$ be a simple closed curve in $\mathbb{Y}$ that is $\omega$-invariant. Let $\eta : S \to \mathcal{C}$ be an injective homeomorphism from the unit circle to $\mathcal{C}$. Then $\eta^{-1}\omega\eta : S\to S$ is an involution on the circle. Suppose $\mathcal{C}$ intersects $\mathrm{ Fix}(\omega,\mathbb{Y}) $ at $z$, say and put $z=\eta(t)$. Then $\omega(z)=z$ implies $\eta^{-1}\omega\eta(t)=t$ so that $t$ is a fixed point of an involution of $S$. Hence, the number of $t$'s does not exceed two.
\end{proof}

\begin{rem} The Prop.\ref{p2} and Lemma \ref{lem2} above are valid on non-compact surfaces $\mathbb{Y}$, where we allow a finite number of cusps. An additional requirement is that the set of ovals $\mathfrak{C}$ satisfy the  condition that for any cusp, there is a neighourhood devoid of points from $\mathfrak{C}$. Then, after compactification the arguments above hold for such surfaces.
\end{rem}

Let $\mathbb{Y}$ be a genus $g\geq 0$, compact orientable Riemannian manifold of dimension two, and suppose it has an orientation reversing isometry $\sigma$. Let $\phi$ be an even eigenfunction of the Laplace operator on $\mathbb{Y}$, so that $\phi(\sigma z)=\phi(z)$. The nodal set $\mathcal{Z}_{\phi}$ of $\phi$ consists of smooth curves with possible self-intersections. The complement $\mathbb{Y}\backslash\mathcal{Z}_{\phi}$ is a finite union of connected components and its number is denoted by $N(\phi)$. To obtain a lower bound for $N(\phi)$, the argument in \cite{GRS1} shows that we may deform the nodal set to remove the singularities without increasing the number of nodal domains. Thus, in what follows, we may assume that the deformed curves are smooth and simple. Since $\phi$ is even, the nodal set is invariant under $\sigma$. 

Let $\Sigma$ denote  $\mathrm{ Fix}(\sigma,\mathbb{Y}) $. In general, it is known that $\Sigma$ is the disjoint union of a finite number (not exceeding $g+1$) of curves and that the nodal set intersects $\Sigma$ orthogonally. If $z$ is a sign-change of $\phi$ on $\Sigma$, then a nodal curve $\mathcal{C}_{z}$ emanates from it and must terminate at $\Sigma$. The reflected curve $\sigma\mathcal{C}_{z}$ has the same endpoints so that we obtain a simple smooth closed curve, an oval, which determines an $inert$ nodal domain. These ovals satisfy the conditions of Lemma \ref{lem2} so that they have at most 2 intersections with $\Sigma$. Let $n_{\phi}$ denote the total number of sign-changes of $\phi$ on $\Sigma$. Let $A$ be the number of ovals with intersection number one, and $B$ with number two. Then $A +2B= n_{\phi}$ while the number of ovals is $A+B$, so that $|\mathfrak{C}| \geq \frac{1}{2}n_{\phi}$ in Prop.\ref{p2}. Then, Prop.\ref{p2} gives us Theorem 1.1 .

\begin{rem} As mentioned before, this is valid for noncompact manifolds with a finite number of cusps provided that there are no nodal curves in a neighbourhood of each cusp.
\end{rem}

\subsection{Trees and connectivity}\quad 

Consider the nodal set of a function $\phi$ on a surface $\mathbb{X}$, and assume it is non-singular. We construct a tree associated with the nodal domains as follows: to each nodal domain we associate a vertex $v$, with edges connecting vertices that correspond to neighbouring nodal domains (with a common boundary), so that the degree $\mathrm{ deg}(v)$ counts the number of neighbours of a given nodal domain. If we remove a vertex $v$, the tree decomposes into $\mathrm{ deg}(v)$ connected components, and so the average connectivity is given by $\frac{2|E|}{|V|}$. where $|E|$ and $|V|$ are the number of edges and vertices respectively. Since the number of faces and the genus is bounded, Euler's formula shows that the average connectivity is $2 + O(\frac{1}{|V|})$. This tends to $2$ as the number of nodal domains grow. We also note that the tree contains a connected subtree consisting of the vertices corresponding to the inert nodal domains (since the involution fixes the corresponding vertices and edges), and the average connectivity of this subtree is also asymptotically $2$.

In Figure 1, we have a part of the nodal set for an eigenfunction corresponding to the eigenvalue $\lambda = \frac{1}{4} + (125.52)^{2}$ on the modular surface from \cite{HR92}, with the associated tree in Figure 2. The dark dots denote the inert domains while the light dots one-half of the split domains (the split domains occuring in pairs). The numbers at the nodes count the total number of  neighbours and also the number of inert domains, if there is a mixture.

\begin{figure}[h!]
\centering
\begin{minipage}{.6\linewidth}
  \includegraphics[width=\linewidth]{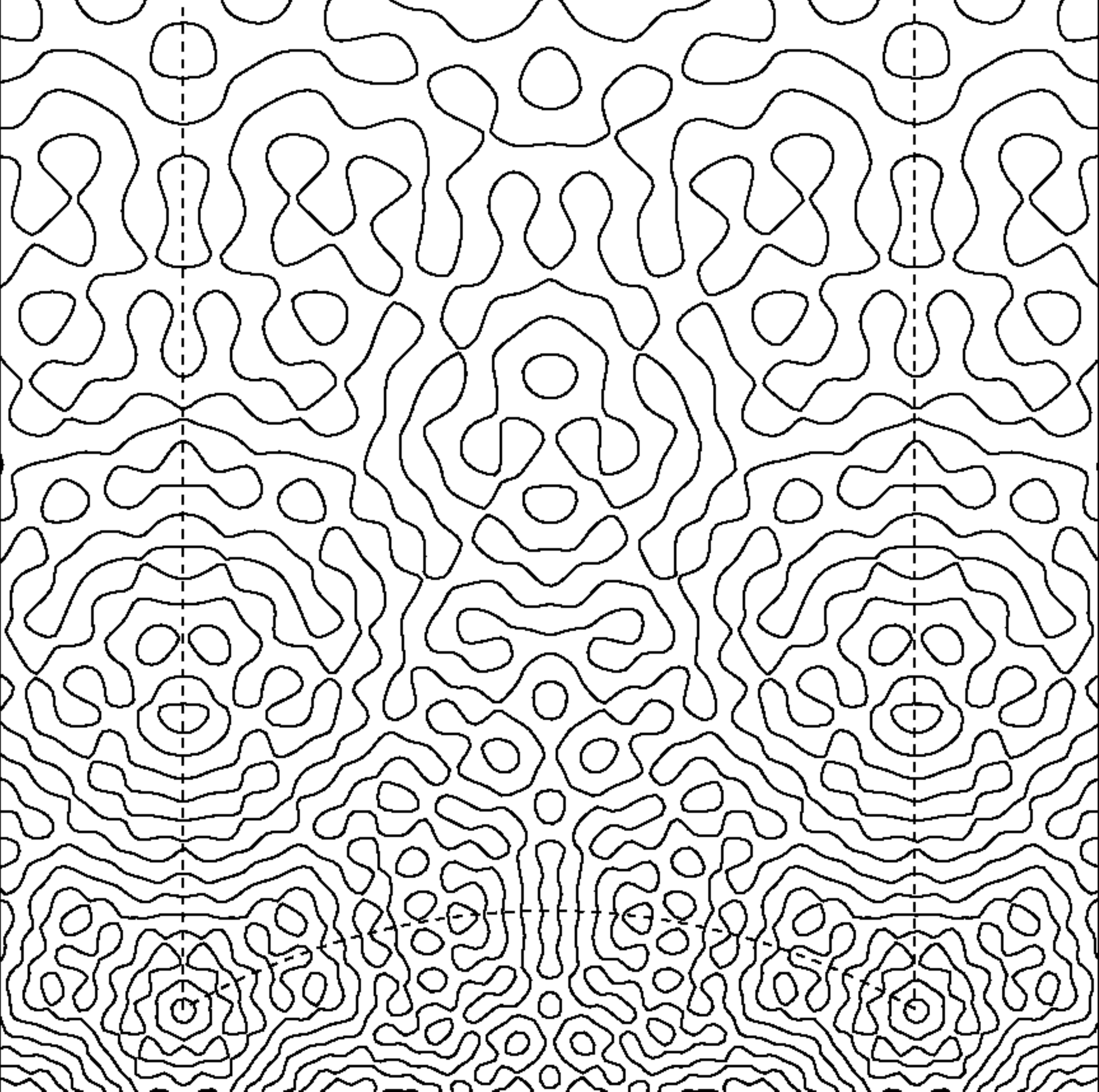}
  \end{minipage}
  \caption{$\lambda = \frac{1}{4}+ (125.52)^{2}$ from \cite{HR92}}
  \label{img1}

\end{figure}
\begin{figure}[h!]
\centering
\begin{minipage}{\linewidth}
  \includegraphics[width=0.9\linewidth]{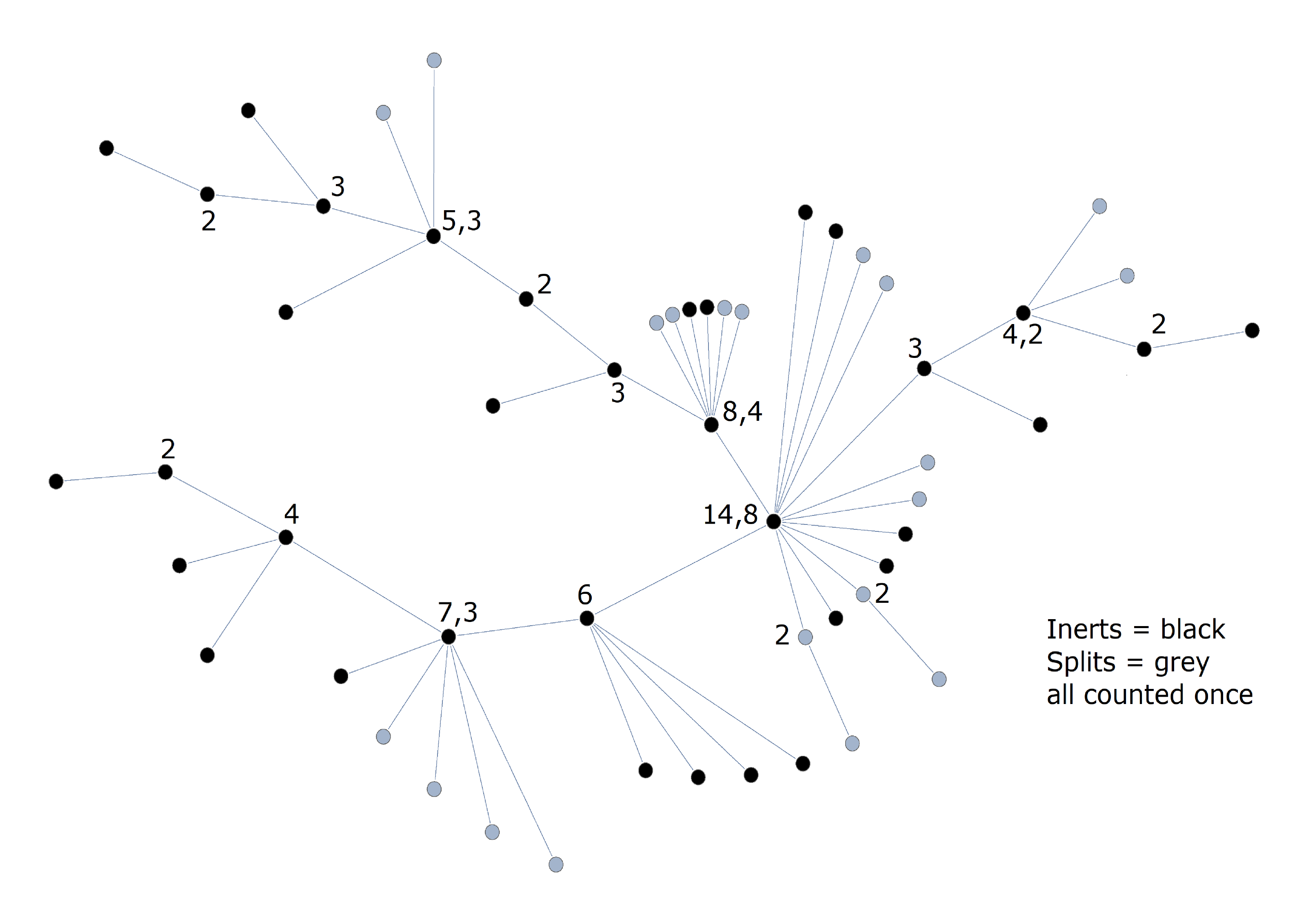}
  \end{minipage}
  \label{img2}

\caption{Associated tree of Fig. 1}
\end{figure}


\section{Soft restriction lower bounds}

In this section, we give some explicit lower bounds for certain combinations of the $L^{2}$-norms of an eigenfunction and its normal derivative, restricted to some chosen closed curves on surfaces. More precisely, we consider certain families of closed curves $\mathcal{C}_{\zeta}$ depending on a parameter $\zeta \geq 0$ and, for a real valued eigenfunction $\psi$ with eigenvalue $\lambda_{\psi} >0$ we construct a  functional $\mathcal{G}(\psi,\zeta)$ given by
\[
\mathcal{G}(\psi,\zeta) = a(\zeta)\int\limits_{\mathcal{C}_{\zeta}} \psi^{2} \mathrm ds + \frac{b(\zeta)}{\lambda_{\psi}}\int\limits_{\mathcal{C}_{\zeta}}(\partial_{n}\psi)^{2} \mathrm ds - \frac{c(\zeta)}{\lambda_{\psi}}\int\limits_{\mathcal{C}_{\zeta}}(\partial_{s}\psi)^{2}\ \mathrm ds ,
\]
where $a,b$ and $c$ are explicit non-negative functions of $\zeta$, independent of $\psi$, with $\partial_{n}\psi$ a normal derivative and $\partial_{s}\psi$ a tangential  derivative. 

For the cases under consideration,  $\mathcal{G}(\psi,\zeta)$ is an increasing function of $\zeta$; we find that $\mathcal{G}(\psi,\zeta) - \mathcal{G}(\psi,0) $ is the $L^{2}$-integral of a positively weighted multiple of $\psi^{2}$ (by Green's theorem). In the first two examples, the curve $\mathcal{C}_{0}$ is a point and so the functional  vanishes there, giving us a suitable lower bound for the restrictions. For the last case,  $\mathcal{C}_{0}$ is a (closed) geodesic and  we cannot show that $\mathcal{G}(\psi,0)\geq 0$.   

\subsection{Hyperbolic circles}\quad

Let $\mathcal{C}_{R}$ be a hyperbolic circle on $\mathbb{H}$ centered at $z=i$ with radius $R$ . In hyperbolic polar coordinates $r$ and $\theta$, this means $r=R$ and $0\leq \theta\leq 2\pi$ where $z=x+iy$ satisfies 
\[
x=\frac{\sinh r \sin{\theta}}{\cosh r + \sinh r \cos{\theta}} \quad \mathrm{ and}\quad y=\frac{1}{\cosh r + \sinh r \cos{\theta}}.
\] 
The hyperbolic Laplacian then becomes $\Delta = -(\frac{\partial^{2}}{\partial r^{2}} + \coth (r) \frac{\partial}{\partial r} + \frac{1}{\sinh ^{2}(r)}\frac{\partial^{2}}{\partial \theta^{2}})$. Let $\psi=\psi(r,\theta)$ be a real valued bounded eigenfunction with eigenvalue $\lambda_{\psi}>0$.

Consider any real smooth functions $U(r,\theta)$ and $V(r,\theta)$, periodic in $\theta$ with period $2\pi$ and satisfying additionally that $U(0,\theta)=0$ for any $\theta$. Then, we have the trivial
\begin{equation}\label{HC1}
\int\limits_{\mathcal{C}_{R}} U \mathrm d\theta = \iint\limits_{B_{R}} \left(\frac{\partial  U}{\partial r} - \frac{\partial V}{\partial \theta}\right) \mathrm dr\mathrm d\theta \ ,
\end{equation}
where $B_{R} = \{(r,\theta) : 0\leq r\leq R \ \mathrm{ and} \ 0\leq\theta\leq 2\pi\}$ is the hyperbolic disc determined by $\mathcal{C}_{R}$. We apply this formula with the choice $U(r,\theta)= \lambda_{\psi}\sinh^{2}(r) \psi^{2} + \sinh^{2}(r) (\partial_{r}\psi)^{2} - (\partial_{\theta}\psi)^{2}$ and $V(r,\theta) = -2 (\partial_{r}\psi)(\partial_{\theta}\psi)$. One checks that $U(0,\theta)=0$ as follows: $\psi$ has a Fourier expansion of the form
\[
\psi(r,\theta) = \sum_{n\in\mathbb{Z}} a_{\psi}(n)W_{|n|,t_{\psi}}(\cosh r)e^{in\theta},
\]
where the functions $W$ are the real valued  Mehler functions   $W_{|n|,t_{\psi}}(s) = P^{-|n|}_{-\frac{1}{2}+it_{\psi}}(s)$. At $s=1$, they vanish for all $n\neq 0$, and is one when $n=0$, so that $\partial_{\theta}\psi(0,\theta)=0$. Substituting $U$ and $V$ in \eqref{HC1} and using the Laplacian gives us
\begin{equation}
\begin{split}
 \sinh^{2}(R)\left(\lambda_{\psi}\int\limits_{\mathcal{C}_{R}} \psi^{2}\mathrm d \theta   + \int\limits_{\mathcal{C}_{R}} (\partial_{r}\psi) ^{2}\mathrm d \theta\right)   - \int\limits_{\mathcal{C}_{R}}(\partial_{\theta}\psi)^{2}\mathrm d \theta \quad \quad\\ = 2\lambda_{\psi}\iint\limits_{B_{R}} \cosh(r) \psi ^{2} \mathrm d \mu,
 \end{split}
\end{equation}
where we have used $\mathrm d \mu = \sinh(r)\mathrm dr\mathrm d\theta$ as the invariant area measure. Since the invariant arclength measure on the circle is $\mathrm ds = \sinh(R)\mathrm d\theta$, this gives us 

\begin{proposition} Let $\mathcal{C}_{R}$ denote a hyperbolic circle with fixed positive radius $R$.  Then,
\[
\int\limits_{\mathcal{C}_{R}} \psi^{2}\mathrm d s   + \frac{1}{\lambda_{\psi}}\int\limits_{\mathcal{C}_{R}} (\partial_{n}\psi) ^{2}\mathrm d s \geq \frac{2}{\sinh{R}}\iint\limits_{B_{R}} \cosh(r) \psi^{2}\ \mathrm d \mu,
\]
where $\partial_{n}\psi$ is the normal derivative.
\end{proposition} 

\subsection{Closed horocycles}\quad

The proof here is similar to the one above. For a positive number $A>0$  we consider the cylinder  $\mathbb{Y}_{A} = \{z=x+iy\ : \ y\geq A\ \mathrm{ and}\  -\frac{1}{2}<x\leq \frac{1}{2}\}\subset \mathbb{H}$, where we identify the two vertical sides and include  a cusp at infinity. Let $\Delta = -y^{2}(\frac{\partial^{2}}{\partial x^{2}} + \frac{\partial^{2}}{\partial y^{2}})$ be the Laplacian and let $\psi(z)$ is a real-valued eigefunction on $\mathbb{Y}_{A}$ with the eigenvalue $\lambda_{\psi} = \frac{1}{4} + t_{\psi}^{2}$, periodic in $x$ and vanishing at the cusp.  We assume that $\psi$ is $L^{2}$ in $\mathbb{Y}_{A}$. For $A\leq Y <Y_{0}$ we consider the image of the rectangle $\mathcal{R}_{Y,Y_{0}} = \{z=x+iy\ ;\ Y\leq y\leq Y_{0}\ \mathrm{ and}\  -\frac{1}{2}<x\leq \frac{1}{2}\}$ in $\mathbb{Y}_{A}$   (we will denote $\mathcal{R}_{Y,\infty}$ by $\mathcal{R}_{Y}$).

We apply Green's theorem on $\mathbb{R}^{2}$ to the oriented rectangle $\mathcal{R}=\mathcal{R}_{Y,Y_{0}}$ in the form $\int_{\partial\mathcal{R}} P\mathrm dx +Q\mathrm dy = \iint_{\mathcal{R}}\left(\frac{\partial Q}{\partial x}-\frac{\partial P}{\partial y}\right)\ \mathrm dx\mathrm dy $. We choose $P= \frac{\lambda_{\psi}}{y^{2}}\psi^{2} + (\partial_{y}\psi)^{2} - (\partial_{x}\psi)^{2}$ and $Q= -2(\partial_{x}\psi)(\partial_{y}\psi)$ so that $\left(\frac{\partial Q}{\partial x}-\frac{\partial P}{\partial y}\right)= \frac{2\lambda_{\psi}}{y^{3}}\psi^{2}$. Since $\psi$ is periodic in $x$, the contribution from the vertical integrals cancel. For $A\leq y<\infty$, denote by $\mathcal{C}_{y}$ the closed loop on $\mathbb{Y}_{A}$ given by the points $z = x+iy$, with $-\frac{1}{2} < x\leq \frac{1}{2}$. Then we have
\[
\begin{split}
\int\limits_{\mathcal{C}_{Y}}\left(\frac{\lambda_{\psi}}{y^{2}}\psi^{2} + (\partial_{y}\psi)^{2} - (\partial_{x}\psi)^{2}\right) \mathrm dx = \int\limits_{\mathcal{C}_{Y_{0}}}\left(\frac{\lambda_{\psi}}{y^{2}}\psi^{2} + (\partial_{y}\psi)^{2}- (\partial_{x}\psi)^{2}\right)\mathrm dx  \\ + 2\lambda_{\psi}\iint\limits_{\mathcal{R}_{Y,Y_{0}}} \psi^{2} \frac{1}{y^{3}} \mathrm dx\mathrm dy \ .
\end{split}
\]
We now let $Y_{0} \to \infty$ so that the integral over $\mathcal{C}_{o}$ vanishes. This is due to the fact that $\psi$ has a Fourier expansion of the form
\[
\psi(z) = \sqrt{y}\sum_{n \neq 0} a_{\psi}(n) \mathrm{ K}_{it_{\psi}}(2\pi |n|y)e^{2\pi inx},
\]
where  $\mathrm{ K}_{s}(y)$  is the MacDonald-Bessel function having exponential decay in $y$ and the Fourier coefficients have at most polynomial growth in $n$. Using the invariant arclength measure $\mathrm ds = Y^{-1}\mathrm dx$ on $\mathcal{C}_{Y}$, we have

\begin{proposition}\label{CH20} For  fixed $Y \geq A$ and $\psi$ as above,
\begin{equation}\label{new1}
\int\limits_{\mathcal{C}_{Y}}\psi^{2}  \mathrm ds  + \frac{1}{\lambda_{\psi}}\int\limits_{\mathcal{C}_{Y}}(y\partial_{n}\psi)^{2}  \mathrm ds  \geq 2Y\iint_{\mathcal{R}_{Y}} \psi^{2}y^{-1} \ \mathrm d\mu ,
\end{equation}
where $\partial_{n}\psi$ is the normal derivative.
\end{proposition}

The Proposition is sharp in the following sense: suppose there are smooth positive functions $\alpha(y)$ and $G(y)$ such that
\begin{equation}\label{CH10}
\int\limits_{\mathcal{C}_{Y}}\psi^{2}  \mathrm ds  +  \frac{1}{\lambda_{\psi}}\int\limits_{\mathcal{C}_{Y}}(y\partial_{n}\psi)^{2}  \mathrm ds  \geq 2\alpha(Y)\iint\limits_{\mathcal{R}_{Y}} \psi^{2}G(y) \ \mathrm d\mu\ ,
\end{equation}
for all $Y\geq A$ and eigenfunctions $\psi$.  We seek $\alpha$ and $G$ that maximize the right hand side and so may assume that $\alpha(Y)G(y) \geq Yy^{-1}$ for all $y\geq Y\geq A$. By substituting the Fourier series for $\psi$ in \eqref{CH10} , one sees that \eqref{CH10} holds with $\psi$ replaced by $\mathrm{ K}_{it_{\psi}}^{*}(ny)$  for each $n\geq 1$, where  $\mathrm{ K}_{it_{\psi}}^{*}(y)= \sqrt{y} \mathrm{ K}_{it_{\psi}}(y)$ satisfies the differential equation $z''(y) +a(y)z(y)=0$ with $a(y)= \frac{\lambda_{\psi}}{y^2} -1$.  Put $u(y)= a(y)\mathrm{ K}_{it_{\psi}}^{*}(y)^{2} + (\partial_{y}\mathrm{ K}_{it_{\psi}}^{*}(y))^{2}$ where $\partial_{y}$ is the derivative with respect to $y$. Then one sees that $u'(y)=-2\frac{\lambda_{\psi}}{y^{3}}\mathrm{ K}_{it_{\psi}}^{*}(y)^{2}$ so that $u(y) =2\lambda_{\psi} \int_{y}^{\infty} \mathrm{ K}_{it}^{*}(y)^{2}y^{-3}\mathrm dy \ $. We then conclude from \eqref{CH10} with $n=1$ that 
\[
Y \mathrm{ K}_{it_{\psi}}^{*2}(Y) \geq  2\lambda_{\psi} \int\limits_{Y}^{\infty}\mathrm{ K}_{it_{\psi}}^{*2}(y)\left(\alpha(Y)G(y)-\frac{Y}{y}\right) \frac{\mathrm dy}{y^{2}} \geq 0.
\]
 Suppose  for some $\delta >0$ we have $(\alpha(u)G(v)-\frac{u }{v})> \delta $  for $(u,v)\in [u_{0},u_{1}]\times[u,v_{1}]$.  We choose $Y = u_{0}$  so that 
\[
Y \mathrm{ K}_{it_{\psi}}^{*2}(Y) \geq  2\delta\lambda_{\psi} \int\limits_{Y}^{v_{1}}\mathrm{ K}_{it_{\psi}}^{*2}(y) \frac{\mathrm dy}{y^{2}} \geq  2\delta\lambda_{\psi} \int\limits_{Y}^{Y +\delta}\mathrm{ K}_{it_{\psi}}^{*2}(y) \frac{\mathrm dy}{y^{2}},
\]
for $\delta$ small enough, independent of $\lambda_{\psi}$. Letting   $\lambda_{\psi}\to \infty$ and noting that the  asymptotics of the Bessel functions on both sides are of comparable size,  we derive a contradiction. Hence $\alpha(Y)G(y) = Yy^{-1}$ in \eqref{CH10}.

For the modular surface with the standard fundamental domain  truncated at $y=A=1$ we may apply Prop. \ref{CH20} to a Maass eigenform. The curves $\mathcal{C}_{Y}$ correspond to closed horocycles with $Y\geq 1$. The double integral in \eqref{new1} can then be bounded below by a constant using arithmetic QUE (compare with Theorem 1.1 of \cite{GRS1}).
\par

\subsection{Closed geodesics}\quad \label{closed-geo-sect}

The situation here starts the same way as the preceding two cases. We consider a compact surface which we view as a cylinder, closed at one end and having a circular boundary (a closed geodesic) at the other. This can be obtained by the action of  $\Gamma$,  a cocompact subgroup of $\PSL(2,\mathbb{R})$, on $\mathbb{H}$. Let $\ell$ denote a closed geodesic on $\mathbb{X}=\Gamma \backslash \mathbb{H}$. We transform the group to a conjugate group so that the closed geodesic is the arc $\{ z=iy : 1\leq y\leq k\}$ where $k>1$ and $\gamma =\left(\begin{smallmatrix}\sqrt{k}&0\\ 0& \frac{1}{\sqrt k}\end{smallmatrix}\right) \in \Gamma$. Using geodesic polar coordinates $(\rho,\theta)$, we have $x=e^{\theta}\mathrm{ tanh}\ \rho$ and $y=e^{\theta}\mathrm{ sech}\ \rho$ so that the closed geodesic corresponds to $\rho = 0$ and $0\leq \theta \leq \kappa$, with the endpoints identified and with $\kappa = \log k$. Let $\phi(z) =\phi(\rho,\theta)$ be  an even $L^{2}(\mathbb{X})$-normalised Maass form with eigenvalue $\lambda_{\phi} = \frac{1}{4} + t_{\phi}^{2}$, satisfying  $\phi(\rho,\theta + \kappa)=\phi(\rho,\theta)$ and $\rho \geq 0$. Then, we consider the cylinder $\mathcal{R}_{\zeta} = \{(\rho,\theta) : 0\leq \rho \leq \zeta \ \mathrm{ and} \ 0\leq \theta \leq \kappa\}$ (giving us a flared cylinder on $\mathbb{X}$). Let $\ell_{\zeta}$ denote the closed curve on $\mathbb{X}$ given by $\rho=\zeta$, so that $\ell=  \ell_{0}$ is the closed geodesic.

The (even) eigenfunction $\phi$ has the Fourier expansion of the type 
\begin{equation}\label{z45}\phi(\rho,\theta)=\sum_{n\in \mathbb{Z}} \alpha_{n}c_{n}(\sinh \rho) e(\frac{n\theta}{\kappa})\ ,
\end{equation}
where $c_{n}(s)= e_{-\frac{1}{2}+it_{\phi}}^{i\mu_{n}}(s)$ is a conical function with  $\mu_{n}=\frac{2\pi |n|}{\kappa}$. These conical functions are real and even, satisfying $e_{-\frac{1}{2}+it_{\phi}}^{i\mu_{n}}(0)=1$ and $\frac{\mathrm d}{\mathrm{ d s}}e_{-\frac{1}{2}+it_{\phi}}^{i\mu_{n}}(0)=0$ (see \cite{Dun}).

We apply Green's theorem as in the previous cases : here we choose 
\[P(\rho,\theta)= \lambda_{\phi}\cosh^{2}(\rho) \phi^{2} + \cosh^{2}(\rho) (\partial_{\rho}\phi)^{2} - (\partial_{\theta}\phi)^{2}\  \mathrm{ and} \  Q(\rho,\theta) = -2 (\partial_{\rho}\phi)(\partial_{\theta}\phi).
\]
 If we put $G(\zeta) =  \cosh^{2}(\zeta)\left(\lambda_{\phi}\int_{\ell_{\zeta}} \phi^{2}\mathrm d \theta +\int_{\ell_{\zeta}} (\partial_{\rho}\phi) ^{2}\mathrm d \theta\right) -  \int_{\ell_{\zeta}} (\partial_{\theta}\phi)^{2}\mathrm d \theta$, we conclude that
\begin{equation}\label{CH1}
G(\zeta) = G(0) + 2\lambda_{\phi}\iint\limits_{\mathcal{R}_{\zeta}}\sinh(\rho)\phi^{2}\mathrm d\mu.
\end{equation}

To obtain the analog of  the propositions in the two cases considered above, it would be ideal if one could show $G(0) \geq 0$. However in this generality, that is not true . This can be seen quite easily since the Fourier expansion implies that 
\[
G(0) = \kappa\sum_{n}\alpha^{2}_{n}(\lambda_{\phi} - \mu_{n}^{2}),
\]
so that if $\phi$ is supported only for $n$ with $\mu_{n}>\sqrt{\lambda_{\phi}}$  we have $G(0)<0$. If one were to incooperate the term $\beta(\zeta):= \int_{\ell_{\zeta}} (\partial_{\theta}\phi)^{2}\mathrm d \theta$ from $G(\zeta)$ into $G(0)$, then using the asymptotics of the conical functions \cite{Dun}, (8.18) for $\mu_{n}>(1+\delta)\sqrt{\lambda_{\phi}}$, the corresponding terms in the sum for $G(0) + \beta(\zeta)$ can be shown to be positive. There remains a resonance range (see Remark \ref{resonance}) of $n$'s with $\sqrt{\lambda_{\phi}}<\mu_{n}\leq(1+\delta)\sqrt{\lambda_{\phi}}$ for which we are unable to show positivity due to the fluctuations of the conical functions. However it is possible to obtain an inequality with an error term if one assumes some strong conditions on the size of the Fourier coefficients.

The question remains if one can show for a closed geodesic $\ell_{0}$ an analog of \eqref{CH10}  of the following form: does there exist   a function $G(\rho)\geq 0$ but not always zero, such that

\[
\int\limits_{\ell_{0}} \phi^{2}\mathrm d \theta +\frac{1}{\lambda_{\phi}}\int\limits_{\ell_{0}} (\partial_{\rho}\phi) ^{2}\mathrm d \theta \geq 2\iint\limits_{\mathcal{R}_{\zeta}}G(\rho)\phi^{2}\mathrm d\mu,
\]
for some $\zeta >0$ ?   We find that this too is not possible. Using the Fourier series above and the properties of the conical functions on the closed geodesic, it follows that there is a $\delta >0$ small enough, and a $\zeta >0$ such that for all $n\in \mathbb{Z}$
\[\kappa \geq 2\delta\iint\limits_{\mathcal{R}_{_{\zeta, \zeta+\delta}}}\left (e_{-\frac{1}{2}+it_{\phi}}^{i\mu_{n}}(\rho)\right)^{2}\sinh(\rho)\mathrm d\rho. 
\]

We now use the asymptotics of the conical function off the closed geodesic. By \cite{Dun}, (8.18) we have that for $n$ in a suitable range (see above) relative to $\lambda_{\phi} \to \infty$ and $\zeta$, one can make the conical function grow exponentially in the parameter $t_{\phi}$ so that we have a contradiction, giving us $G(\rho)=0$.

\vspace{20pt}
\section{Uniform lower bound for the norm of a geodesic restriction}
Our aim here to prove Theorem \ref{closed-g-intro}. We show that it is a consequence of the QUE property of eigenfunctions. We use representation theory of the group $\PGLR$ to carry out the necessary analysis. There is no arithmetic input in this section (beyond the fact that QUE is known only in the arithmetic situation). 
\subsection{Restriction to closed geodesics}\label{closed-geodesic-sect}\ 

 Let $\bX$ be a compact oriented Riemann surface endowed with a hyperbolic Riemann metric $g_\bX$, the corresponding (positive) Laplace-Beltrami operator $\Dl$ and the volume element $d_\bX vol$ as in Section \ref{geom-pre}. Let  $\ell\subset \bX$ be a closed geodesic. We are interested in a {\it lower} bound for the $L^2$-norm of the restriction to $\ell$ for  eigenfunctions  with large eigenvalue. We denote by $\psi_\tau$ an $L^2$-normalized (i.e., by $\|\psi_\tau\|_{L^2(\bX)}=1$) eigenfunction  of $\Dl$ with the eigenvalue $\lm=(1-\tau^2)/4$, $\tau\in i\br\cup (0,1]$ (slightly deviating from our previous notation $\lambda_{\psi} = \frac{1}{4} + t_{\psi}^{2}$).

Our results are applicable to sequences of Maass forms satisfying a certain equidistribution property called the quantum unique ergodicity (QUE) property.  In simple words, QUE for a sequence $\{\psi_{\tau_i}\}_{i=1}^{\8}$ of eigenfunctions means that the probability measures $|\psi_{\tau_i}|^2d_\bX vol$  became equidistributed on $\bX$ when $\lm_i\to\8$. In practice, one needs to extend the equidistribution property to microlocalization of eigenfunctions. This is done in representation-theoretic language by Zelditch in \cite{Z1}, \cite{Z2} (see Section \ref{QUE-sect} for the exact formulation we use). According to the  theorem of Shnirel'man (see \cite{Sh}, \cite{CdV}, \cite{Z2}) any orthonormal basis of Maass forms contains a full density subsequence satisfying the QUE property. In fact,  it was conjectured in \cite{RS} that QUE  holds for a complete orthonormal basis  for a general compact hyperbolic surface. QUE  is known to hold for arithmetic surfaces coming from congruence subgroups and for the special (arithmetic) basis of Hecke-Maass forms thanks to the striking work of Lindenstrauss \cite{Li} (see also \cite{BrooksLi} for results involving just one Hecke operator).

\begin{theorem}\label{thm-geodesic-low-bound} Let $\{\psi_{\tau_i}\}_{i=1}^{\8}$ be a sequence of Maass forms satisfying the QUE property. There exists a constant $a>0$ such that the following bound holds:
\begin{equation}\label{phi-lower-bound}
\int\limits_\ell\left(|\psi_{\tau_i}|^2+\lm_i\inv\cdot|\partial_n\psi_{\tau_i}|^2\right)\mathrm d\ell\geq a\
\end{equation}  for all $i$. Here $\partial_n\psi_\tau$ denotes the normal derivative along the closed geodesic $\ell\subset \bX$.
\end{theorem}

\noindent{\it Remark.} The constant $a$ in the theorem is non-effective as long as the rate of convergence in QUE is not known. 

Our proof uses in an essential way the full force of QUE on the co-tangent bundle of $\bX$. We approach the related analysis via representation theoretic language. 

We describe now the translation of the restriction problem into the well-known representation-theoretic language introduced by Gelfand and Fomin (see \cite{B}, \cite{GPS}, \cite{L} among many other sources).

\subsection{Representation theory}
\subsubsection{Group action}\label{gr-act-setup}\  Let $\bG=\PGLR$. We denote by $\bG^+$ the neutral  connected component of $\bG$ , and by $\bG^-$ the second connected component. In what follows it is crucial that we work with both components.  The reason is that we will use the crucial multiplicity one property (which fails for $\bG^+\simeq\PSLR$). We will identify the compact subgroup $K$ from Section \ref{geom-pre} with $K=\PSO(2,\br)\subset\bG$.  As before, there exists a lattice $\G\subset \bG^+$  such that the Riemann surface $\bX$ is given by the quotient $\bX\simeq\G\backslash \bG^+/K$.  We denote by $Z=\G\backslash \bG$ what we call the automorphic space and by $Z_\pm=\G\backslash \bG^\pm$ two of its connected components. The Riemann surface $\bX_-=Z_-/K$ is naturally identified with the orientation reversed copy of $\bX$. Let $R$ be the right action of $\bG$ on functions on $Z$ (i.e., $R(g)f(x)=f(xg)$). We extend the volume element on $\bX$ to the $\bG$ invariant measure $\mu$ on $Z$. The representation $R$ becomes a unitary representation of $\bG$ on $L^2(Z,\mu)$.

Let $A=\{\mathrm{ diag}(a,b)\}/\{\mathrm{ diag}(a,a)\}\subset \bG$ be the full (disconnected) diagonal subgroup, and denote by $A_\pm=A\cap \bG^\pm$ its connected components. Let $\dl=\left(
                                          \begin{smallmatrix}
-1  & \\
                                   &\ \ 1 \\
\end{smallmatrix}
                                        \right)\in \GLR$,
and we denote by the same letter the corresponding element in $\bG$ (in general, we will often construct elements in $\bG$ by their representatives in $\GLR$). We have $A=A_+\cdot\langle\dl\rangle$. We now consider closed $A$-orbits. To a (closed) geodesic $\ell\subset \bX$ there corresponds  a (closed) $A_+$-orbit  in $Z_+$, which we denote by the same letter $\ell$. If $\ell\subset Z_+$ is a closed $A_+$-orbit, then $l=\ell\cdot A=\ell\cup \ell\cdot\dl\subset Z$ is a closed $A$-orbit consisting of two connected components $l_\pm\subset Z_\pm$ interchanged by $\dl$.
We have then the {\it pointwise} stabilizer  subgroup $A_l=\mathrm{ Stab}_A  l=\mathrm{ Stab}_{A_+}  \ell=\langle a_l\rangle \simeq \bz$,
 where $\bar a_l=\mathrm{ diag}(e^q,e^{-q})\in A_+$, $q>0$, is an element which,  as in Section \ref{closed-geo-sect}, is conjugated to a primitive hyperbolic element   $\g_l\in \G$ corresponding to the closed geodesic $\ell$ on $\bX$ (i.e., $g_l\inv \g_l g_l=\bar a_l$ for an appropriate $g_l\in \bG^+$ and  the corresponding geodesic is given by $\ell=g_l A_+\cdot i\subset \bX\simeq\G\backslash\bG^+/K$).

We fix an $A$-invariant measure $\mathrm dl$ on $l$ (e.g., consistent with the Riemannian length of $\ell$).

 We consider characters of the disconnected group $A$. Any such unitary character is given by $\chi_{s,\epsilon}\left(\bar a\right)=|a|^s\cdot \mathrm{ sign}(a)^\epsilon$, where $\bar a=\mathrm{ diag}(a,|a|\inv)\in A$, $s\in i\br$, and $\epsilon\in\{0, \ 1\}$.  A character  is called  {\it even} if $\epsilon=0$  and {\it odd} if $\epsilon=1$ (note that  $\chi_{s,\epsilon}(\dl)=(-1)^\epsilon$).

 We next consider characters which are trivial on the subgroup $A_l$.
These are given by $\chi_{s_n,\epsilon}$ with $s_n=2\pi in/q$, $n\in \bz$, and $\epsilon\in\{0,1\}$.

We now construct $A$-equivariant functionals on the space of smooth functions on $Z$.  Let $l\subset Z$ be a closed $A$-orbit as above. We fix a point $\al_0\in l$ and associate the function $\chi^\cdot(\al_0\cdot \bar a)=\chi(\bar a)$ to any  character $\chi=\chi_{s_n,\epsilon}$ which is trivial on $A_l$.  This gives rise to an  $A$-equivariant functional $d^{\mathrm{aut}}_{\chi}:C^\8(Z)\to\bc$  given by the integration $d^{\mathrm{aut}}_\chi(f)=\int_{t\in l}f(t)\bar\chi^\cdot(t)\mathrm dl$. The functional $d^{\mathrm{aut}}_{\chi}$ is $\chi$-equivariant with respect to the right action of $A$: $d_\chi(R(\bar a)f)=\chi(\bar a)f$.

\subsubsection{Maass forms and representations}\  Let $\psi_\tau$ be a Maass form as above. By the principle of Gelfand and Fomin (\cite{GPS}), $\psi_\tau$ generates an irreducible unitary representation $V_{\psi_\tau}$ of the principal series under the action of $\PSLR$ on the space $L^2(Z_+,\mu)$ (in fact, we only consider the smooth part $V_{\psi_\tau}\subset C^\8(Z_+)$ of the corresponding representation). Such a representation is called an automorphic representation (of $\PSLR$). The function $\psi_\tau\in V_{\psi_\tau}$ corresponds  to (up to a multiple)  the unique (up to a scalar) $K$-invariant  function in $V_{\psi_\tau}$  under the natural imbedding  $C^\8(\bX)\subset C^\8(Z_+)$.

For reasons explained later (i.e., uniqueness of invariant functionals), we need to switch to representations of $\bG$. To do this, we extend the function $\psi_\tau\in C^\8(Z_+)$ to a $\dl$-even function  $\tilde\psi_\tau:Z\to\bc$ by
$\tilde\psi_\tau(z\dl)=\psi_\tau(z)$ for $z\in Z_+$. It is then easy to see that the (smooth) representation $V_{\tilde\psi_\tau}\subset C^\8(Z)$ of $\bG$ generated by $\tilde\psi_\tau$ is  an irreducible representation of $\bG$. In fact, one can easily see that the space generated by the function $\psi_\tau$ extended by
{\it zero} to $Z_-$ is the sum of two irreducible representations of $\bG$. The second irreducible component is generated by  $\psi_\tau$ extended  to a $\dl$-odd function on $Z$. Note that the restriction of functions on $Z$ to functions on $Z_+$ maps $V_{\tilde\psi_\tau}$ to $V_{\psi_\tau}$. While  this is not the restriction of a representation of $\bG$ to a representation of $\bG^+$, it commutes with the action of $\bG^+$. 

\begin{remark}We also would like to point out that the action of $\dl$ on functions on $Z$ should not be confused with a symmetry $\s$ of the Riemann surface $\bX$ stabilizing a geodesic which we crucially exploit in Section \ref{geom-pre}. The element $\dl$ is a part of a general setup we use and comes from the action of $\bG$ on connected components of $Z$. In particular, this is present for all Riemann surfaces and not only for those with $\s$-symmetry. This is the reason our restriction Theorem \ref{thm-geodesic-low-bound} holds for a general surface and a closed geodesic. The (orientation reversing) symmetry $\s$ appears once we have identification of surfaces $\bX$ and $\bX_-$\ . \end{remark}

 All unitary representations of $\bG$ are well-known, and in particular could be modeled in spaces of homogenous functions on $\br^2\backslash\{0\}$ (or in spaces of functions on $S^1$, $\br$, etc.). We will always work with the spaces of smooth vectors of such representations. In fact, we consider principal series representations of $\GLR$ which are  trivial on the center. These are parametrized by two parameters  $\tau\in i\br$ and $\eps\in\{0, 1\}$.  The corresponding space $V_{\tau,\eps}$ could be realized as the space of even smooth homogeneous functions of the homogeneous degree $\tau-1$ on the plane $\br^2\backslash\{0\}$, that is  $f(\al\cdot x)=|\al|^{\tau-1}f(x)$ for $x\in \br^2\backslash\{0\}$ and $t\in \br^\times$. The action of $\GLR$ which is trivial on the center is given by the natural action on $\br^2$, namely $\pi_{\tau,\eps}(g)f(x)=f(g\inv\cdot x)|\det(g)|^\frac{\tau-1}{2}\mathrm{ sign}(\det(g))^\eps$ for $g\in \GLR$.  Hence the evenness condition is required since $-1\in\GLR$ should act as the identity. The invariant norm on such a representation is given by the (normalized by $1/2\pi$) integration over $S^1\subset \br^2$.

By restricting homogeneous functions to $S^1 \subset \br^2$ we obtain the circle model of the representation $V_{\tau,\eps}\simeq C^\8_{\mathrm{even}}(S^1)$, and by restricting to the line $\{(x,1)|\ x\in \br\}\subset \br^2$ we obtain the line model.  Smooth vectors in the line model are smooth functions on $\br$ with a certain polynomial decay at infinity.  These models are convenient for various computations.

We can view the automorphic representation $V_{\tilde\psi_\tau}\subset C^\8(Z)$ as a realization of the model representation $V_{\tau,\eps}$. Namely, we will denote by $\nu_{\tilde\psi_\tau}:V_{\tau,\eps}\to V_{\tilde\psi_\tau}$ the corresponding $\bG$-map which we will assume to be an isometry. We will denote by $\psi_v=\nu_{\tilde\psi_\tau}(v)\in C^\8(Z)$ the image of a vector $v\in V_{\tau,\eps}$ under the map $\nu_{\tilde\psi_\tau}$.

We choose a (unique up to a multiple) $K$-invariant vector in the space $V_{\tau,\eps}$ to be the constant function $e_0\equiv 1$ on  $S^1$ (extended by the homogeneity to $\br^2\backslash\{0\}$). The vector $e_0$ corresponds to the Maass form $\tilde\psi_\tau$ under the isometry $\nu_{\tilde\psi_\tau}: V_{\tau,\eps}\to V_{\tilde\psi_\tau}\subset C^\8(Z)$, i.e., $\nu_{\tilde\psi_\tau}(e_0)=\tilde\psi_\tau$ on $Z$.  In particular, $\dl$ acts on $e_0$, and by the uniqueness we have $\pi(\dl)e_0=(-1)^\eps e_0$. For the automorphic realization, this translates into  the identity $\tilde\psi_\tau(z\dl)=(-1)^\eps \tilde\psi_\tau(z)$, i.e., (the extension of) the Maass form is assumed to be either $\dl$-even or $\dl$-odd on $Z$. Hence we can assume that our Maass form has even extension. Later on we will have to deal with odd representations as well.

\subsubsection{Closed geodesic expansion}\label{geod-exp-sect} We now construct certain $\chi$-equivariant functionals on a representation of the principal series of $\bG$, and compute these on some special vectors. These functionals give the representation theoretic construction of the special functions $e_{-\frac{1}{2}+it_{\psi}}^{i\mu_{n}}$ in \eqref{z45}. Let $V_\tau=V_{\tau,\eps}$ be such a representation (while we assumed that $\eps=0$, we still will treat the general case). Let $\chi_{s,\epsilon}$ be a character of $A$ as before. To any such character, we associate the corresponding equivariant functional $d_{s,\epsilon;\tau,\eps}$ on the representation $V_{\tau}$ (i.e., an element in $\Hom_A(V_{\tau},\chi_{s,\epsilon})$) given in the line model on $\br$ by
\begin{equation}\label{d-chi-line}
d_{s,\epsilon;\tau,\eps}(v)=\pi\inv\int\limits_{ \br}v(x)|x|^\frac{-1-\tau+s}{2}\mathrm{ sign}(x)^{\epsilon+\eps} \mathrm dx\  ,
\end{equation} for any smooth vector $v\in V_\tau\subset C^\8(\br)$.  The integral is absolutely convergent on functions in $V_\tau$ (since the function $|v(x)|$ decays as $|x|\inv$ at infinity, as required by the smoothness of vectors in $V_\tau$). In the circle model, the same functional is given by  
\begin{equation}\label{d-chi-circle}
d_{s,\epsilon;\tau,\eps}(v)=(2\pi)\inv\int\limits_{ S^1}v(\theta)|\cos(\theta)|^\frac{-1-\tau+s}{2}|\sin(\theta)|^\frac{-1-\tau-s}{2}\mathrm{ sign}(\theta)^{\epsilon+\eps} \mathrm d\theta\  ,
\end{equation} for any smooth vector $v\in V_\tau\simeq C^\8_{\mathrm{even}}(S^1)$. Here $\mathrm{ sign}(\theta)$ denotes the {\it even} function on $S^1$ taking values $1$ on $[0,\pi/2)$ and $-1$ on $[\pi/2,\pi)$ (i.e., it is odd with respect to the action of $\dl$).

We can easily evaluate the value of $d_{s,\epsilon,\tau}(e_0)$ since this is one of the classical integrals (see \cite{Ma}). We obtain
\begin{eqnarray}\label{d-chi-e-0}
d_{s,\epsilon;\tau,\eps}(e_0)&=&\dl_{\epsilon\eps}\cdot\pi\inv\int\limits_{ \br}(1+x^2)^\frac{\tau-1}{2}|x|^\frac{-1-\tau+s}{2} \mathrm dx\\ &=&\dl_{\epsilon\eps} 2\pi\inv\cdot \frac{\G\left(\frac{1-\tau+s}{4}\right) \G\left(\frac{1-\tau-s}{4}\right)}{\G\left(\frac{1-\tau}{2}\right) }\  ,\nonumber
\end{eqnarray} i.e., $d_{s,\epsilon;\tau,\eps}(e_0)=0$ if $\epsilon+\eps\not\equiv 0 \ (\mathrm{ mod}\ 2)$, and in fact is non-zero otherwise. In particular, from the classical asymptotic for the  $\G$-function, we see that for any {\it fixed} constant $0<\omega<1$,   the bound
\begin{equation}\label{d-chi-e-0-lower-bound}
|d_{s,\epsilon;\tau,\eps}(e_0)|^2\geq c\dl_{\epsilon\eps}(1+|\tau|)\inv
\end{equation} holds for  $|s|\leq \omega |\tau|$, $s\in i\br$, with some explicit constant $c=c_\omega>0$ depending on $\omega$ (we assume that $|\tau|\geq 1$).

\begin{remark}\label{resonance} Note that in the resonance range when $|s|-|\tau|$ is small  compared to $|\tau|$, the value of $|d_{s,\epsilon;\tau,\eps}(e_0)|^2$ jumps. This poses a serious problem for an {\it upper} bound for a norm of a geodesic restriction (see \cite{R}). For a {\it lower} bound, however, this complication could be treated with the help of representation theory as we show below. 
\end{remark}

We  evaluate model functionals on one more vector. Consider $e_0'(x)=\frac{\mathrm d}{\mathrm dx}e_0(x)$. We have  $e'_0=\pi(\mathfrak{n})e_0$ for the standard (w.r.t. $A$) nilpotent element  $\mathfrak{n}=\left(
                                          \begin{smallmatrix}
0  &1 \\
       0                            &0 \\
\end{smallmatrix}
                                        \right)$ in the Lie algebra of $\bG$. 	We have (see \cite{Ma})
\begin{eqnarray}\label{d-chi-e-0-prime}\nonumber
d_{s,\epsilon;\tau,\eps}(e'_0)&=&(1-\dl_{\epsilon\eps})\cdot\pi\inv(\tau-1)\int\limits_{ \br}x(1+x^2)^\frac{\tau-3}{2}\mathrm{ sign}(x)|x|^\frac{-1-\tau+s}{2} \mathrm dx \\ &=&(1-\dl_{\epsilon\eps})2\pi\inv(\tau-1)\cdot \frac{\G\left(\frac{3-\tau+s}{4}\right) \G\left(\frac{3-\tau-s}{4}\right)}{\G\left(\frac{3-\tau}{2}\right) }\  .
\end{eqnarray} Hence we have the lower bound
\begin{equation}\label{d-chi-e-0-lower-bound-prime}
|d_{s,\epsilon;\tau,\eps}(e_0')|^2\geq c(1-\dl_{\epsilon\eps})(1+|\tau|)\ ,
\end{equation}  for  $|s|\leq \omega |\tau|$, $s\in i\br$, with some explicit constant $c=c_\omega>0$ depending on $0<\omega<1$.

Note that vectors $e_0$ and $e_0'$ have complementary parity with respect to $\dl$.

Geometrically, the restriction of the automorphic function $\psi_{\pi(\mathfrak{n})e_0}\in C^\8(Z)$ to a point on the orbit $l=g_l A\subset Z$ gives  the normal derivative of the Maass form $\psi=\psi_{e_0}$ along the  geodesic on $\bX$.
In fact, we have
\begin{eqnarray*}
\psi_{\pi(\mathfrak{n})e_0}\left(g_l\left(\begin{smallmatrix}
y^\haf & \\
                         &    y^{-\haf}      \\
\end{smallmatrix} \right)\right)
=\frac{\mathrm d}{\mathrm dx}\psi_{e_0}\left(g_l
\left(\begin{smallmatrix}
y^\haf  &x \\
                               &  y^{-\haf}  \\
\end{smallmatrix}
                                        \right)\right)_{x=0} \ .
\end{eqnarray*}This  is equal to the normal derivative $$\partial_n\psi(z)= \frac{\mathrm d}{\mathrm dx}\psi(g_l(yi+yx))_{x=0}=y\frac{\mathrm d}{\mathrm dt}\psi(g_l(yi+t))_{t=0}$$   evaluated at the point   $z=g_l(yi)\in \ell$ which we view as a point on the upper half plane.

\begin{remark}\label{spec-funct-matrix-coeff} Note that the function $d_{s,\epsilon;\tau,\eps}(\pi_{\tau,\eps}(g)e_0)$ (which is nothing else than a matrix coefficient in the representation $\pi_{\tau,\eps}$) as a function on $\bG$ is left $A$-equivariant and right $K$-invariant (and hence descends to the upper plane $\mathbb{H}$).  This implies that it is proportional to the function  $e_{-\frac{1}{2}+it_{\psi}}^{i\mu_{n}}$ from Section \ref{closed-geo-sect} since both of them are eigenfunctions of the Laplacian on $\mathbb{H}$ with the same eigenvalue. Classically, special functions are normalized by a value at a point (e.g., by $e_{-\frac{1}{2}+it_{\psi}}^{i\mu_{n}}(0)=1$ as in Section \ref{closed-geo-sect}) or by an integral representation (e.g., our choice in \eqref{d-chi-e-0}). Of course one have similarly defined odd special functions  related to $d_{s,\epsilon;\tau,\eps}(\pi_{\tau,\eps}(g)e'_0)$. 

\end{remark}

The central fact from representation theory that we will use is the multiplicity one theorem for  certain equivariant functionals. 
Namely, that for any irreducible representation $V$ of $\bG$ and for any character $\chi$ of $A$, the vector space $\Hom_A (V, \chi)$ is at most one-dimensional (see, for example, \cite{B} for this standard fact). Here the fact that we work with the full diagonal subgroup is important, otherwise the space of $A_+$-equivariant functionals is two-dimensional (naturally splitting into $\dl$-even and $\dl$-odd subspaces). For any character $\chi=\chi_{s,\epsilon}$ of $A$  which is trivial on $\mathrm{ Stab}_A l$, we encountered two equivariant functionals on every automorphic representation: the automorphic functional  $d^{\mathrm{aut}}_{\tilde\psi_\tau, \chi}=d^{\mathrm{aut}}_\chi\big|_{V_{\tilde\psi_\tau}}$ given by the integral along the $A$-orbit, and the model functional $d_{s_n,\epsilon;\tau,\eps}$. The model functional is in fact defined for any $(\tau,\eps)$ and $(s,\epsilon)$. From the multiplicity one theorem, we deduce that there exists a constant $a_{\chi_{s_n,\epsilon}}(\tilde\psi_\tau)\in \bc$ such that  \begin{equation}\label{a-n-coeff-deff}
d^{\mathrm{aut}}_{\tilde\psi_\tau, \chi}(\nu_{\tilde\psi_\tau}(v))=a_{\chi_{s_n,\epsilon}}(\tilde\psi_\tau)\cdot d_{s_n,\epsilon;\tau,\eps}(v) ,
\end{equation} for any vector $v\in V_\tau$. The power of this relation is that the constant $a_{\chi_{s_n,\epsilon}}(\tilde\psi_\tau)$ does not depend on the vector $v$ in the same automorphic representation.  The coefficients $a_{\chi_{s_n,\epsilon}}(\tilde\psi_\tau)$ are closely related to  the  coefficients $\al_n$ appearing classically in the expansion \eqref{z45}. Namely, it follows from Remark \ref{spec-funct-matrix-coeff} that 
\begin{equation}\label{a-n-al-n}
\al_n c_n(0)=a_{\chi_{s_n,\epsilon}}(\tilde\psi_\tau)\cdot d_{s_n,\epsilon;\tau,\eps}(e_0)\ ,
\end{equation} and hence these coefficients are essentially the same (i.e., only depend on the choice of normalization of the corresponding special functions). Under the choice we made for $d_{s_n,\epsilon;\tau,\eps}$, the coefficients $a_{\chi_{s_n,\epsilon}}(\tilde\psi_\tau)$ satisfy the bound
\begin{equation}\label{a999} \sum_{|s_n|\leq T}|a_{\chi_{s_n,\epsilon}}(\tilde\psi_\tau)|^2\ll_{\bX,l}\max (|\tau|, T)\ ,   \end{equation} for any $T\geq 1$ (see \cite{R}, \cite{R5}).

The Fourier series expansion on $l\simeq S^1$ implies, via the orthogonality of functions  $\chi_{s_n,\epsilon}^\cdot$ in $L^2(l,\mathrm dl)$, the following expansion for the norm of the restriction to $l$:
\begin{equation}\label{expansion-on-l}
\|\nu_{\tilde\psi_\tau}(v)\|^2_{L^2(l)}=\mathrm{ length}(l)\cdot \sum_{\chi_{s_n,\eps}\in \widehat{A/A_l}}|a_{\chi_{s_n,\eps}}(\tilde\psi_\tau)|^2 |d_{s,\epsilon;\tau,\eps}(v)|^2\ ,
\end{equation} for any vector $v\in V_{\tau,\eps}$.

Our aim is to prove the following lower bound for  coefficients $a_{\chi_{s_n,\eps}}(\tilde\psi_\tau)$.
\begin{theorem}\label{thm-a-lower} Assume that QUE holds for a sequence $\{\psi_{\tau_i}\}$ of Maass forms on $\bX$ (e.g., in in the form \eqref{QUE}). There exists an explicit constant $\omega\in(0,1)$ depending on $\bX$ and $\ell$, but not on $\psi_{\tau_i}$, such that the following bound holds:
\begin{equation}\label{a-n-lower-bound}
\sum_{|s_n|\leq \omega|\tau_i|,\ \epsilon}|a_{\chi_{s_n,\eps}}(\tilde\psi_{\tau_i})|^2\geq c' |\tau_i|\ ,
\end{equation} for some constant $c'>0$ and all $i$.
\end{theorem}
We prove this bound by choosing an appropriate test vector $v$ in the expansion \eqref{expansion-on-l}, and estimating the left hand side  from below by appealing to QUE.

\subsubsection{Proof of Theorem \ref{thm-geodesic-low-bound}} We have
\begin{equation*}
\int\limits_l\left(|\psi_\tau|^2+\lm\inv\cdot|\partial_n\psi_\tau|^2\right)\mathrm dl =\|\nu_{\tilde\psi_\tau}(e_0)\|^2_{L^2(l)} + \lm\inv \|\nu_{\tilde\psi_\tau}(e'_0)\|^2_{L^2(l)} .
\end{equation*}

Using the expansion \eqref{expansion-on-l} and lower bounds \eqref{d-chi-e-0-lower-bound} and \eqref{d-chi-e-0-lower-bound-prime} for values of $d_{s_n,\epsilon;\tau,\eps}(e_0)$ and of $d_{s_n,\epsilon;\tau,\eps}(e'_0)$, we deduce the bound \eqref{phi-lower-bound} in Theorem \ref{thm-geodesic-low-bound} from the lower bound \eqref{a-n-lower-bound}.\qed

\subsection{Fattening the geodesic}

\subsubsection{Bound on geodesic coefficients}\label{proof-lower-bound-cl-geodesic} The idea of the proof for the bound \eqref{a-n-lower-bound} is that for a given $\omega\in(0,1)$, we need to construct a test vector $v=v_\omega\in V_\tau=V_{\tau,\eps}$  (i.e., a function $v\in C^\8_{\mathrm{even}}(S^1)$ of norm one), and a small {\it fixed} neighborhood $U=U_\omega\subset \bG^+$ (to be specified later)  such that for any $g\in U$ and some explicit constant $\al>0$, the following spectral bounds hold:
\begin{equation}\label{v-regular-range}
|d_{s,\epsilon;\tau,\eps}(\pi_{\tau,\eps}(g)v)|^2\leq\al|\tau|\inv\ ,
\end{equation} for any  $|s|\leq \omega|\tau|$ and any $\epsilon$, $\eps$,
and
\begin{equation}\label{v-cutoff}
|d_{s,\epsilon;\tau,\eps}(\pi_{\tau,\eps}(g)v)|^2\leq\al|s|^{-3}\ ,
\end{equation} for any  $|s|\geq \omega|\tau|$ and any $\epsilon$, $\eps$ (in fact, for the vector we construct, the bound \eqref{v-regular-range} is sharp, possibly in a shorter range of $s$). We will then show that for  the  constructed vector $v$,  QUE implies the lower bound
\begin{equation}\label{v-lower-bound-on-Z}
\int\limits_{x\in l\cdot U}|\psi_v(x)|^2f(x)\mathrm d\mu \geq b\ ,
\end{equation} for some  constant  $b>0$ which is independent of $\psi_\tau$. Here $f$ is a non-negative function on $Z$ constructed in the following way: we first choose a smooth non-negative function $f_U\in C^\8(\bG)$ with $\mathrm{ supp}(f_U)\subset U$ and satisfying $\int_\bG f_U\mathrm dg=1$, and then consider $f=\dl_l\ast f_U$ (i.e., the average of the delta function on $l$ under the right action by $f_U$ -- this being the ``fattening'' of the geodesic).

Assuming  bounds \eqref{v-regular-range}, \eqref{v-cutoff} and \eqref{v-lower-bound-on-Z}, we have
\begin{eqnarray*}b&\leq&\int\limits_{x\in l\cdot U}|\psi_v(x)|^2f(x)\ \mathrm d\mu\\ &=&\int\limits_{g\in U}\int\limits_l|\psi_v(lg)|^2f_U(g)\ \mathrm dl\mathrm dg=\int\limits_{g\in U}\int\limits_{t\in l}|\psi_{\pi_{\tau,\eps}(g)v}(t)|^2f_U(g)\ \mathrm dt\mathrm dg\\
&=& \int\limits_{g\in U}\left[\sum_{\chi_{s_n,\epsilon}\in \widehat{A/A_l}}|a_{\chi_{s_n,\epsilon}}(\tilde\psi_\tau)|^2|
d_{s_n,\epsilon;\tau,\eps}(\pi_{\tau,\eps}(g)v)|^2\right]f_U(g)\ \mathrm dg\\
&=&\sum_{|s_n|\leq \omega|\tau|}|a_{\chi_{s_n,\epsilon}}(\tilde\psi_\tau)|^2\cdot
\int\limits_U|d_{s_n,\epsilon;\tau,\eps}(\pi_{\tau,\eps}(g)v)|^2f_U(g)\ \mathrm dg\\
&&+\sum_{|s_n|\geq \omega|\tau|}|a_{\chi_{s_n,\epsilon}}(\tilde\psi_\tau)|^2\cdot
\int\limits_U|d_{s_n,\epsilon;\tau,\eps}(\pi_{\tau,\eps}(g)v)|^2f_U(g)\ \mathrm dg\\
&\leq&
\al|\tau|\inv\sum_{|s_n|\leq \omega|\tau|}|a_{\chi_{s_n,\epsilon}}(\tilde\psi_\tau)|^2 +\al\sum_{|s_n|\geq \omega|\tau|}|s_n|^{-3}|a_{\chi_{s_n,\epsilon}}(\tilde\psi_\tau)|^2 \ .
\end{eqnarray*}
The second sum is easily bounded above  by $|\tau|^{-2}$ using the geometric bound \eqref{a999}; in fact, this bound is also proved by the ``fattening'' argument (see \cite{BR1}, \cite{R}, \cite{R5}). Hence we deduce  the bound \eqref{a-n-lower-bound} from the bounds \eqref{v-regular-range}, \eqref{v-cutoff} and \eqref{v-lower-bound-on-Z}.

\subsubsection{Test vectors}\label{test-vect-v}\  We now construct the test vector satisfying the spectral bounds \eqref{v-regular-range} and \eqref{v-cutoff}. Let $v\in C^\8_{\mathrm{even}}(S^1)$ be a fixed even smooth function with small support around the points $\pi/4$ and $-3\pi/4$ (e.g., supported in the interval of size $0.01\cdot\omega$ around these points). We normalize it by  the $L^2$-norm. Note that while the function $v$ is fixed, we view it as a family of vectors in changing representations  $V_\tau$ (i.e., the action of $\PGLR$ in each of these spaces is different). We then define  the open neighborhood  $U\subset \bG$ of the identity by the condition $\|g-Id\|_{\bG}\leq 0.001\cdot\omega$ on the matrix norm (e.g., with respect to the adjoint action of $\bG$ on its Lie algebra). It is easy to see that this condition implies that the support of all vectors $\pi_\tau(g)v$ is contained in the interval of the size $0.02\cdot\omega$ around the points $\pi/4$ and $-3\pi/4$. Moreover, we have the bound $|g'(\theta)|\leq 1.1$ for the derivative on these intervals.

We consider the integral \eqref{d-chi-circle} defining $d_{s,\epsilon;\tau,\eps}(\pi_\tau(g)v)$. Recall that  for $g\in \bG^+$, the action is given by $\pi_{\tau,\eps}(g)v(\theta)=v(g\inv(\theta))|g'(\theta)|^\frac{\tau-1}{2}$ in the circle realization of $V_{\tau,\eps}$. We have
\begin{multline}\label{value-d-on-v}
d_{s,\epsilon;\tau,\eps}(\pi_{\tau,\eps}(g)v)= \cr  (2\pi)\inv\int\limits_{ S^1}v(g\inv(\theta))|g'(\theta)|^\frac{\tau-1}{2}|\cos(\theta)|^\frac{-1-\tau+s}{2}|\sin(\theta)|^\frac{-1-\tau-s}{2} \mathrm{ sign}(\theta)^{\epsilon+\eps} \mathrm d\theta\ .
\end{multline}

We view this integral as an oscillatory integral in the parameter $|\tau|\to\8$, and apply to it the stationary phase method. The main feature of the vectors $v$ that we are going to exploit is that the support of $\pi_\tau(g)v$ is disjoint from singularities of the {\it amplitude}  function in the oscillatory integral \eqref{d-chi-circle} at points $0,\ \pi$ and $\pm\pi/2$.
It is easy to see that in this oscillatory integral the phase function  has at most one (up to the central symmetry)  non-degenerate  critical point on the support of $v(g\inv(\theta))$ if $|s|\leq 0.5\cdot\omega|\tau|$. The amplitude function in \eqref{value-d-on-v} is always smooth, and hence the stationary phase method implies the bound \eqref{v-regular-range} in this range. On the other hand, for $|s|\geq 0.5\cdot\omega|\tau|$, the phase function does not have critical points on the support of $v(g\inv(\theta))$ and integration by parts implies the bound \eqref{v-cutoff}.
 Hence bounds \eqref{v-regular-range} and \eqref{v-cutoff} are satisfied for our choice of vector $v$.

 We remark that the resonance case in the integral \eqref{value-d-on-v} appears when the stationary point of the phase is close to the singularity of the amplitude. By choosing the function $v$ to have small enough support, we have eliminated such a possibility.

\subsection{Proof of bound \eqref{v-lower-bound-on-Z}}\ 
We now deduce the lower bound \eqref{v-lower-bound-on-Z} from QUE on the unit cotangent bundle $S^1(\bX)\simeq Z_+$ of the compact surface $\bX$.
\subsubsection{QUE}\label{QUE-sect}\  We use the following form of QUE. For a {\it fixed} even integer $m$, we consider a norm one Maass form $\psi^m_\tau$ of weight $m$ which is obtained from the Maass form $\psi_\tau=\psi^0_\tau$ by Maass lowering/raising  operators. In fact, $\psi^m_\tau$ is the unique, up to a multiple constant, vector of $K$-type  $m$ in the representation $V_{\psi_\tau}$. The QUE property is then translated into the decay of matrix coefficients (see \cite{Z1}, \cite{Z2}). Namely, we assume that for any smooth  function $\Psi\in C^\8(Z_+)$ with mean zero, we have
\begin{equation}\label{QUE} \langle \Psi\psi^m_\tau,\psi^n_\tau\rangle_{L^2(Z_+)}\to 0
\end{equation} as $|\tau|\to\8$ and $m,\ n $ are fixed. Note that we do not assume any dependence on $m$ and $n$ of the rate of the decay in \eqref{QUE}.  We note that usually QUE is formulated geometrically on the manifold $Z_+$ while we will to work on $Z$. It is enough to assume the decay of matrix coefficients of the form $\langle \Psi\psi^m_\tau,\psi^0_\tau\rangle_{L^2(Z_+)}$ since one can shift indexes using Maass operators (or even only the decay of the coefficients $\langle \Psi\psi^0_\tau,\psi^0_\tau\rangle_{L^2(Z_+)}$; however in this case the proof is more complicated and uses trilinear functionals as in \cite{Re3}).

\subsubsection{Approximation} We consider an approximation of the test vector $v$ by a $K$-finite combination. We denote by  $t\in(0,1)$ a real parameter whose value  we will  specify later.  We choose a natural number $N\geq 0$ and coefficients $\al_i\in\bc$ such that the vector $u_t=v-\sum_{i=-N}^N\al_i e_i$ satisfies    $\|u_t\|\leq t \|v\|$. Let $v_t=\sum_{i=-N}^N\al_i e_i$ be the {\it finite} combination of $K$-types. 
Let $f\in C^\8(Z)$ be as in \eqref{v-lower-bound-on-Z}. 
We have 
\begin{equation}\label{aprox} \langle |\psi_v|^2,f\rangle_{L^2(Z)}=\langle |v_t+u_\eps|^2,f\rangle= \langle |v_t|^2,f\rangle+\langle 2(u_\eps\bar v_t+\bar u_t v_t),\psi\rangle+\langle |u_t|^2,f\rangle\ .
\end{equation}

From QUE \eqref{QUE} we have 
\begin{enumerate}
	\item $\langle |v_t|^2,f\rangle\to\int\limits_Zf\cdot \|v_t\|^2>0$ as $|\tau|\to\8$, 
	\item $|\langle 2(u_t\bar v_t+\bar u_t v_t),f\rangle|\leq 2\sup|f|\cdot \|v_t\|\cdot\|u_t\|=2t\sup|f|\cdot \|v_t\|^2,$
	\item $|\langle |u_t|^2,f\rangle|\leq 2\sup|f|\cdot \|u_t\|^2=t^2\sup|f|\cdot \|v_t\|^2.$
\end{enumerate}

Hence by choosing $t$  sufficiently small depending on $f$, but not on $\tau$ (e.g., $t\leq (10\sup|f|)\inv\cdot \int f$)  we can make  the first term on the right in \eqref{aprox} to dominate other terms. This proves there is $b>0$ depending on $f$ only, such that   $\int_{x\in l\cdot U}|\psi_v(x)|^2f(x)\mathrm d\mu \geq b$. This finishes the proof of bound \eqref{v-lower-bound-on-Z}.

\section{Geodesic period, $L$-functions and Waldspurger's theorem}\label{waldspurger-sect}\ 
The proof of the lower bound 
\[
n(\psi_{\lambda}) \gg_{_ \eps} \lambda_{\psi}^{\frac{1}{27}- \eps}\ ,
\]
for the number of sign changes of a Hecke-Maass form on a closed geodesic (from which the Theorem 
\ref{nodal-domains-thm-intro} follows; see  Section \ref{nodal-count-sect}) goes exactly along the lines explained in \cite{GRS1},  p. 1558, for the case of a split geodesic (i.e., connecting two cusps).  There are two main inputs (beyond the restriction Theorem \ref{closed-g-intro}) in the proof of the lower bound for $n(\psi_{\lambda})$: a non-trivial bound on $L^\8$-norm of an eigenfunction and a (conjectural) bound on Fourier coefficients  of eigenfunctions along closed geodesics.  In the present setup of general co-compact arithmetic Riemann surfaces the corresponding $L^{\infty}$ bound of $\lambda^{\frac{1}{4}-\frac{1}{27}}$ is proved  by a straightforward adaption of \cite{BM}. This explains Remark \ref{sup-norm-remark}(b). 
To treat Fourier coefficients along split geodesic, in \cite{GRS1} we appealed to  the Lindel\"{o}f conjecture for the Hecke $L$-function $L(s,\psi)$. For the present situation of a {\it closed} geodesic the analogous assumption concerns the size of coefficients $\al_n(\psi_\lambda)$ coming from the expansion \eqref{z45} (or what is the same the representation theoretic coefficients $a_{\chi_{s_n,\epsilon}}(\tilde\psi_\tau)$ from \eqref{a-n-coeff-deff}). The crucial assumption we use is that these satisfy the  bound 
\begin{equation}\label{Lindelof-g-periods}
|a_{\chi_{s_n,\epsilon}}(\tilde\psi_\tau)| \ll_{_ \eps} \lambda_{\psi}^{\eps}\ ,\end{equation} for any $\eps>0$.

We claim that such a bound again would follow from  the Lindel\"{o}f conjecture for the appropriate $L$-functions.

In this section we discuss the relation between geodesic periods  of Hecke-Maass forms (i.e., coefficients $a_{\chi_{s_n,\epsilon}}(\tilde\psi_\tau)$) and certain automorphic $L$-functions. This  is based on the celebrated theorem of Waldspurger \cite{Wa}. We use this relation in order to justify our assumption \eqref{Lindelof-g-periods} on the size of geodesic periods of Hecke-Maass forms (see also \cite{R5}) by showing that it again follows from the Lindel\"{o}f conjecture for the corresponding $L$-functions.

\subsection{Waldspurger's formula} We start with the adelic formulation of the Waldspurger's formula (in order to quote \cite{Wa}), and then translate it into the classical language. Let $D$ be a quaternion algebra over $\bq$ which we assume is split over reals. It is well-known that the group of integer units of $D$ gives rise to a co-compact Fuchsian group. In fact, to cover all arithmetic Fuchsian  groups, we have to consider quaternion algebras $D$ over a number field $k$. One has to assume that $k$ is totally real and that $D$ is split over exactly one archimedian completion of $k$. Up to comensurability, all arithmetic Fuchsian  groups are obtained from such quaternion algebras (see \cite{Ta75}). The Waldspurger theorem holds over a general number field. In the exposition  of the Waldsurger theorem below we will assume, for simplicity, that $k=\bq$. We indicate small changes needed in order to cover the general case (see Remark \ref{num-field}). All the explicit examples in Section \ref{sect-examples} come from quaternion algebras defined over $\bq$.   

The group $G_1=D^\times/\bq^\times$ could be viewed as a special orthogonal group of a $3$-dimensional quadratic  space $W$ over $\bq$ (i.e., preserving a rational quadratic form). Let $T$ be an anisotropic torus of $G_1$. Then $T$ could be considered as a special orthogonal group of a rational $2$-dimensional plane in $W$. Any such torus corresponds to a quadratic field $E$ (the splitting field of $T$), and we have the map $1\to \bq^\times\to E^\times\to T\to 1$. We consider (rational) adelic points of $G_1$ and $T$. A (unitary) character $\om$ of $T(\ba)/T(\bq)$ could be considered as a (Hecke) character of $\ba^\times_E/ E^\times$ trivial on $\ba^\times_\bq/ \bq^\times$. Let $v$ be a finite or infinite valuation of $\bq$ and let $\bq_v$ be the corresponding local field (i.e., $\bq_v$ is equal to $\bq_p$ or $\br$). We  fix a Haar measure on each local component $T_v=T(\bq_v)$ as we will explain below (see Section \ref{local-int-sect}). The Tamagawa measure $\mathrm d_\ba t$ on $T(\ba)$ then satisfies $\mathrm d_\ba t=C_T\prod_v\mathrm dt_v$ for some positive constant $C_T>0$ to be described later.  

Let $\pi$ be an automorphic cuspidal representation of $G_1(\ba)$. It could be considered as a representation of $D^\times (\ba)$ with a trivial central character. Let $\Pi$ be its base change to $\GL_2(\ba_E)$. Choose a non-zero cusp form $\psi=\otimes_{v}\psi_v\in \pi\simeq\otimes_v\pi_v$.  Then Proposition 7 of \cite{Wa}  states that for an appropriate  finite set $S$ of (ramified) primes depending on $G_1$, $T$, $\om$ and $\psi$, the following relation holds:
\begin{equation}\label{Walds1}
\frac{\left|\int\limits_{T(\ba)/T(\bq)}\psi(t)\om(t)\ \mathrm d_\ba t\right|^2}{\langle\psi,\psi\rangle }=  C(G_1,T)\frac{L^S(\Pi\otimes \om, 1/2)}{L^S(\pi, Ad, 1)}\times \prod_{v\in S}\frac{I_v(\psi_v,\om_v)}{\|\psi_v\|_v^2}\ .
\end{equation}

Here $L^S(\Pi\otimes \om, s)$ is the partial  Hecke-Jacquet-Langlands  $L$-function on $\GL_2$, $L^S(\pi, Ad, s)$ is the partial adjoint $L$-function, $C(G_1,T)$ is an explicit constant depending on the group and the torus only, and  the quantities $I_v(\psi_v,\om_v)$ are given by an explicit local integral described below (see \eqref{Walds2-local-int}). 

An important point (for us) is that the set of ramified primes $S$ is {\it fixed} in our problem. Indeed, we are interested in a set of Hecke-Maass forms of a {\it fixed} level (i.e., on a fixed Riemann surface), and hence  representations $\pi$ we consider have  bounded ramification, and we only consider vectors in these representations of a bounded ramification. This allows us to split the problem of bounding the right hand side in \eqref{Walds1} into two separate problems: the (difficult) global problem of bounding $L$-functions and the (easy) problem of bounding local integrals.  

\subsection{Geodesic periods}\ Our aim is to show that bounds on $L$-functions appearing in \eqref{Walds1} imply  bounds for  periods along closed geodesics of the corresponding Hecke-Maass form. In particular, we will show that the Lindel\"{o}f conjecture for relevant $L$-functions is consistent with the bound \eqref{Lindelof-g-periods}. We first note that in the classical language of geodesic periods of a Hecke-Maass form $\psi$ one assumes that the level $N$ of $\psi$ is fixed, and hence the Riemann surface is of the form $Y_N=G_1(\bq)\backslash G_1(\ba)/K_\8K(N)$ for some appropriate open compact subgroup $K(N)\subset G_1(\ba)^f$ of the group of finite adeles of $G_1$. This means that outside of some finite set $S$ of primes (i.e., those dividing $N$)  the vector $\psi_v$ in the decomposition  $\psi=\otimes_{v}\psi_v$ is the standard $K_v$-fixed vector. Consider the set $\Sigma_E(N)=T(\bq)\backslash T(\ba)/T(\ba)\cap K(N)\subset Y_N$ of closed geodesics on the (congruence) Riemann surface $Y_N$ corresponding to the quadratic field $E$ (here we assume that $E$ is a real quadratic field and hence the set $\Sigma_E(N)$ is a union of circles where the number of circles is equal to the appropriate class number). In fact, the set $\Sigma_E(N)$ has  the natural structure of the abelian group of the form $S^1\times Cl$ where $Cl$ is a finite group (which coincides with an appropriate class group). We are interested in the restriction $\psi|_{\Sigma_E(N)}$, or more precisely in the restriction to every connected component of $\Sigma_E(N)$. Any (smooth) function on $\Sigma_E(N)$ can be decomposed into the Fourier series with respect to characters of $T(\ba)$. Clearly we need to consider only Hecke characters (i.e., characters trivial on $T(\bq)\simeq E^\times/\bq^\times$) which are  unramified  outside $S$. Moreover, at ramified primes the ramification is also bounded by that of $\psi$ (i.e., by $N$). The set of such characters is a discrete group $\bz\times \widehat{Cl}$. We consider bounds on (partial) $L$-functions in term of their analytic conductor (see \cite{IS-GAFA}). The analytic conductor  of $L(\Pi\otimes \om, s)$ depends on the infinity type of $\Pi_\8\otimes \om_\8$ and on the ramification of $\Pi\otimes \om$. The latter however is bounded in the problem we are interested in as
explained above. Hence the analytic conductor of  $L(\Pi\otimes \om, s)$ for the set of characters we are dealing with essentially depends only on the infinity type of $\pi$ and $\om$. In order to use the Waldspurger
formula \eqref{Walds1} for bounds on periods of $\psi$, we need to have a bound on $L$-functions appearing on the right in \eqref{Walds1} and a bound on the local integrals $I_v(\psi_v,\om_v)$. There is a very strong bound $ L(\pi, Ad, 1)\gg_{\eps} C(\pi)^{-\eps}$ for any $\eps>0$, in terms of  the analytic conductor $C(\pi)$ of $\pi$. Over $\bq$, this bound  is  due to Hoffstein and Lockart \cite{HL94} and  uses a corresponding upper bound of Iwaniec \cite{I}. Over a number field, such a bound follows from results on Siegel zeros in \cite{Ba} and \cite{HR} and the upper bound in \cite{Mo}.

While there are subconvexity bounds in all parameters for $L(\Pi\otimes \om, s)$ (see \cite{MiVe}), these are far from enough for our purposes and hence we assume the Lindel\"{o}f conjecture for this $L$-function in order to justify the bound \eqref{Lindelof-g-periods}. We are left to deal with the local integrals $I_v(\psi_v,\om_v)$. 

For the (split) infinite prime, the representation $\pi_\8$ is of principal unitary series and the above local integrals  can be evaluated explicitly in terms of the  Gamma function.  It turns out that these coincide with  values of invariant functionals we computed in Section \ref{geod-exp-sect}. Namely, for even characters $\om_\8=\chi_{s,\epsilon}$ as in Section \ref{gr-act-setup}, and the function $\psi_{e_0}$ corresponding to the $K_\8$-fixed vector, this is given by 
$|d_{s,\epsilon;\tau,\eps}(e_0)|^2$ from  \eqref{d-chi-e-0}.   For odd characters and the function $\psi_{e'_0}$ with $e'_0=\pi(\mathfrak{n})e_0$, the corresponding value is given by
$|(\tau-1)\inv d_{s,\epsilon;\tau,\eps}(e'_0)|^2$ from \eqref{d-chi-e-0-prime} (the extra factor $(\tau-1)\inv$ appears because of the $L^2$-normalization).
																				
	Next we claim that local integrals are uniformly bounded  at ramified finite primes.

\subsection{Local integrals}\label{local-int-sect} To describe local integrals at finite primes we first have to fix local measures.  Fix a nontrivial additive character $\psi$ of $\ba/\bq$. We choose a Haar measure $\mathrm dx_v=\zeta_v(1)\inv|x_v|\inv \mathrm dx_v$ on $\bq_v^\times$, where $\mathrm dx_v$ is the self-dual (additive) Haar measure on $\bq_v$ with respect to $\psi_v$.
We choose a Haar measure on $E_v^\times$ in a similar way.  This defines the Haar measure $\mathrm dt_v$ on $T_v$. These measures satisfy the global identity $\mathrm d_\ba t=\Lambda(\chi_{E/\bq},1)\inv \prod_v \mathrm dt_v$, where $\Lambda(\chi_{E/\bq},s)$ is the completed $L$-function of the quadratic character associated to $E$. We have then 

\begin{eqnarray}\label{Walds2-local-int}
I_v(\psi_v,\om_v)=  \int\limits_{T_v}\langle \pi_v(t_v)\psi_v,\psi_v\rangle_v\ \om(t)d t_v \ .
\end{eqnarray} It is easy to see that the integral above is absolutely convergent. First, we note that for places $v\in S$ where the algebra $D$ is ramified (i.e., it is not isomorphic to the algebra of matrices $M_2$), the corresponding group of units $D_v^\times$ is compact modulo the center. Hence the integration in \eqref{Walds2-local-int} is over a fixed compact set and the integral $I_v(\psi_v,\om_v)$ is uniformly bounded for all $\pi_v$ and $\psi_v$ (in fact in this case $\pi_v$ is finite-dimensional). For places $v\in S$ where $D$ splits,  the unitary representation $\pi_v$ is non-trivial, and hence the decay of matrix coefficients imply the absolute convergence. However, we need a little bit more in order to conclude that $I_v(\psi_v,\om_v)$ are uniformly bounded as the cuspidal representation $\pi$ changes (i.e, $\pi_\8$ has a growing parameter, but the ramification of $\pi$ is fixed).  Note that $|\om(t)|\equiv 1$ and the measure $d t_v$ is essentially integrable over $T_v$, i.e., any effective polynomial decay of the matrix coefficient $\langle \pi_v(t_v)\psi_v,\psi_v\rangle_v$ would imply the uniform bound for $I_v(\psi_v,\om_v)$ for all $\pi_v$. Such a bound indeed is available thanks to non-trivial bounds towards the Ramanujan-Petersson conjecture. Indeed, we are looking for an effective polynomial decay for matrix coefficients $\langle \pi_v(t_v)\psi_v,\psi_v\rangle_v$ (i.e., a bound of the type $|\langle \pi_v(t_v)\psi_v,\psi_v\rangle_v|\leq C||t_v||^{-\s}$ for some absolute constants $C>0$ and $\s>0$ which are {\it independent} of the representation $\pi_v$ and  a normalized vector $\psi_v$ in it). The crucial point here is that $\pi_v$ are assumed to be local components of a cuspidal representation $\pi$ at a {\it finite} set of places $v\in S$, and vectors $\psi_v$ have some bounded ramification (i.e., belong to a {\it fixed} finite-dimensional space of $K_v$-finite vectors).  For such a family of matrix coefficients the effective polynomial decay holds if representations $\pi_v$ are separated  from the trivial representation, and this is what follows from  any non-trivial bound towards the Ramanujan-Petersson conjecture.

Hence we have shown that for the family of vectors $\psi_v$ coming from Hecke-Maass forms of fixed level, local integrals appearing in \eqref{Walds1} are uniformly bounded. This shows that an upper bound \eqref{Lindelof-g-periods} on geodesic periods of Hecke-Maass forms would  follow from the Lindel\"{o}f conjecture for $L(\Pi\otimes \om, 1/2)$ via the Waldspurger's formula \eqref{Walds1}. We note that other remarkable applications of the Waldspurger's formula \eqref{Walds1} include non-negativity of $L(\Pi\otimes \om, 1/2)$. This however requires much more delicate study of local integrals $I_v(\psi_v,\om_v)$ (including their non-vanishing) which we luckily can avoid. 

\begin{remark}\label{num-field} To cover more general co-compact arithmetic Fuchsian groups one needs to consider quaternion algebras $D$ defined over a totally real field $k$. There is very little change needed in the above argument in order to cover these. The algebra $D$ have to be split at exactly one real place. Hence at all other real places a cuspidal representation $\pi$ corresponding to a Hecke-Maass eigenfunction should be trivial (otherwise it would correspond to a function defined on a Riemann surface times copies of the sphere $S^2$ and would not descend to a Riemann surface alone). For such cuspidal representations $\pi$, the corresponding local integrals $I_v(\psi_v,\om_v)$ are trivial. Also, non-trivial bounds towards  the Ramanujan-Petersson conjecture are well-known over a number field.  Hence the  arguments above we made for $k=\bq$ are easily applicable. 
\end{remark}

\section{Explicit examples}\label{sect-examples}

\subsection{Symmetric fundamental domains: the non-cocompact case}\ 

Let $\sigma$ denote the orientation reversing (reflection) isometry $\sigma(z)= -\overline{z}$, for $z\in\mathbb{H}$. Let $\Gamma$ be a Fuchsian group of the first kind, that is a subgroup of $\PSL(2,\mathbb{R})$ that acts discontinuously on $\mathbb{H}$ with every point on $\mathbb{R}$ being a limit point of the orbit $\Gamma z$ for a $z\in\mathbb{H}$. We shall restrict ourselves to those $\Gamma$ such that $\sigma$ is in the normaliser of $\Gamma$. Every such group has  Dirichlet fundamental domains given by polygons
\[
\mathcal{D}_{w} = \left\{z\in\mathbb{H}\ :\ d(z,w)<d(z,\gamma w)\ \mathrm{ for\ all}\ \gamma\in \Gamma, \gamma\neq I\right\},
\]
where $w\in\mathbb{H}$ is chosen to be a point not fixed by any  $\gamma \in \Gamma$ (except the identity), and $d(\ ,\ )$ denotes the hyperbolic distance. Then, since $\sigma\Gamma\sigma=\Gamma$, if $\Re{w} =0$, one has $\sigma(\mathcal{D}_{w})=\mathcal{D}_{w}$. By choosing $w$ in this manner, one determines that $\Gamma$ always has fundamental domains whose closure are symmetric with respect to the action of $\sigma$. We call such domains (by abuse of language) {\it symmetric} fundamental domains. Dirichlet fundamental domains are sometimes not so practical for applications and in what follows seek simpler fundamental domains.

Let $\Gamma$ be a subgroup of finite index in $\Gamma(1)= \mathrm{ PSL}(2,\mathbb{Z})$ such that $\sigma$ is in the normalizer of $\Gamma$ (note that $\sigma$ is already in the normalizer of $\Gamma(1)$ ). We will use the following notation: a closed domain $\mathcal{D}$ will be considered {\it admissible} if after suitable identification of boundary arcs, the resulting surface, denoted by $\tilde{\mathcal{D}}$ is a fundamental domain of $\Gamma$ (this latter is not unique and is said to be associated to the former). We allow for the inclusion of cusps in our closed sets. Having chosen such an admissible domain $\mathcal{D}$, we seek to determine the sets 
\[
\mathrm{ Fix}_{\Gamma}(\sigma;\mathcal{D}) = \left\{z\in \mathcal{D}: \exists\ \gamma\in\Gamma\ \mathrm{such\  that}\  \sigma z = \gamma z \right\},
\]
and $\mathrm{ Fix}_{\Gamma}(\sigma;\tilde{\mathcal{D}})$ which is obtained from the former by requiring points  on the boundaries to be inequivalent modulo $\Gamma$. 

Let $\tilde{\mathcal{F}}$ denote the standard fundamental domain of $\Gamma(1) $ obtained from
\[
\mathcal{F} = \left\{z=x+iy \colon -\frac{1}{2}\leq x\leq \frac{1}{2},\quad x^2 + y^2 \geq 1\right\},
\]
\begin{figure}

\begin{center}
  \includegraphics[width=0.5\linewidth]{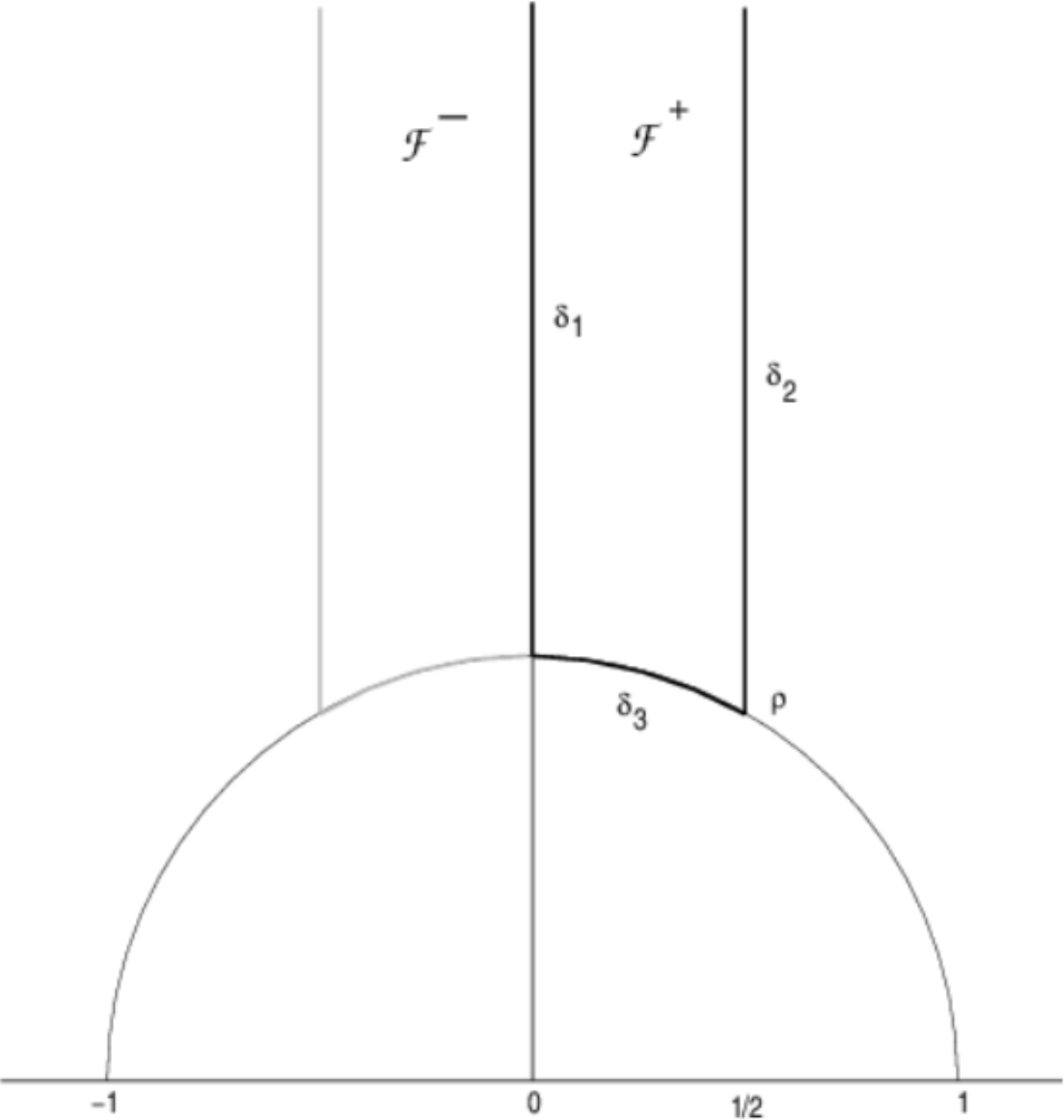}
  \caption{}
  \end{center}
\end{figure}
\begin{figure}
\begin{center}
  \includegraphics[width=0.5\linewidth]{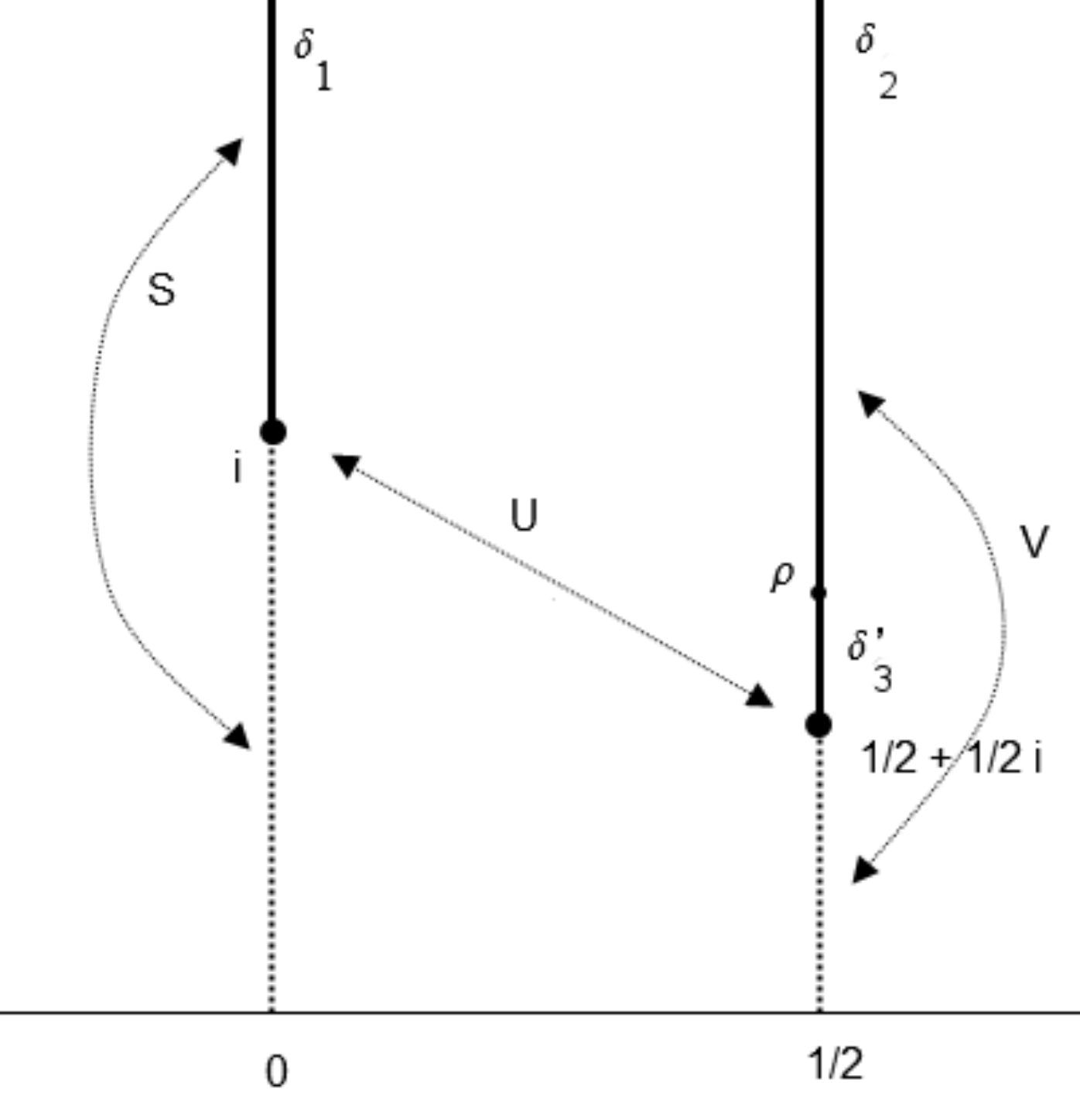}
  \end{center}
\caption{}
\end{figure}

\noindent with the appropriate identification of the boundaries; $\mF$ is invariant under $\sigma$ and we will call $\tmF$ a {\it symmetric standard} fundamental domain.  We partition $\mathcal{F}$ and write $\mathcal{F} = \mathcal{F}^{-} \cup \mathcal{F}^{\text{+}}$ with $\mathcal{F}^{\text{+}} = \{z\in \mathcal{F} \colon \Re(z)\geq 0\}$ and $\mathcal{F}^{-}=\sigma \mathcal{F}^{\text{+}}$. The common boundary $\mathcal{F}^{-} \cap \mathcal{F}^{\text{+}}$ is denoted by $\delta_{1}$ and the other boundaries of $\mF^{\text{+}}$ being $\delta_{2}$ and $\delta_{3}$ as indicated in Figure 2. Let $S =\left( \begin{smallmatrix} 0&-1\\ 1&0 \end{smallmatrix}\right )$ and $T =\left( \begin{smallmatrix} 1&1\\ 0&1 \end{smallmatrix}\right )$. Then, since $\sigma z = Sz$ for $z\in \delta_{3}$ and $\sigma z=T^{-1}z$ for $z\in \delta_{2}$, we see that 
\[\mathrm{ Fix}_{\Gamma(1)}(\sigma;\mF) = \delta\ \cup\ \sigma\delta \quad \mathrm{ and} \quad \mathrm{ Fix}_{\Gamma(1)}(\sigma;\tmF) = \delta,
\]
where $\delta := \delta_{1} \cup \delta_{3} \cup \delta_{2}$ is topologically a closed curve (containing the cusp). The structure of $\delta$ can be further clarified as follows (see Figure 3): the arc $\delta_{3}$ is mapped to the vertical line segment $\delta_{3}^{'}$ joining $\half +\half i$ to $\rho$ using $U:=\left( \begin{smallmatrix} 1&-1\\ 0&-1 \end{smallmatrix}\right )$, where $\rho$ is fixed while $i$ is mapped to the other endpoint. Next note that $\delta_{1}$ is half of the full closed geodesic, the imaginary axis (using $S$) while the vertical ray $\delta_{3}^{'}\cup\delta_{2}$ is half of the closed geodesic $\Re(z)=\half$ using $V :=\left( \begin{smallmatrix} 1&-1\\ 2&-1 \end{smallmatrix}\right )$ ; the combined curves being closed since $i$ is identified with $\half+\half i$, (the endpoints are equivalent). On the Riemann surface, the two closed geodesics are distinct, each containing the cusp and another common point.

 Let $\{\alpha_{1}, \alpha_{2}, \ldots, \alpha_{t}\}$ be a set of right coset representatives of $\Gamma\backslash \Gamma(1)$. It follows that it is then possible to choose coset representatives so that they are of the form $Q = \{\omega_{1}, \hat{\omega}_{1}, \omega_{2}, \hat{\omega}_{2}, \ldots,$ $ \omega_{r}, \hat{\omega}_{r}, \nu_{1}, \ldots , \nu_{s}\}$   such that $\nu_{j} \sim \hat{\nu}_{j} \mod \Gamma$ but with the $\omega_{j}$'s and $\hat{\omega}_{j}$'s inequivalent (such a set always exists with the integers $r$ and $s$ uniquely determined).

With $\mF$ as above, let $\mF_{\Gamma,Q} = \bigcup_{q\in Q}\ q\mF$. This is an admissible domain for $\Gamma$ and we call an associated $\widetilde{\mF_{\Gamma,Q}}$ a $standard$ fundamental domain of $\Gamma$, after identification of sides.  Such a domain $\mF_{\Gamma,Q}$ is made up of disjoint triangles (except for possible common boundary arcs). Moreover, any two interior points (in distinct triangles) are never $\Gamma$-equivalent. Applying $\sigma$ gives $\sigma\mF_{\Gamma,Q}=\bigcup_{q\in Q}\ \hat{q}\mF$, so that $\sigma\mF_{\Gamma,Q} =\mF_{\Gamma,Q}$ only if $ \hat{\nu}_{j} = \nu_{j}$, which is certainly true if $\nu_{j} = I$ or $S$. Thus, $\mF_{\Gamma,Q}$ is $\sigma$-symmetric only if $Q$ contains no $\nu_{j}$'s except $I$ and possibly $S$. Now suppose  there exist an $\alpha \in   \Gamma(1)$ such that $\alpha \notin \Gamma \cup \Gamma S$, with $\alpha$ satisfying $\hat{\alpha} \sim \alpha \mod \Gamma$. Then any set $Q$ must contain a $\nu_{j} \neq I\ \mathrm{ or}\ S$ so  that the corresponding $\mF_{\Gamma,Q}$ cannot be symmetric. This then gives us a criteria for checking the existence of symmetric standard fundamental domains.

Let us now consider subgroups $\Gamma$ not containing $S$ in what follows. Suppose $\Gamma$ has a symmetric, connected fundamental domain $\mathfrak{X}$, that is, $\sigma \mathfrak{X}^{\mathrm{ o}}=\mathfrak{X}^{\mathrm{ o}}$ for interior points. Also suppose that a set of cosets $Q$ exists containing a $\nu \notin \Gamma \cup \Gamma S$ with $\hat{\nu}\sim \nu$. Since $S$ is not in $\Gamma$, we will choose $Q$ to contain both $I$ and $S$. It follows that the associated standard fundamental domain contains the imaginary axis, denoted by $\delta^{\text{*}}$. Let $z\in \delta^{\text{*}}$ so that there is a $\gamma \in \Gamma$ with $\gamma z\in \mathfrak{X}$. We assume this image is an interior point, so that $\sigma\gamma z$ is too. Since $\sigma\gamma z = \hat{\gamma}z$, it follows that $\hat{\gamma}\gamma^{-1}\in \Gamma(1)$ fixes $z$, so that $\gamma = I$ (with $\gamma = S$ being excluded by assumption). It follows that $\delta^{\text{*}}$ is contained in $\mathfrak{X}^{\mathrm{ o}}$. Now consider $z=\nu\zeta \in \widetilde{\mF_{\Gamma,Q}}$, with $\zeta \in \delta_{1}$. If $\gamma v\zeta \in \mathfrak{X}^{\mathrm{ o}}$, then again so is the reflection, so that now we conclude there are two interior points $x_{1}$ and $x_{2}$ of $\mathfrak{X}$ so that $x_{1} = \widehat{\gamma \nu}(\gamma \nu)^{-1}x_{2}$. Since $\hat{\nu}\nu^{-1}\in\Gamma$, it follows that $\gamma \nu = I$ or $S$, which is impossible since $\nu \notin \Gamma \cup \Gamma S$. It follows that if such a $\nu\in \Gamma(1)$ exists with $\hat{\nu}\sim \nu \mod \Gamma$, then one cannot have such a symmetric fundamental domain (i.e. domains that map interior points to interior points of standard domains).\ 

As an application of our criteria about standard fundamental domains,  consider $\Gamma_{0}(N)$ for  $N\geq 2$. It has coset representatives given by matrices $\left( \begin{smallmatrix} a&b\\ c&d \end{smallmatrix}\right)$ with $d|N$, $(c,d)=1$, $c \mod \frac{N}{d}$ and with $a$ and $b$ chosen to satisfy $ad-bc=1$. 

If $N=p^{k}$ with $p$ an odd prime number and $k\geq 1$, we may choose the following set of representatives: $I$, $S$, $\left\{\left( \begin{smallmatrix} a&b\\ c&1 \end{smallmatrix}\right), \left( \begin{smallmatrix} a&-b\\ -c&1 \end{smallmatrix}\right)\ \mathrm{ with}\ 1\leq |c|\leq \frac{p^{k} -1}{2}\right\}$, and for any $1\leq l \leq k-1$ the set $\left\{\left( \begin{smallmatrix} a&b\\ c&p^{l} \end{smallmatrix}\right), \left( \begin{smallmatrix} a&-b\\ -c&p^{l} \end{smallmatrix}\right)\ \mathrm{ with}\ 1\leq |c|\leq \frac{p^{k-l}-1}{2},\ (c,p)=1\right\}$ (this latter being empty if $k=1$). Thus, for $p$ odd, $\Gamma_{0}(p^k)$ has a symmetric standard  fundamental domain.

On the other hand, if $N \neq p^k$ as above, then one can construct a $\nu$ as discussed above. For example, if $N=N_{1}N_{2}$, with $(N_{1},N_{2})=1$ and $N_{1}, N_{2} \neq 1$, then we choose  $\nu = \left( \begin{smallmatrix} a&b\\ N_{1}&N_{2} \end{smallmatrix}\right)\in \Gamma(1)$. Then $\hat{\nu}\nu^{-1} = \left( \begin{smallmatrix} *&*\\ -2N&* \end{smallmatrix}\right)\in \Gamma_{0}(N)$, and $\nu$ is not equivalent to $I$ and $S$. Thus by the discussion above, $\Gamma_{0}(N)$ will not have a symmetric standard fundamental domain.

It is also of interest to determine when, given a fundamental domain $\widetilde{D}$ of $\Gamma$ containing a full set of inequivalent cusps (we assume $0$ and infinity are cusps in $\widetilde{D}$ assuming that $\Gamma$ does not contain $S$), that all the cusps are contained in $\mathrm{ Fix}_{\Gamma}(\sigma;\widetilde{D})$.  Clearly 0 and infinity belong to $\mathrm{ Fix}_{\Gamma}(\sigma;\widetilde{D})$. If $\mathfrak{a}$ is any other cusp in $\widetilde{D}$   with $\alpha(\infty) = \mathfrak{a}$ for some $\alpha \in \Gamma(1)$, then there exists a non-zero $b\in \mathbb{Z}$ such that $\alpha T^{b}\alpha^{-1}\in \Gamma$. It then follows that  $\sigma\mathfrak{a} \sim \mathfrak{a} \mod \Gamma$ if and only if there exist a $k \in \mathbb{Z}$ such that $\hat{\alpha} T^{k}\alpha^{-1} \in \Gamma$. This occurs if and only if $\hat{\alpha}\alpha^{-1} \in \Gamma$. We illustrate when this does and does not occur with two explicit examples. 

Suppose $p\neq q$ are  odd primes. Then $\Gamma = \Gamma_{0}(pq)$ has exactly four inequivalent cusps, which we may take to be infinity, 0, $\frac{1}{p}$ and $\frac{1}{q}$ in a suitable fundamental domain (see Section 6.3.1). Choose integers $a$ and $b$ so that $ap-bq=1$. Then put $\alpha = \left( \begin{smallmatrix} 1&-a\\ p&-bq \end{smallmatrix}\right)$ so that $\alpha(\infty) = \frac{1}{p}$ and $\hat{\alpha}\alpha^{-1}\in \Gamma_{0}(pq)$. We then see that $\left( \begin{smallmatrix} 1-2ap&2a\\ 2bpq&1-2ap \end{smallmatrix}\right)(\frac{1}{p}) = -\frac{1}{p} = \sigma(\frac{1}{p})$ as required. The argument is analogous for the other cusp.

Now consider $\Gamma=\Gamma_{0}(p^{2})$, having $p+1$ inequivalent cusps. By the discussion above there is a symmetric standard fundamental domain, so that may choose the cusps symmetrically. It is easily shown that a set of inequivalent cusps are infinity, 0 and the points of the set $\mathcal{M}=\left\{\frac{1}{jp} ; 1\leq |j| \leq \frac{p-1}{2}\right\}$. We know that there are no $\alpha$'s (except for the discounted ones) in $\Gamma(1)$ satisfying $\hat{\alpha}\alpha^{-1} \in \Gamma_{0}(p^2)$ so that we do not expect any of the $p-1$ cusps in $\mathcal{M}$ to be in the fixed set, which is trivially the case since $\sigma(\frac{1}{jp})= \frac{1}{-jp}$.

\vspace{20pt}
We now give a detailed analysis in the case of $\Gamma_{0}(p)$ and $\Gamma_{0}(pq)$.


\subsubsection{Symmetric standard fundamental domain for $\Gamma_{0}(p)$.}\ 

Let $\rho = \frac{1}{2} + i\frac{\sqrt{3}}{2}$ denote the corner of $\mathcal{F}$, it having order 3, with the other corner being $\sigma\rho := \breve{\rho}$.\footnote{When convenient we shall denote $\sigma z$ by $\breve{z}$\ .}
Define for each $l\in \mathbb{Z}$
\[
\alpha_{l} =\left( \begin{smallmatrix} 0&-1\\ 1&-l \end{smallmatrix}\right ) = ST^{-l}\ .
\]
We note for now that
\[
\hat{T} = T^{-1}\quad ; \quad \hat{S} = S \quad \mathrm{ and} \quad \alpha_{l-j}=\alpha_{l}T^{j}\ ,
\] 
and that
\begin{equation}\label{e340}
 \hat{\alpha_{l}} = \alpha_{-l}\quad ; \quad \sigma \alpha_{l} =\alpha_{-l}\sigma\ .
\end{equation}
Let $\mF_{l}$ denote the image of $\mF$ under the map $\alpha_{l}$. It follows from the above considerations that
\[
\sigma\mF_{l} = \alpha_{-l}\sigma\mF = \alpha_{-l}\mF = \mF_{-l}\ .
\]
Put $c_{l}=\alpha_{l}\breve{\rho}$. Since $T^{-1}\rho = \breve{\rho}$ we have by \eqref{e340}
\[
\sigma c_{l}= \alpha_{l}\rho =\alpha_{l}T\breve{\rho} = (ST^{-l})T\breve{\rho}=c_{l-1}\ .
\]
The $\mF_{l}$'s are hyperbolic triangles with corners 
\[
<\alpha_{l}\infty,\ \alpha_{l}\rho,\ \alpha_{l}\breve{\rho}> \quad=\quad <0,\ \sigma c_{-l},\ c_{l}>\quad=\quad<0,\ c_{l-1},\ c_{l}>.
\]
 Thus $\mF_{l}$ and $\mF_{l+1}$ share a common geodesic boundary arc joining $0$ to $c_{l}$. By direct computation, we have
\[
c_{l} = \frac{l+\frac{1}{2}}{(l+\frac{1}{2})^{2}+\frac{3}{4}} + i\frac{\frac{\sqrt{3}}{2}}{(l+\frac{1}{2})^{2}+\frac{3}{4}}\ ,
\]
so that $\Im{(c_{l})} \leq \frac{\sqrt{3}}{2}$ and $|\Re{(c_{l})}|\leq \frac{1}{2}$\ . Moreover, 
\[
\Re{(c_{0})}=\Re{(c_{1})}=\frac{1}{2}\quad ; \quad \Re{(c_{-1})}=\Re{(c_{-2})}= -\frac{1}{2}\ ,
\]
 and otherwise, for $l\geq 1$
\[
\Re{(c_{l})}>\Re{(c_{l+1})}\to 0^{\text{+}} \quad;\quad \Im{(c_{l})}>\Im{(c_{l+1})}\to 0^{\text{+}}.
\]
Since $\sigma\mF_{l} = \mF_{-l}$, the corners of $\mF_{-l}$, for $l\geq 0$, are $<0,c_{-l},c_{-l-1}>$. Hence, corresponding to the sequence $c_{0}, c_{1}, c_{2}, \ldots$ in $x>0$, we have in $x<0$ the mirror image sequence $c_{-1}, c_{-2}, c_{-3}, \ldots$ obtained by the reflection $\sigma$. Gluing together the triangles $\mF_{l}$ for $0\leq l \leq L$, together with their reflections $\mF_{-l}$ and with $\mF$ gives the  geometric configuration in Fig.5 :

\begin{figure}[!ht]
  \begin{center}
    \includegraphics[width=0.7\linewidth]{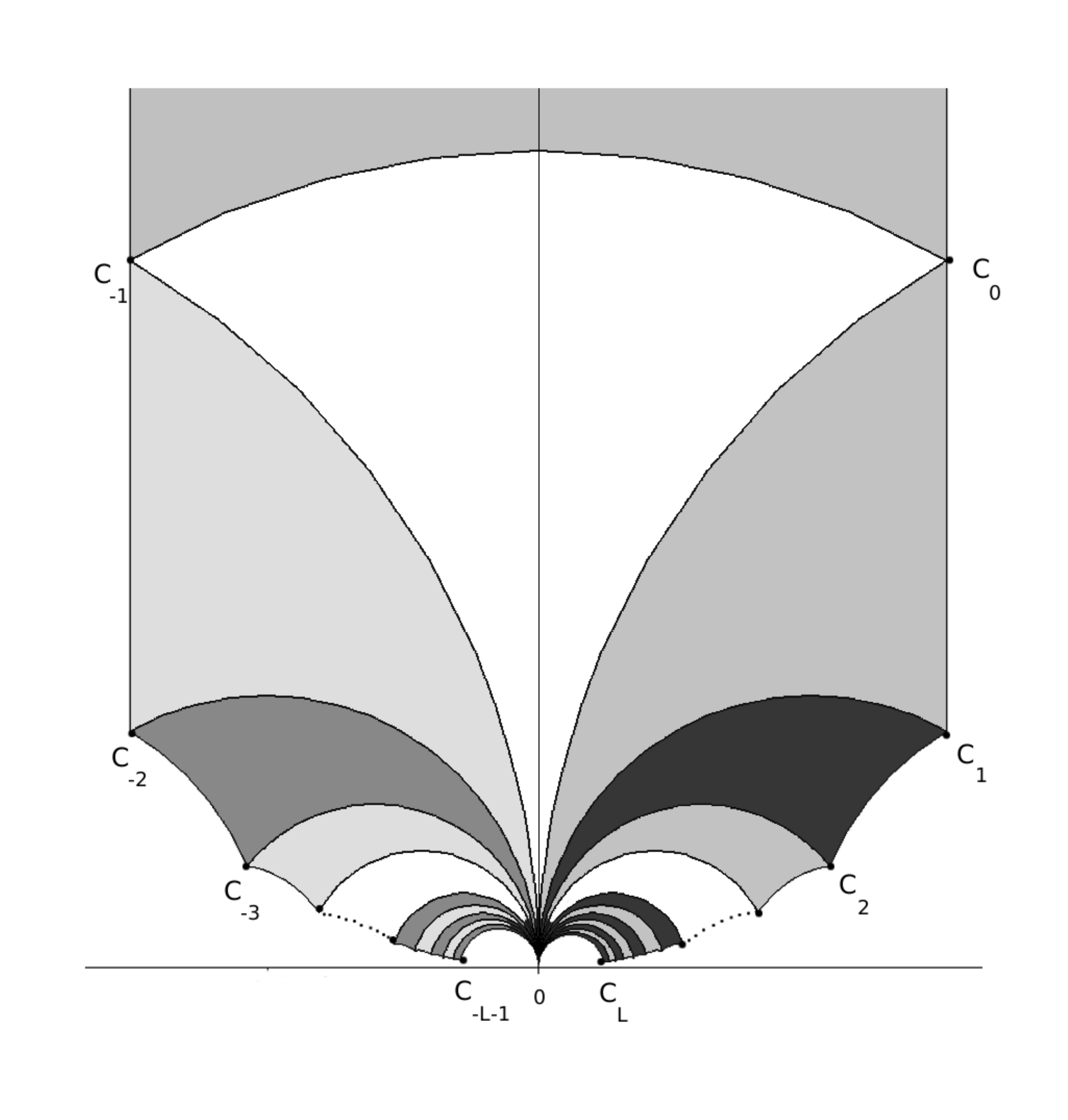}
  \end{center}
  \caption{}
\end{figure}

Let $p\neq 2$ be a prime number and put $L=\frac{p-1}{2}$. Then the set $\{I, \alpha_{l}\}$ forms a set of right coset representatives of $\Gamma_{0}(p)$ in $\Gamma(1)$ with $0 \leq |l| \leq L$, so that 
\begin{equation}\label{e380}
\mathbb{X}_{p} := \mF\ \cup \ \bigcup_{l}\mF_{l},
\end{equation}
gives rise to $\widetilde{\mathbb{X}_{p}}$, a standard fundamental domain of $\Gamma_{0}(p)$ in $\mathbb{H}$, after pairwise identifications of the boundary arcs. We will call the arc of the boundary of the $\mF_{l}$ not having a cusp as an endpoint  an $edge$ of $\mathbb{X}_{p}$, and denote it by $\delta_{l}^{(p)}$.  By construction, $\mathbb{X}_{p}$ is invariant under $\sigma$ and we may decompose it so that $\mathbb{X}_{p} = \mF^{\text{+}}(p)\ \cup\ \sigma\mF^{\text{+}}(p)$, where $\mF^{\text{+}}(p)$ consists of the elements with non-negative real part.

In what follows, we shall use $\sim$ to mean equivalence modulo $\Gamma_{0}(p)$.

\subsubsection{}\ When is $c_{u} \sim c_{v}$ and when is $\sigma c_{u}\sim c_{u}$ ?

We have $\breve{\rho}$ is fixed by $1, \alpha_{-1}$ and $\alpha^{2}_{-1}$, so that $c_{u} \sim c_{v}$ if and only if $\alpha_{u}\alpha^{\eta}_{-1}\alpha_{u} \in \Gamma_{0}(p)$ with $\eta = 0$ or $\pm 1$. This gives us three cases to consider namely (a) $u \equiv v$, (b) $u(v+1)\equiv -1$ or (c) $v(u+1)\equiv -1$, all congruences modulo $p$.

In case (a), since $-L-1 \leq u, v \leq L$, the only non-trivial solution is $u=L, v= -L-1$, so that 
\[
c_{L} \sim c_{-L-1} = \sigma c_{L}.
\]
For case (b), we first note that if $v=1$, then there are two values of $u$, namely $L$ and $-L-1$, so that
\[
c_{1} \sim c_{L} \sim c_{-L-1} = \sigma c_{L}.
\]
For all other $v \neq \pm 1$ there exist a unique $-L < u < L$ such that $c_{u} \sim c_{v}$. The case $v=2$ gives 
\[
c_{-2} \sim c_{1} \sim \sigma c_{1},
\]
using $c_{-2}=\sigma c_{1}$. Case (c) follows from case (b) by symmetry.

Collecting these together gives an answer to the second part of the question as follows: since $\sigma c_{u} = c_{-u-1}$, putting $v=-u-1$ in the congruences above give us the three cases namely (a) $2u\equiv -1$, (b) $u^{2} \equiv 1$ and (c) $(u+1)^2 \equiv 1$, so that the only non-negative choices are $u = L, 1$ and $0$. As we have done above we have only
\begin{equation}\label{e405}
c_{L}\sim c_{-L-1} = \sigma c_{L}\quad ;\quad c_{1}\sim c_{-2} = \sigma c_{1}\quad ;\quad c_{0}\sim c_{-1} = \sigma c_{0}.
\end{equation}


\subsubsection{} We are interested in the points in $\mathbb{X}_{p}$ fixed by $\sigma$ modulo $\Gamma_{0}(p)$. If $z$ is such a point, say $\sigma z = \gamma z$ for some $\gamma \in \Gamma_{0}(p)$, then $w=\sigma z$ is also a fixed point since $\sigma w = \sigma \gamma \sigma w := \hat{\gamma}w$ and $\hat{\gamma}\in \Gamma_{0}(p)$ (as $\sigma$ lies in the normalizer of $\Gamma_{0}(p)$). Thus, it suffices to find the fixed points in $\mF^{\text{+}}(p)$, and they must necessarily occur on the boundary. Let $\delta^{(p)}$ denote the set of all points on $\mF^{\text{+}}(p)$ that are fixed by $\sigma$, that is 
\[
\delta^{(p)} = \left\{z\in \mF^{\text{+}}(p) \ :\ \sigma z \sim z \mod \Gamma_{0}(p)\right\}.
\]
Let $\delta^{\text{*}}$ denote the imaginary axis, so that it is a component of $\delta^{(p)}$. We now determine which of the other arcs of the boundary of  $\mF^{\text{+}}(p)$ contain points that make up $\delta^{(p)}$. Let $\delta^{(p)}_{l}$ denote one such arc not containing a cusp. It is an edge of $\mF_{l}$ and is the image under $\alpha_{l}$ of the segment in $\mF$ that is part of the unit circle,  for some $l$ satisfying $0\leq l\leq L$. So, if $w\in \delta^{(p)}_{l}$, write $w=\alpha_{l}z$ with $|z|=1$. Since $w$ is fixed by $\sigma$ modulo $\Gamma_{0}(p)$, there exists a $\gamma \in \Gamma_{0}(p)$ so that 
\[
\sigma w = \gamma w \iff \sigma\alpha_{l}z = \gamma\alpha_{l}z \iff \alpha_{-l}\sigma z = \gamma\alpha_{l}z \iff \alpha^{-1}_{-l}\gamma\alpha_{l}z = \sigma z.
\]
Now write $\alpha^{-1}_{-l}\gamma\alpha_{l} = \left( \begin{smallmatrix} A&B\\ C&D \end{smallmatrix}\right ) \in \Gamma(1)$ so that $\frac{Az+B}{Cz+D} = -\overline{z}$. Taking imaginary parts, with $z=x+iy$  gives us $\frac{y}{|Cz+D|^2} = y$, and so $C^{2} + D^{2} + 2CDx = 1$. Also, we have $Az+B = -\overline{z}(Cz+D)$ and taking imaginary parts give us $A=D$. If $CD\neq 0$, then $x=\frac{1-C^{2}-D^{2}}{2CD}$ and since $|x|\leq \frac{1}{2}$, it follows that $(2|C|-|D|)^{2} + 3D^{2} \leq 4$. Hence $|C|=|D|=1$, implying $x=\pm \frac{1}{2}$ so that $z=\rho$ or $\breve{\rho}$. Thus $z=c_{l}$ for some $l$ so that by \eqref{e405}, the only possible non-negative values for $l$ are $0$, $1$ and $L$. 

Next, if $C=0$, we necessarily have $x = -\frac{B}{2A}$ and $|A|=|D|=1$, so that $|x|= \frac{1}{2}$, with the same conclusion as before. Finally, if $D=0$, then $\left( \begin{smallmatrix} A&B\\ C&D \end{smallmatrix}\right ) = \pm S$. This implies that $\alpha_{-l}S\alpha_{l} \in \Gamma_{0}(p)$, which can occur only when $l^{2} \equiv 1 \mod p$, so that $l$=1. This means the edge with endpoints $\alpha_{1}\rho$ and $\alpha_{1}\breve{\rho}$, is fixed. This is the edge with endpoints $c_{0}$ and $c_{1}$.

It remains to check the two arcs in $\mF^{\text{+}}_{p}$ with cusps. The first, joining $0$ with $c_{L}$ is a segment of the image of the line $z= -\frac{1}{2} + iy$ under the map $\alpha_{L}$. A direct computation shows that $\sigma \alpha_{L}z = \left( \begin{smallmatrix} 1&0\\ -p&1 \end{smallmatrix}\right )\alpha_{L}z$ so that this arc, the one joining $0$ to $c_{L}$ is fixed. Also, the vertical line segment joining $c_{0}$ to $i\infty$ is trivially fixed using the map $T$. Finally, one observes that the $\sigma$-reflection of these fixed curves are $\Gamma_{0}(p)$-equivalent to themselves. We thus have

\begin{lemma}\label{B}\ 
\begin{itemize}
\item[(i)] $\delta^{(p)}$ above, is a simple  curve consisting of three geodesic arcs: the vertical line segment joining the two cusps $0$ and $i\infty$, the arc joining the cusp $0$ and the vertex $c_{L}$, and the vertical line segment joining the vertex $c_{1}$ and the cusp at infinity. Moreover,
\[
\mathrm{ Fix}_{\Gamma_{0}(p)}(\sigma;\mathbb{X}_{p}) = \delta^{(p)}\ \cup\ \sigma\delta^{(p)}.
\]
\item[]
\item[(ii)] $\mathrm{ Fix}_{\Gamma_{0}(p)}(\sigma;\widetilde{\mathbb{X}_{p}} )$ is a simple  closed curve containg the cusp at infinity. It is closed since $c_{L} \sim c_{1}$. We have
\[
\mathrm{ Fix}_{\Gamma_{0}(p)}(\sigma;\widetilde{\mathbb{X}_{p}}) = \delta^{(p)},
\]
after gluing $c_{L}$ with  $c_{1}$.
\item[]
\item[(iii)] $\mathrm{ Fix}_{\Gamma_{0}(p)}(\sigma;\mathbb{H} )$ modulo $\Gamma_{0}(p)$ can be viewed as the union of two inequivalent complete geodesics with the endpoints equivalent. More precisely, it is the union of the geodesics $\Re(z)=0$ and $\Re(z)=\half$, with the endpoints $0$ and $\half$ being equivalent.
\end{itemize}
\end{lemma}
\begin{proof} To verify the third part, we observe that the arc joining $0$ with $c_{L} = \frac{2p}{p^{2}+3}+\frac{2\sqrt{3}}{p^{2}+3}i$ \ is equivalent to the vertical segment joining $\half$ to $c_{1} = \half + \frac{1}{2\sqrt{3}}i$ using $\left( \begin{smallmatrix} L&-1\\ p&-2 \end{smallmatrix}\right )\in \Gamma_{0}(p)$.
\end{proof}
Recall the definition of a non-separating closed curve on a (connected) surface: it is non-separating if when the surface is cut along the curve, the cut surface is still connected.

\begin{lemma} 
$\mathrm{ Fix}_{\Gamma_{0}(p)}(\sigma;\widetilde{\mathbb{X}_{p}})$ is a separating  curve in $\widetilde{\mathbb{X}_{p}}$ for $p = 2, 3, 5, 7 $ and $13$, these having a simply connected fundamental domain. For all other $p$, $\mathrm{ Fix}_{\Gamma_{0}(p)}(\sigma;\widetilde{\mathbb{X}_{p}})$ is a non-separating  curve in $\widetilde{\mathbb{X}_{p}}$.
\end{lemma}
\begin{proof}
It is seen from Fig. 5 and Fig. 6 that $\delta^{(p)}$ is not a separating  curve if and only if one may join any point in $\mF^{\text{+}}(p)$ with any other in $\mF^{-}(p)$ by a curve that passes through an non-vertical edge in $\mF^{\text{+}}(p)$. This will be so only when there is a non-vertical edge in $\mF^{\text{+}}(p)$ that is $\Gamma_{0}(p)$-equivalent to a non-vertical edge in $\mF^{-}(p)$ (noting from the discussion above that edges are not a part of the fixed point set except for the vertical edges from $\mF_{\pm1}$). Let $\alpha_{l}$ and $\alpha_{k}$ determine two such equivalent edges with $2\leq l\leq L$ and $2\leq |k|\leq L$ . Then it follows that this is only possible if $\alpha_{l}^{-1}\gamma\alpha_{k}$ fixes the unit circle for some $\gamma\in \Gamma_{0}(p)$. Thus either $l=k$ or $\alpha_{l}S\alpha_{k}^{-1} \in \Gamma_{0}(p)$, so that $lk \equiv -1\ ({\rm mod}\ p)$. Hence the existence of  equivalent non-vertical edges on opposite sides of the imaginary axis is determined by the answer to the following:

\vspace{10pt}
\noindent{\it Question}: Are there any $l$ satisfying $2\leq l\leq L$ such that, if $l^{\text{*}}$ denotes the unique inverse of $l$ in the interval $[-L,L]$, then $l^{\text{*}}\in [2,L]$ ?

\vspace{10pt}
As $2L \equiv -1\ ({\rm mod}\ p)$, we may disregard $\pm2$ and $\pm L$ from consideration. If $p+1 = 2Q$ for an odd composite number $Q=uv$, then  since $3 \leq u,v \leq \frac{p+1}{6}$, one may take $l=u$ and $k=-\frac{p+1}{u}$. On the other hand if $Q$ is even, then take $l=4$ and $k=-\frac{p+1}{4}$. This argument is valid if $p\geq 7$ and accounts for almost all $p$ since the remaining $p$'s are of the form  $p+1 =2q$ for a prime number $q$, which is a thinner set of primes. We do not have any obvious candidates for these latter primes.

We now show that if $p$ is sufficiently large, there are many choices for $l$ above. For $1\leq K_{1}, K_{2} \leq  L$, let $N_{p}(K_{1},K_{2}) = |\{ k\in [1,K_{1}] : k^{\text{*}} \in [1,K_{2}]\}|$. We determine the asymptotic behaviour of $N_{p}(K_{1},K_{2})$ as $p$ grows and apply this to $N_{p}(L,L)$ to deduce our conclusion. It is easily seen that
\[
N_{p}(K_{1},K_{2}) = \frac{1}{p^2}\sum_{a,b(\mathrm{ mod}\ p)} S(a,b;p)\overline{F(a,K_{1})F(b,K_{2})},
\]
where $S(a,b;p)$ is the Kloosterman sum and $F(a,K) = \sum_{s=1}^{K}e(\frac{as}{p})$. Since $S(a,b;p)=-1$ if $p|a$ or $p|b$ but not both while $S(a,b;p)=p-1$ if $p|a$ and $p|b$, we have
\begin{multline}\label{eq20}
N_{p}(K_{1},K_{2}) =  \frac{p-1}{p^2}K_{1}K_{2}\ -\ \frac{1}{p^2}\sum_{\substack{b(\mathrm{ mod}\ p)\\(b,p)=1}}\left(K_{1}F(b,K_{2}) + K_{2}F(b,K_{1})\right) \\ + \frac{1}{p^2}\sum_{\substack{a,b(\mathrm{ mod}\ p)\\ (ab,p)=1}} S(a,b;p)\overline{F(a,K_{1})F(b,K_{2})}. 
\end{multline}
As
\[
\sum_{\substack{b(\mathrm{ mod}\ p)\\(b,p)=1}}F(b,K)\ =\ \sum_{s=1}^{K}\sum_{\substack{b(\mathrm{ mod}\ p)\\(b,p)=1}}e(\frac{sb}{p})\ =\ -K,
\]
the first two terms above contribute $\frac{p-1}{p^2}K_{1}K_{2} + 2\frac{K_{1}K_{2}}{p^2} = \frac{(p+1)K_{1}K_{2}}{p^2}$.

The sum above involving Kloosterman sums is estimated using Weil's estimate $|S(a,b;p)| \leq 2\sqrt{p}$ so that the total contribution to the last sum in \eqref{eq20} is 
\[
\leq 2\frac{\sqrt{p}}{p^2}\sum_{\substack{a(\mathrm{ mod}\ p)\\(a,p)=1}}|F(a,K_{1})|\sum_{\substack{b(\mathrm{ mod}\ p)\\(b,p)=1}}|F(b,K_{2})|.
\]
Now $|F(a,K)|\leq |\sin(\pi\frac{a}{p})|^{-1} \leq 2(\pi\frac{a}{p})^{-1}$, if $1 \leq |a| \leq L$. Hence for say $p\geq 100$ 
\[
\sum_{\substack{a(\mathrm{ mod}\ p)\\(a,p)=1}}|F(a,K)| \leq 2 \sum_{a=1}^{L} 2\left(\frac{\pi a}{p}\right)^{-1} \leq \frac{4}{\pi}p(\log L +1).
\]
Thus, the last term in \eqref{eq20} is bounded by $\frac{32}{\pi^2}\sqrt{p}(\log L  +1)^{2}$.

Hence for all $p$ sufficiently large, 
\[
N_{p}(K_{1},K_{2}) = \frac{K_{1}K_{2}}{p} + O(\sqrt{p}\log^{2}p).
\]

Consequently, $N_{p}(L,L) = \frac{1}{2}L + O(\sqrt{p}\log^{2}p)$, so that it is equally likely that $k^{\text{*}}$ is in $[1,L]$ as not.

Thus, on average, about half the edges in $\mF^{\text{+}}(p)$ are equivalent to edges in $\mF^{-}(p)$.  Moreover since  $N_{p}(K_{1},L) \neq 0$ if $K_{1}\gg \sqrt{p}\log^{2}p$, it follows that  there is always a $l$ in the interval $2\leq l\ll \sqrt{p}\log^{2}p$ which gives rise to a pair of such edges.

To deduce the statement of our lemma, we have for $p \geq 100$, 
\[
\big{|}N_{p}(L,L) - \frac{(p+1)(p-1)^{2}}{4p^2}\big{|} \leq \frac{32}{\pi^2}\sqrt{p}(\log p  +1)^{2},
\]
so that $ N_{p}(L,L) \geq \frac{p}{1000}$ for all primes $ p\geq 3\times10^{7}$. A direct computer calculation verifies the statement of the lemma for the remaining primes greater than 13 and for $p=11$.

\end{proof} 
\begin{figure}[!ht]
  \begin{center}
    \includegraphics[width=0.7\linewidth]{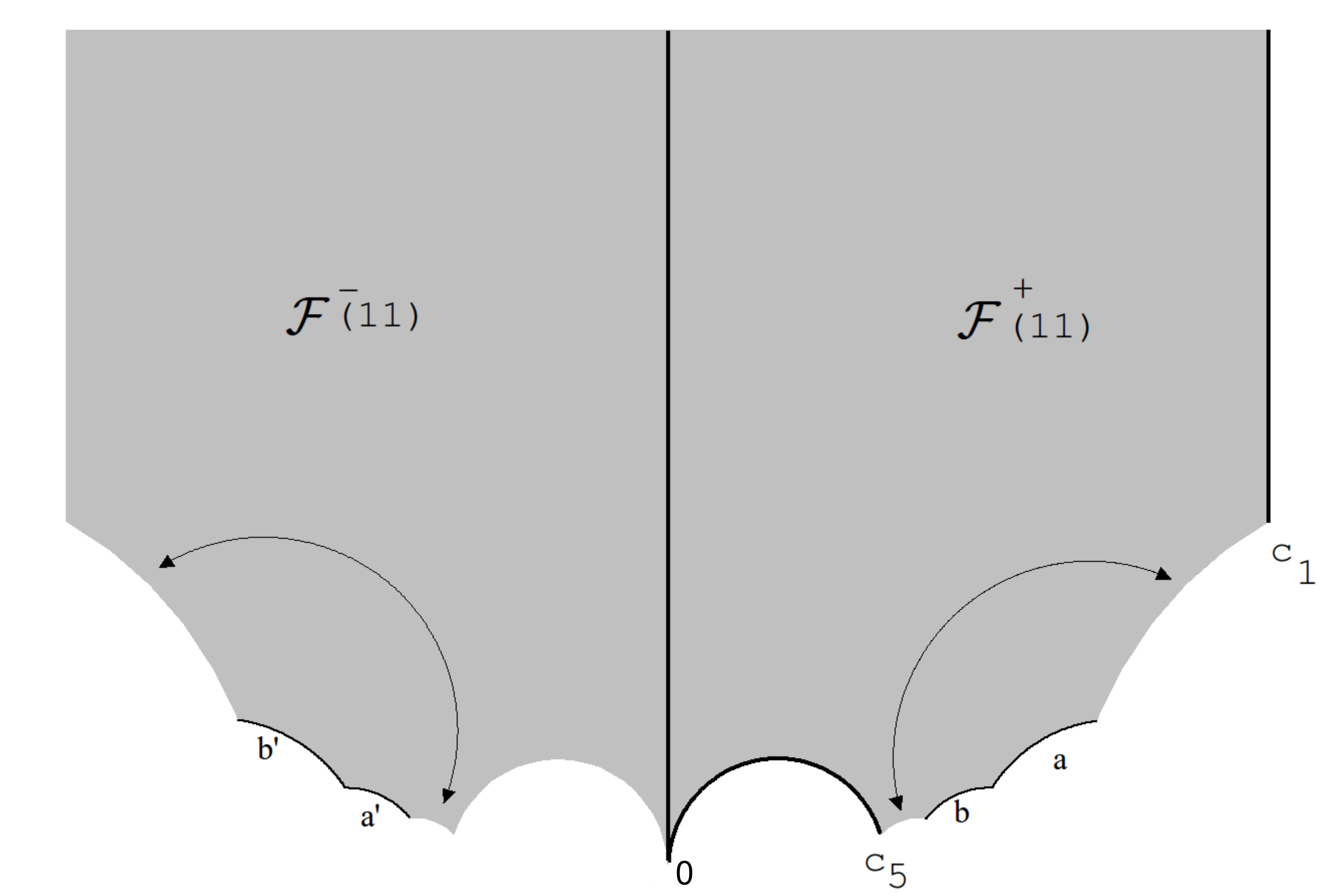}
  \end{center}
  \caption{}
\end{figure}

We illustrate in Figure 6  the above results with the example from $\Gamma_{0}(11)$, which is a subgroup of $\Gamma(1)$ of index $12$, with the resulting Riemann surface having genus 1. The dark arcs make up the fixed point set, with $c_{5}$ indentified with $c_{1}$, giving a closed curve.  After identification of sides, the ``hole'' in the Riemann surface can be seen by the two arcs denoted by $a$ and $b$, with equivalent images modulo $\Gamma_{0}(11)$ denoted by $a'$ and $b'$ respectively. The fact that $\mathrm{ Fix}_{\Gamma_{0}(11)}(\sigma;\mathbb{X}_{11})$ is not separating can then be easily seen as one passes from the right to the left via $a\cup b$.


\subsubsection{Non-symmetric standard fundamental domain for $\Gamma_{0}(pq)$.}\ 

For $p\neq q$ be odd prime numbers, $\Gamma_{0}(pq)$ has index $(p+1)(q+1)$ in $\Gamma(1)$ and has four inequivalent cusps, which may be taken to be infinity, $0$, $\frac{1}{p}$ and $\frac{1}{q}$. We construct a fundamental domain for $\Gamma_{0}(pq)$ using coset representatives as follows: $Q= \{I\}\cup Q_{1}\cup Q_{2}\cup Q_{3}$, where 
\begin{equation}\label{Gammapq}\end{equation}
\begin{itemize}
\item[(1)] $Q_{1} = \left\{ \alpha_{l} : 0\leq l \leq \frac{pq -1}{2}\right\}$,
\item[(2)] $Q_{2} = \left\{\alpha_{p}\alpha_{j} : 0\leq j \leq \frac{q -1}{2}\right\}$, and
\item[(3)] $Q_{3} = \left\{\alpha_{q}\alpha_{k} : 0\leq k \leq \frac{p -1}{2}\right\}$.
\end{itemize}

While this choice is not of the type discussed in Section 6.1, it is more convenient in terms of the geometry. Let us denote by $\mathbb{X}^{\text{*}}_{p}$ the closed set obtained from  $\mathbb{X}_{p}$ by removing $\mF$ (see \eqref{e380}); it is symmetric relative to $\sigma$. The closed domain
\[
\mathbb{X}_{pq} := \mF\ \ \cup  \bigcup_{\alpha \in Q_{1}} \alpha\mF \ \cup\ \alpha_{p}\mathbb{X}_{q}^{\text{*}}\ \cup\ \alpha_{q}\mathbb{X}_{p}^{\text{*}},
 \]
is admissible so that we let $\widetilde{\mathbb{X}_{pq}}$  denote an associated fundamental domain of $\Gamma_{0}(pq)$, after identifications of sides. It has the four inequivalent cusps mentioned above. Here, the first two subdomains of $\mathbb{X}_{pq}$ in \eqref{Gammapq} are $\sigma$-symmetric while the latter two are not (though they do have a conjugated symmetry). As before, we are interested in the structure of the fixed set of $\sigma$ in $\mathbb{X}_{pq}$. One notes that since $-\frac{1}{p} \sim \frac{1}{p}$ and $-\frac{1}{q} \sim \frac{1}{q}$, all the cusps are in $\mathrm{ Fix}_{\Gamma_{0}(pq)}(\sigma;\widetilde{\mathbb{X}_{pq}})$ (see Section 6.1 for a discussion). For the rest, we now consider three cases with $z \in \mathrm{ Fix}_{\Gamma_{0}(pq)}(\sigma;\mathbb{X}_{pq})$:

\ 
\begin{itemize}[leftmargin=*]
\item[(i)] Suppose $z\in \mF\  \cup\  \bigcup_{\alpha \in Q_{1}} \alpha\mF$ such that $\sigma z \sim z \mod \Gamma_{0}(pq)$. The domain here has the structure of Figure 5 (see Fig. 7 as an example) and we use the notation there with $L = \frac{pq -1}{2}$. If $z$ is in the interior, then it must be on the imaginary axis. Hence if $z$ is on the boundary of $\mF$, it must be on $\delta_{2}\ \cup\ \sigma\delta_{2}$ since they are $\sigma$-fixed by $T^{\pm 1}\in \Gamma_{0}(pq)$. Next suppose $z \in \bigcup_{0\leq |l|\leq L} \alpha_{l}\mathcal{S}$, where $\mathcal{S}$ is the part of $\mF$ on the unit circle. This happens only for those $l$ with $\hat{\alpha}_{l}S\alpha^{-1} \in \Gamma_{0}(pq)$, so that $l^{2} \equiv 1\ ({\rm mod}\ pq)$. Apart from $l= \pm 1$, there are two other choices for $l$, which we denote by $\pm l^{\text{*}}$. These give us the arcs $\alpha_{\pm 1}\mathcal{S}\ \cup\ \alpha_{\pm l^{\text{*}}}\mathcal{S}$ = $\alpha_{1}\mathcal{S}\ \cup\ \sigma\alpha_{1}\mathcal{S}\ \cup\ \alpha_{l^{\text{*}}}\mathcal{S}\ \cup\ \sigma\alpha_{l^{\text{*}}}\mathcal{S}$. The arcs corresponding to $l = \pm 1$ are vertical line segments on the lines $x=\pm \frac{1}{2}$ and they connect with $\delta_{2}$ and  $\sigma\delta_{2}$ respectively. Finally, one checks the arcs coming from $\alpha_{\pm L}\mF$ containing the cusp $0$ (these being the only remaining arcs not in the interior). The arcs joining $0$ to $c_{L}$ and $0$ to $c_{-L-1}$ are fixed by $\left( \begin{smallmatrix} 1&0\\ \mp pq&1 \end{smallmatrix}\right )\in \Gamma_{0}(pq)$ respectively. Thus, $\mathrm{ Fix}_{\Gamma_{0}(pq)}(\sigma;\mF\  \cup\  \bigcup_{\alpha \in Q_{1}} \alpha\mF)$ is the union of the image $\alpha_{l^{\text{*}}}\mathcal{S}$ and a simple curve $\mathcal{C}_{1}$, consisting of the imaginary axis, the arc joining $0$ with $c_{L}$ and the vertical segment joining $c_{1}$ with the cusp at infinity, together with their $\sigma$-reflections (see Fig. 7).
\item[]
\item[(ii)] If $z\in \alpha_{p}\mathbb{X}_{q}^{\text{*}}$, then write $z=\alpha_{p}\alpha_{j}w$, with $w\in\mF$. Then, $\sigma z\sim z \mod \Gamma_{0}(pq)$ if and only if $\widehat{\alpha_{p}\alpha_{j}}\sigma w\sim \alpha_{p}\alpha_{j}w \mod \Gamma_{0}(pq)$, so that then $\sigma w \sim w \mod \Gamma(1)$. Thus $w\in \delta \cup \sigma\delta$. 

If $w\in \delta_{1}$, then we have  $\widehat{\alpha_{p}\alpha_{j}}(\alpha_{p}\alpha_{j})^{-1} \in \Gamma_{0}(pq)$  since $\sigma w =w$ (being an interior point of $\mF$). This happens only when $pj \equiv 1\ ({\rm mod}\ q)$, for which there is a unique $j$, which we denote by $p^{\text{*}}$. Thus $\alpha_{p}\alpha_{p^{\text{*}}}\delta_{1} \subset \mathrm{ Fix}_{\Gamma_{0}(pq)}(\sigma;\alpha_{p}\mathbb{X}_{q}^{\text{*}})$.

If $w\in \mathcal{S}$, we have $\sigma w = Sw$, so that $\widehat{\alpha_{p}\alpha_{j}}S(\alpha_{p}\alpha_{j})^{-1} \in \Gamma_{0}(pq)$. This is impossible since it implies $(pj -1)^{2}\equiv p^{2}\ ({\rm mod}\ pq)$.

If $w\in \delta_{2}$, we have $\sigma w = T^{-1}w$, so that $\widehat{\alpha_{p}\alpha_{j}}T^{-1}(\alpha_{p}\alpha_{j})^{-1} \in \Gamma_{0}(pq)$. This holds if $2p(pj -1)\equiv p^{2}\ ({\rm mod}\ pq)$, so that there is a unique $j_{1} \equiv p^{\text{*}} + \frac{q+1}{2}\ ({\rm mod}\ q)$, with the corresponding image  $\alpha_{p}\alpha_{j_{1}}\delta_{2}$  in $\mathrm{ Fix}_{\Gamma_{0}(pq)}(\sigma;\alpha_{p}\mathbb{X}_{q}^{\text{*}})$.

Finally, if $w\in \sigma\delta_{2}$, we repeat the argument above with $T^{-1}$ replaced by $T$ to find that $\alpha_{p}\alpha_{j_{2}}\sigma\delta_{2}$ is in $\mathrm{ Fix}_{\Gamma_{0}(pq)}(\sigma;\alpha_{p}\mathbb{X}_{q}^{\text{*}})$, with $j_{2}\equiv j_{1} -1\ ({\rm mod}\ q)$. One checks that $j_{1} \neq -\frac{q-1}{2}$, so that $j_{2} = j_{1}-1$. But then $\alpha_{p}\alpha_{j_{2}}\sigma\delta_{2}= \alpha_{p}\alpha_{j_{1}-1}\sigma\delta_{2} = \alpha_{p}\alpha_{j_{1}}T\sigma\delta_{2} = \alpha_{p}\alpha_{j_{1}}\delta_{2}$ as above.

\item[]
\item[(iii)] The analysis for $\alpha_{q}\mathbb{X}_{p}^{\text{*}}$ is identical to that above, with $q^{\text{*}}\ ({\rm mod}\ p)$ defined accordingly, and with the analogous values $k_{1}\ ({\rm mod}\ p)$. So  $\mathrm{ Fix}_{\Gamma_{0}(pq)}(\sigma;\alpha_{q}\mathbb{X}_{p}^{\text{*}})$ is $\alpha_{q}\alpha_{q^{\text{*}}}\delta_{1}\ \cup\ \alpha_{q}\alpha_{k_{1}}\delta_{2}$, where $k_{1} \equiv q^{\text{*}} + \frac{p+1}{2} \mod p$ and $qq^{\text{*}}\equiv 1\ ({\rm mod}\ p)$. Note that necessarily $pp^{\text{*}} + qq^{\text{*}}=1$, so that $p^{\text{*}}$ and $q^{\text{*}}$ have opposite signs.
\item[]
\end{itemize} 

Collecting the above together gives $\mathrm{ Fix}_{\Gamma_{0}(pq)}(\sigma,\mathbb{X}_{pq}) = \mathcal{C}_{1}\ \cup\ \sigma\mathcal{C}_{1}\ \cup\ \mathcal{C}_{2}$, where 
\[
\mathcal{C}_{2}= \alpha_{l^{\text{*}}}\mathcal{S}\ \cup\ \alpha_{-l^{\text{*}}}\mathcal{S}\ \cup\ \alpha_{p}\alpha_{p^{\text{*}}}\delta_{1}\ \cup\ \alpha_{q}\alpha_{q^{\text{*}}}\delta_{1}\ \cup\ \alpha_{p}\alpha_{j_{1}}\delta_{2}\ \cup\ \alpha_{q}\alpha_{k_{1}}\delta_{2}.
\]

To determine $\mathrm{ Fix}_{\Gamma_{0}(pq)}(\sigma,\widetilde{\mathbb{X}_{pq}})$ we have to remove any equivalent arcs from the above. The analysis in (i) shows that the arcs making up $\sigma\mathcal{C}_{1}$ are equivalent to the corresponding arcs in $\mathcal{C}_{1}$. Also note (as in the the previous section) that $c_{L} \sim c_{1}$. Next, one has $\alpha_{-l^{\text{*}}}S\alpha_{l^{\text{*}}}^{-1} \in \Gamma_{0}(pq)$, so that $\alpha_{-l^{\text{*}}}\mathcal{S}$ is equivalent to $\alpha_{l^{\text{*}}}\mathcal{S}$. Among these two, we keep the one whose subscript is congruent to $1 \mod p$ and $-1 \mod q$; we will denote this choice by $l^{\text{*}}$. One now checks that the remaining arcs are not equivalent except possibly at endpoints. Checking the endpoints, using $X:=\alpha_{p}\alpha_{p^{\text{*}}}S(\alpha_{q}\alpha_{q^{\text{*}}})^{-1} = \left( \begin{smallmatrix} q-pp{^{\text{*}}}^{2}&-(1+p^{\text{*}}q^{\text{*}})\\ pq(1+p^{\text{*}}q^{\text{*}})&qq{^{\text{*}}}^{2} -p \end{smallmatrix}\right )\in\Gamma_{0}(pq)$, one sees that $\alpha_{p}\alpha_{p^{\text{*}}}i \sim \alpha_{q}\alpha_{q^{\text{*}}}i \mod \Gamma_{0}(pq)$ so that $\alpha_{p}\alpha_{p^{\text{*}}}\delta_{1}$ and  $\alpha_{q}\alpha_{q^{\text{*}}}\delta_{1}$ have a common point in $\widetilde{\mathbb{X}_{pq}}$. Similarly, using the choice of $l^{\text{*}}$ above, and the definitions of $j_{1}$ and $k_{1}$, one shows  that $\alpha_{l^{\text{*}}}\mathcal{S}$ and $\alpha_{p}\alpha_{j_{1}}\delta_{2}$ have  equivalent points at $\alpha_{l^{\text{*}}}\rho \sim \alpha_{p}\alpha_{j_{1}}\rho$, where the element $Y:= \alpha_{l^{\text{*}}}\alpha_{1}^{-1}(\alpha_{p}\alpha_{j_{1}})^{-1} \in \Gamma_{0}(pq)$ maps one to the other; and $\alpha_{l^{\text{*}}}\mathcal{S}$ and $\alpha_{q}\alpha_{k_{1}}\delta_{2}$ have equivalent points at $\alpha_{l^{\text{*}}}\sigma\rho \sim \alpha_{q}\alpha_{k_{1}}\rho$ where $Z:=\alpha_{l^{\text{*}}}T^{-1}\alpha_{1}^{-1}(\alpha_{q}\alpha_{k_{1}})^{-1} \in \Gamma_{0}(pq)$ is the map (thus $\alpha_{l^{\text{*}}}\mathcal{S}$ bridges the other two arcs in $\widetilde{\mathbb{X}_{pq}}$ as in Fig. 8). We thus have schematically the following closed curve $\mathcal{C}_{3}$ with connecting equivalent points shown and the arcs indicated by arrows:
\begin{equation}\label{e556}
\begin{split}
\left(\frac{1}{p}\right) \xrightarrow{\ \ (1)\ \ } \left\{\substack{\alpha_{p}\alpha_{p^{\text{*}}}(i) \\ \alpha_{q}\alpha_{q^{\text{*}}}(i)}\right\}_{X} \xrightarrow{\ \ (2)\ \ } \left(\frac{1}{q}\right) \xrightarrow{\ \ (3)\ \ } \left\{\substack{\alpha_{q}\alpha_{k_{1}}(\rho) \\  \alpha_{l^{\text{*}}}\sigma(\rho)}\right\}_{Z} \quad\quad \\  \xrightarrow{\ \ (4)\ \ } \left\{\substack{\alpha_{l^{\text{*}}}(\rho)\\  \alpha_{p}\alpha_{j_{1}}(\rho) }\right\}_{Y} \xrightarrow{\ \ (5)\ \ } \left(\frac{1}{p}\right).
\end{split}
\end{equation}

We have

\begin{lemma}\label{C}\ 
Let  $p$ and $q$ be distinct odd primes.
\begin{itemize} 
\item[(i)]  $\mathrm{ Fix}_{\Gamma_{0}(pq)}(\sigma,\widetilde{\mathbb{X}_{pq}})$ is  a disjoint union $\mathcal{C}_{1}\cup\mathcal{C}_{3}$ of two closed simple curves, each containing a pair of cusps.
\item[]
\item[(ii)] On $\mathbb{H}$, $\mathrm{ Fix}_{\Gamma_{0}(pq)}(\sigma,\widetilde{\mathbb{X}_{pq}})$ can be realised as the union of two pairs of non-intersecting geodesics with rational endpoints, with each pair having $\Gamma_{0}(pq)$-equivalent endpoints.
\end{itemize}
\end{lemma}
\begin{proof} As discussed in Lemma \ref{B}, the curve $\mathcal{C}_{1}$ is piecewise equivalent to the two vertical lines $\Re(z)=0$ and $\Re(z)=\half$ with $0 \sim \half$. 

For the other curve in $\mathrm{ Fix}_{\Gamma_{0}(pq)}(\sigma,\widetilde{\mathbb{X}_{pq}})$ given in \eqref{e556}, we first consider the arcs denoted by $(1)$ and $(2)$; we may assume $p<q$.  The element $X^{-1}=\alpha_{q}\alpha_{q^{\text{*}}}S(\alpha_{p}\alpha_{p^{\text{*}}})^{-1}\in \Gamma_{0}(pq)$ maps the geodesic arc $(1)$ connecting $\frac{1}{p}$ and the point $\alpha_{p}\alpha_{p^{\text{*}}}(i)$ in \eqref{e556} to the new geodesic arc  connecting $\frac{q*}{qq^{\text{*}} -1}$ and $\alpha_{q}\alpha_{q^{\text{*}}}(i)$; together with the geodesic arc (2) connecting $\alpha_{q}\alpha_{q^{\text{*}}}(i)$ and $\frac{1}{q}$ in \eqref{e556} this gives the full geodesic $\mathcal{E}_{1}$ having $\frac{1}{q}$ and $\frac{q*}{qq^{\text{*}} -1}$ as endpoints since $\alpha_{q}\alpha_{q^{\text{*}}}(0)= \frac{q*}{qq^{\text{*}} -1}$ and $\alpha_{q}\alpha_{q^{\text{*}}}(\infty)=\frac{1}{q}$. Note that if $p^{\text{*}}<0<q^{\text{*}}$, the new endpoint $\frac{q*}{qq^{\text{*}} -1}$ is between 0 and $\frac{1}{q}$, while if  $q^{\text{*}}<0<p^{\text{*}}$, then it lies between $\frac{1}{q}$ and $\frac{1}{p}$. Thus the curves $(1)$ and $(2)$ are piecewise $\Gamma_{0}(pq)$-equivalent to the full geodesic $\mathcal{E}_{1}$. 

Next, for the arcs $(3)$, $(4)$ and $(5)$, we see that $Z$ maps $\frac{1}{q}$ to $\alpha_{l^{\text{*}}}(-1)=\frac{1}{l^{\text{*}}+1}$ so that the arc of $\alpha_{l^{\text{*}}}\mathcal{S}$ joining $\alpha_{l^{\text{*}}}(-1)$ with $\alpha_{l^{\text{*}}}\sigma\rho$ is the image of the arc $(3)$ in \eqref{e556} under $Z$. Similarly, $Y$ maps $\frac{1}{p}$ to $\alpha_{l^{\text{*}}}(1)= \frac{1}{l^{\text{*}}-1}$ ( note that by definition, $l^{\text{*}}\neq \pm 1$), so that the arc of $\alpha_{l^{\text{*}}}\mathcal{S}$ joining $\alpha_{l^{\text{*}}}(1)$ with $\alpha_{l^{\text{*}}}\rho$ is the image of the arc $(5)$ in \eqref{e556} under $Y$. Then together with the arc $(4)$ we get the geodesic $\mathcal{E}_{2}$with end points $\frac{1}{l^{\text{*}}+1}$ and $\frac{1}{l^{\text{*}}-1}$ and the curves $(3)$, $(4)$ and $(5)$ are piecewise equivalent to it.

We see that $\mathcal{C}_{3}$ is piecewise equivalent to the geodesics $\mathcal{E}_{1}\cup\mathcal{E}_{2}$. Comparing endpoints, we have from the above that $Y(\frac{1}{p})= \frac{1}{l^{\text{*}}-1}$ and $YX(\frac{q^{\text{*}}}{qq^{\text{*}}-1})=\frac{1}{l^{\text{*}}-1}$, so that the endpoints are pairwise equivalent.

\end{proof}
We illustrate the analysis above by considering $\Gamma_{0}(15)$, which is a subgroup of index 24, with the associated Riemann surface of genus 1. In Figure 7 we see the closed curve $\mathcal{C}_{1}$ with the points $c_{7}$ and $c_{1}$ identified. Figure 8 shows the closed curve $\mathcal{C}_{3}$, where `a' is $\alpha_{5}\alpha_{5^{\text{*}}}(i)$, and is identified with `d' $=\alpha_{3}\alpha_{3^{\text{*}}}(i)$. Here $p=3$, $q=5$, $3^{\text{*}} = 2$, $5^{\text{*}}= -1$ and $l^{\text{*}}=4$. This gives us $j_{1}= 0$ and $k_{1}=1$. Then, `b' is $ \alpha_{5}\alpha_{1}(\rho)$ and `c' is $\alpha_{3}\alpha_{0}(\rho)$. The bridge is the curve with endpoints `b' and `c'. One sees that the curve $\mathcal{E}_{2}$ is the geodesic joining $\frac{1}{5}$ and $\frac{1}{3}$. The other geodesic $\mathcal{E}_{1}$ joins $\frac{1}{6}$ and $\frac{1}{5}$.
\begin{figure}[!ht]
  \begin{center}
    \includegraphics[width=0.7\linewidth]{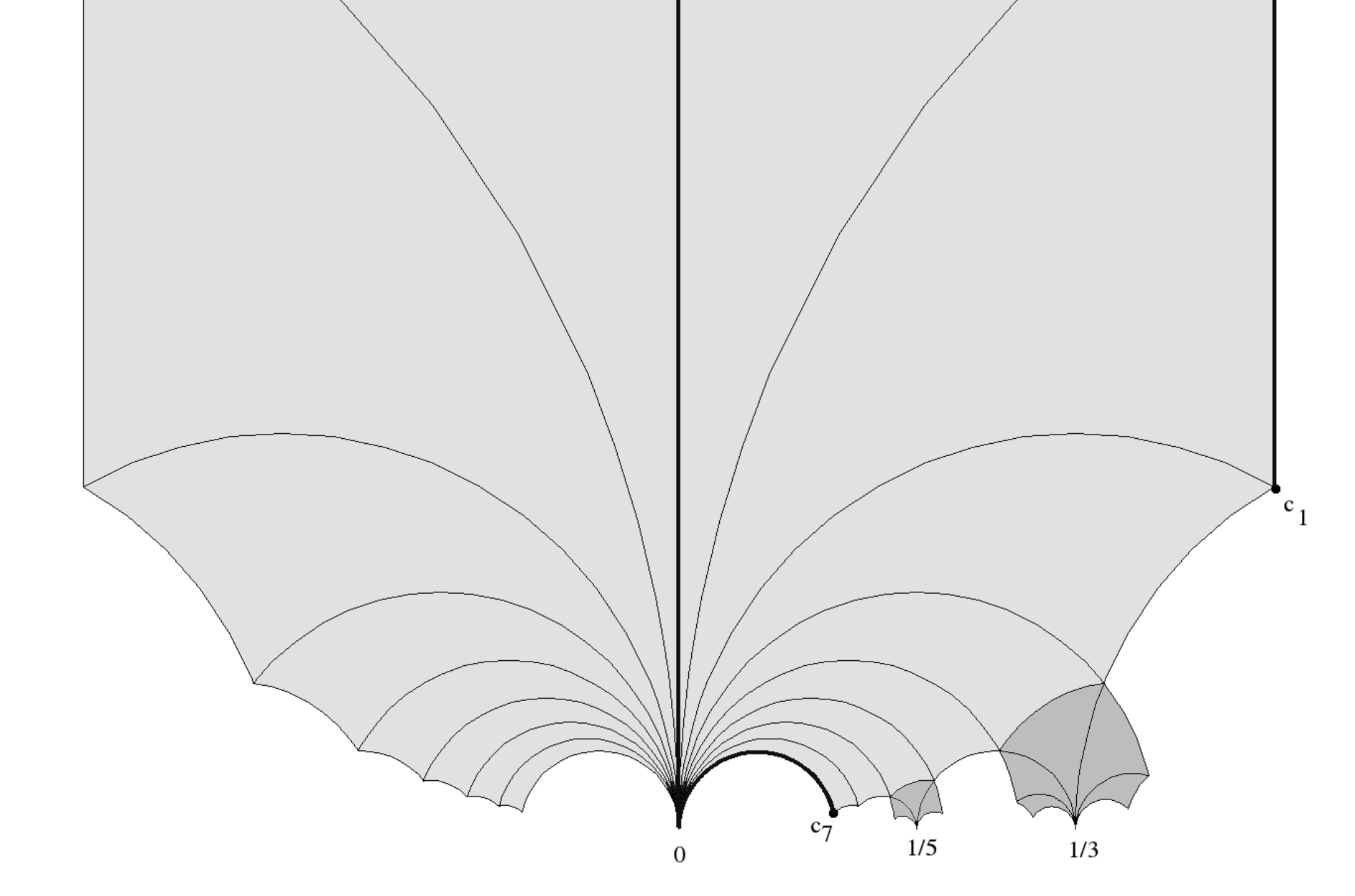}
  \end{center}
  \caption{}
\end{figure}

\begin{figure}[!ht]
  \begin{center}
\includegraphics[width=0.7\linewidth]{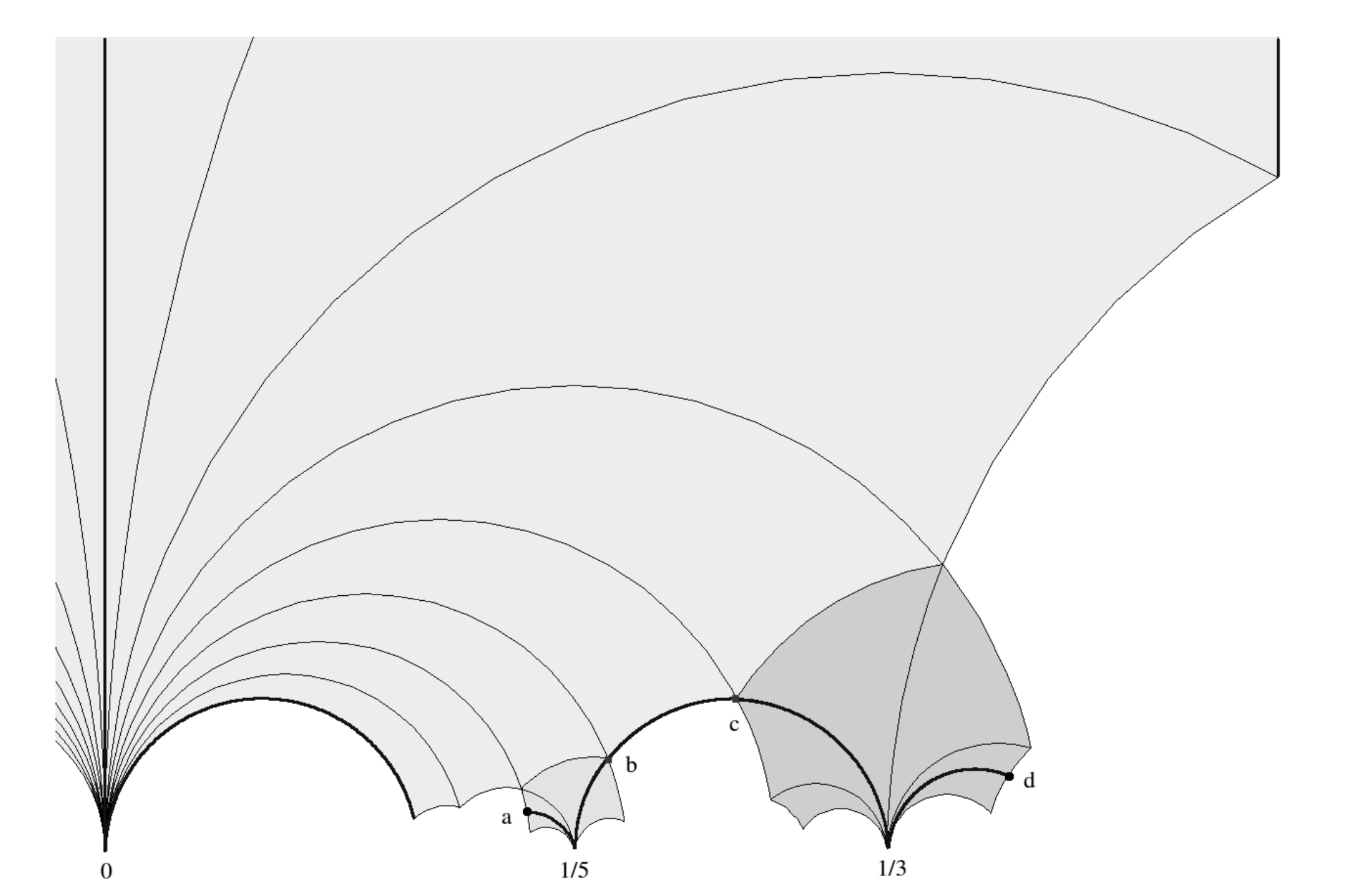}
  \end{center}
  \caption{}
\end{figure}

\subsection{Symmetric fundamental domains: the cocompact case}\

We consider anisotropic indefinite ternary quadratic forms for which the integral automorphism group or a suitable subgroup is arithmetic. Associated with the chosen group is a cocompact arithmetic subgroup of $\PSL(2,\mathbb{R})$ , giving us the compact surface $\mathbb{Y}$ in $\mathbb{H}$. By  choosing the forms suitably, we ensure that $\mathbb{Y}$ has a reflection symmetry $\sigma$ for which the fixed set is composed of closed geodesics in $\mathbb{H}$.

We first give some preliminaries: let $f(x_{1},x_{2},x_{3})= -a_{1}x_{1}^{2}+ a_{2}x_{2}^{2}+a_{3}x_{3}^{2}$ with $a_{i}\in \mathbb{N}$ be an integral anisotropic quadratic form. Associated to $f$ is the two-sheeted hyperboloid given by $f(\mathbf{x})=-k$ where $k>0$ is fixed.  Define the inner-product $({\vec x},{\vec y}) = -a_{1}x_{1}y_{1}+a_{2}x_{2}y_{2}+a_{3}x_{3}y_{3}$ and choose a sheet, say with $x_{1}>0$ and denote it by $H^{\text{+}}$, and the other sheet by $H^{-}$.  Let $O_{f}(\mathbb{R})$ be the real othogonal group of $f$, the group of  $3\times 3$ real matrices whose action fixes $f$.  An element of $O_{f}$ either maps $H^{\text{+}}$ to itself or otherwise to $H^{-}$.  We denote the subgroup (of index two) that maps $H^{\text{+}}$  to itself by $O_{f}^{\text{+}}$, and let $SO_{f}^{\text{+}}$  be those elements with determinant one.  Then $SO_{f}^{\text{+}}(\mathbb{R})$ acts on this sheet as a group of isometries and the  integral automorphisms $SO_{f}^{\text{+}}(\mathbb{Z})$  of $f$ act as a discrete subgroup on this space. Vinberg's algorithm \cite{Vinberg} is used to determine a polygonal fundamental domain for  the subgroup $R_{f}$ in $SO_{f}^{\text{+}}(\mathbb{Z})$ generated by all reflections .  The integrality conditions imposed above are used in Vinberg's algorithm to obtain {\it root vectors} which in turn determine hyperplanes that cut the hyperboloid sheet to give the sides of the polygon. The process of finding the root vectors is algorithmic and involves solving some diophantine inequalities determined by the form $f$; the process always terminates. By mapping $H^{\text{+}}$ to $\mathbb{H}$ one then creates a discrete reflection group $\hat{\Gamma}$ in $\PSL(2,\mathbb{R})$ together with a fundamental domain. From $\hat{\Gamma}$, we construct our discrete groups of interest $\Gamma$ in $\PSL(2,\mathbb{R})$ with $\sigma$-symmetric fundamental domains.  For suitable $f$, the reflective subgroup $R_{f}$ is of finite index in $SO_{f}^{\text{+}}(\mathbb{Z})$ so that a fundamental domain for  $SO_{f}^{\text{+}}(\mathbb{Z})$ can also be obtained. Since $f$ is anisotropic, all the fundamental domains obtained are compact.    By general principles relating $SO_{f}^{\text{+}}$ to quaternion division algebras one can create an arithmetic cocompact  subgroup $\Gamma_{f}$ in $\PSL(2,\mathbb{R})$, obtained from $SO_{f}^{\text{+}}(\mathbb{Z})$  by a suitable representation. In the explicit constructions below, one finds that $\Gamma_{f}=\Gamma$, giving us arithmetic cocompact groups with $\sigma$-symmetric fundamental domains. We then analyse the fixed set for these fundamental domains.

For $p \equiv 3\ ({\rm mod}\ 4)$ a prime number, consider the anisotropic form $f(x_{1},x_{2},x_{3}) = -px_{1}^{2}+ x_{2}^{2}+x_{3}^{2}$ (or the form  $f(x_{1},x_{2},x_{3})= -x_{1}^{2}+ px_{2}^{2}+px_{3}^{2}$, for which the analysis is similar).   The Lie group $SO_{f}(\mathbb{R})$ is a linear algebraic group over $\mathbb{Q}$ and  has two connected components. There is  an isomorphism between $\PSL(2,\mathbb{R})$ and $SO_{f}^{\text{+}}(\mathbb{R})$  the connected 1-component, given as follows:  to each $\left( \begin{smallmatrix} \alpha&\beta\\ \gamma&\delta \end{smallmatrix}\right ) \in \PSL(2,\mathbb{R})$  is associated the matrix 

\[
\left(\begin{matrix} \frac{1}{2}(\alpha^{2}+\beta^{2}+\gamma^{2}+\delta^{2})&\frac{1}{\sqrt{p}}(\alpha\beta +\gamma\delta)&\frac{1}{2\sqrt{p}}(\alpha^{2}-\beta^{2}+\gamma^{2}-\delta^{2})\\\sqrt{p}(\alpha\gamma+\beta\delta)&\alpha\delta+\beta\gamma&\alpha\gamma-\beta\delta\\ \frac{\sqrt{p}}{2}(\alpha^{2}+\beta^{2}-\gamma^{2}-\delta^{2})&\alpha\beta-\gamma\delta&\frac{1}{2}(\alpha^{2}-\beta^{2}-\gamma^{2}+\delta^{2})\end{matrix}\right),
\]
in $SO_{f}^{\text{+}}(\mathbb{R})$ (see \cite{FrickeKlein}) Sec 3, Chap.2 ; \cite{Plesken}). Restricting the latter to $SO_{f}^{\text{+}}(\mathbb{Z})$ gives rise to a discrete cocompact  subgroup $\Gamma_{f}$ of $\PSL(2,\mathbb{R})$ given by $\Gamma_{f} = \Gamma_{f}^{\text{*}}\  \cup\  U\Gamma_{f}^{\text{*}}$, where $\Gamma_{f}^{\text{*}}$ is the subgroup of $\PSL(2,\mathbb{R})$ consisting of all matrices of the type
\begin{equation}\label{e613}
\left( \begin{matrix} \alpha&\beta\\ \gamma&\delta \end{matrix}\right ) = \left( \begin{matrix} \frac{a+\sqrt{p}b}{2}&\frac{c+\sqrt{p}d}{2}\\\frac{-c+\sqrt{p}d}{2}&\frac{a-\sqrt{p}b}{2}\end{matrix}\right)\quad \mathrm{ or} \quad \left( \begin{matrix} \frac{a+\sqrt{p}b}{2}&\frac{c+\sqrt{p}d}{2}\\\frac{c-\sqrt{p}d}{2}&\frac{-a+\sqrt{p}b}{2}\end{matrix}\right),
\end{equation}
where $a,b,c$ and $d$ are integers all of the same parity with $a^{2}-pb^{2}+c^{2}-pd^{2} = \pm 4$ respectively, and $U= \left( \begin{smallmatrix} \frac{1}{\sqrt{2}}&\frac{1}{\sqrt{2}}\\ -\frac{1}{\sqrt{2}}&\frac{1}{\sqrt{2}} \end{smallmatrix}\right )$. Fix integers $\eps_{1}$, $\eps_{2}$ both odd such that $ \eps_{1}^{2}+  \eps_{2}^{2} = 2(p-2)$. Let $\Gamma_{f}^{\text{**}}$ be the subgroup of $\Gamma_{f}^{\text{*}}$ consisting of those elements in \eqref{e613} of the first kind. Then $\Gamma_{f}^{\text{**}}$ is of index two in $\Gamma_{f}^{\text{*}}$ with $\Gamma_{f}^{\text{*}}=\Gamma_{f}^{\text{**}}\sqcup \eps\Gamma_{f}^{\text{**}}$, where $\eps = \left( \begin{smallmatrix} \half(\eps_{1}+\sqrt{p})&\half(\eps_{2}+\sqrt{p})\\\half(\eps_{2}-\sqrt{p})&\half(-\eps_{1}+\sqrt{p})\end{smallmatrix} \right)$

Let $G_{f}$ be the spin double cover of $SO_{f}(\mathbb{R})$ defined over $\mathbb{Q}$; it is isomorphic to $SO_{f}^{\text{+}}(\mathbb{R})$. Let $D_{f}$ denote the quaternion division algebra over $\mathbb{Q}$ containing elements ${\bf u}= u_{0}+u_{1}E_{1}+u_{2}E_{2}+u_{3}E_{3}$, where $E_{1}^{2}=-1$ and $E_{2}^{2}=E_{3}^{2}=p$. Then $G_{f}$ is composed of  the elements of norm 1 in $D_{f}$, namely those ${\bf u}$ with $N({\bf u})= u_{0}^{2}+ u_{1}^{2}-p u_{3}^{2}-p u_{0}^{2}=1$. Let $G_{f}(\mathbb{Z})$ consist of the elements of $G_{f}$ with $u_{i}$ integers (the Litschitz integral quaternions of norm one).
Let 
\[
\mathcal{E}_{1}=\left( \begin{matrix} 0&-1\\1&0\end{matrix}\right); \quad \mathcal{E}_{2}=\left( \begin{matrix} 0&\sqrt{p}\\\sqrt{p}&0\end{matrix}\right) \quad \text{and} \quad \mathcal{E}_{3}=\left( \begin{matrix} -\sqrt{p}&0\\0&\sqrt{p}\end{matrix}\right).
\]
To each ${\bf u}$ is associated the matrix $M({\bf u})=u_{0}I + u_{1}\mathcal{E}_{1} +u_{2}\mathcal{E}_{2} +u_{3}\mathcal{E}_{3}$ in $\PSL(2,\mathbb{R})$. Then, $G_{f}(\mathbb{Z})$ is mapped to the subgroup $\Gamma_{f}^{\text{**}}(+)$ of $\Gamma_{f}^{\text{**}}$ with $a,b,c$ and $d$ all even, so that $\Gamma_{f}^{\text{**}}(+)$  is arithmetic.  Now fix integers $\eps_{3}$, $\eps_{4}$ both odd such that $ \eps_{3}^{2}+  \eps_{4}^{2} = 2(p+2)$and let $\eta = \left( \begin{smallmatrix} \half(\eps_{3}+\sqrt{p})&\half(\eps_{4}+\sqrt{p})\\\half(-\eps_{4}+\sqrt{p})&\half(\eps_{3}-\sqrt{p})\end{smallmatrix} \right)$. Then, $\Gamma_{f}^{\text{**}}(+)$ is of index two in $\Gamma_{f}^{\text{**}}$ since $\Gamma_{f}^{\text{**}}= \Gamma_{f}^{\text{**}}(+)\  \sqcup \ \eta\Gamma_{f}^{\text{**}}(+)$. Thus $\Gamma_{f}^{\text{**}}$, and so $\Gamma_{f}$, are all arithmetic (note that $\Gamma_{f}^{\text{**}}$ is associated with the Hurwitz integral quaternions of norm one).
\subsubsection{The form ${\bf f(x_{1},x_{2},x_{3})= -3x_{1}^{2}+x_{2}^{2}+x_{3}^{2}\ }$}\

By Vinberg's algorithm, there are three root vectors $\vec{v}_{1}=<0,0,-1>$, $\vec{v}_{2}=<0,-1,1>$ and $\vec{v}_{3}=<1,3,0>$, giving rise to three hyperplanes in $\mathbb{R}^{3}$ namely $H_{1}=\{\vec{x}: x_{3}=0\}$, $H_{2}=\{\vec{x}: x_{2}-x_{3}=0\}$ and $H_{3}=\{\vec{x}: -x_{1}+ x_{2}=0\}$ respectively. Putting $x=\frac{x_{3}}{\sqrt{3}x_{1}}$ and $y=\frac{x_{2}}{\sqrt{3}x_{1}}$ maps the sheet $({\vec x},{\vec x})=-k$ with $x_{1}>0$ to the interior of the Klein unit disc $x^{2}+y^{2}=1-\frac{k}{3x_{1}^{2}}<1$. Under this map, the three hyperplanes give rise to three line segments in the unit disc namely $(1)$ $x=0$, \ $(2)$ $x=y$ \ and \ $(3)$ $y=\frac{1}{\sqrt{3}}$\ with corresponding endpoints on the boundary $(0,\pm1)$, \ $\pm(\frac{1}{\sqrt{2}},\frac{1}{\sqrt{2}})$ \ and \ $(\pm\sqrt{\frac{2}{3}},\frac{1}{3})$ \ respectively. Putting $r=\frac{1+\sqrt{3}}{\sqrt{2}}$ and mapping these to $\mathbb{H}^{2}$ gives us the three geodesics
\[
l_{\mathfrak{a}}: x=0\ ;\quad\quad l_{\mathfrak{b}}: (x-1)^{2}+y^{2}=2 \quad\quad \mathrm{ and}\quad\quad l_{\mathfrak{c}}: x^{2}+y^{2}= r^{2},
\]
whose intersection gives us the $(2,4,6)$-triangle as a fundamental domain for the associated (arithmetic) triangle group. We denote the endpoints of $l_{\mathfrak{b}}$ by $s=1+\sqrt{2}$ and $s'=1-\sqrt{2}$, while those of $l_{\mathfrak{c}}$ are $\pm r$.

The three reflections associated with the geodesics are respectively,
\[
R_{\mathfrak{a}}(z) = \overline{z}\ ; \quad\quad R_{\mathfrak{b}}(z) = \frac{\overline{z}+1}{\overline{z}-1} \quad \mathrm{ and} \quad R_{\mathfrak{c}}(z)=\frac{r^{2}}{\overline{z}}.
\]
Let $\hat{\Gamma}$ be the reflection group generated by $R_{\mathfrak{a}}$, $R_{\mathfrak{b}}$ and $R_{\mathfrak{c}}$. Put
\begin{equation}\label{e640}
U=R_{\mathfrak{a}}R_{\mathfrak{b}} = \left( \begin{smallmatrix} \frac{1}{\sqrt{2}}&\frac{1}{\sqrt{2}}\\ -\frac{1}{\sqrt{2}}&\frac{1}{\sqrt{2}} \end{smallmatrix}\right ) \quad\quad \mathrm{ and}\quad\quad V_{r}=R_{\mathfrak{a}}R_{\mathfrak{c}} = \left( \begin{smallmatrix} 0&-r\\ \frac{1}{r}&0 \end{smallmatrix}\right )\ .
\end{equation}
 We see that $R_{\mathfrak{b}}R_{\mathfrak{a}}=U^{-1}$, $R_{\mathfrak{c}}R_{\mathfrak{a}}=V_{r}^{-1}$, $R_{\mathfrak{b}}R_{\mathfrak{c}}=U^{-1}V_{r} = (R_{\mathfrak{c}}R_{\mathfrak{b}})^{-1}$\ . Let $\Gamma \subset \PSLR$ be the subgroup of $\hat{\Gamma}$ consisting of words of even length. Then, $\Gamma$ is generated by $U$ and $V_{r}$ and $\sigma$ is in the normaliser of $\Gamma$ since $\sigma U\sigma=U^{-1}$ and $\sigma V_{r}\sigma=V_{r}$. Moreover, if $\gamma\in \hat{\Gamma}$ is a word of odd length, then $R_{\mathfrak{a}}\gamma$ has even length so that $\gamma\in R_{\mathfrak{a}}\Gamma$. Thus, $\hat{\Gamma} = \Gamma \sqcup R_{\mathfrak{a}}\Gamma$, with $\Gamma$ a subgroup of $\hat{\Gamma}$ of index two (note that $\Gamma=\hat{\Gamma}\cap\PSLR$).  Since $\Gamma_{f}$ in \eqref{e613} is generated by $U$ and $V_{r}$\ (see \cite{Magnus}, page 127), it follows that $\Gamma$ is a cocompact arithmetic subgroup of $\PSL(2,\mathbb{R})$.

A fundamental domain for $\Gamma$ can be obtained from that of $\hat{\Gamma}$ by reflection on the imaginary axis, giving us a $(2,6,6)$-triangle $\mathfrak{F}$. We denote the segments $l_{*}\cap\mathfrak{F}$ by their respective subscripts (see Figure 9), and their $\sigma$-reflections with a dash. Then, we have the equivalence of the sides $\mathfrak{b}\sim \mathfrak{b}'$ using $U$ and $\mathfrak{c}\sim \mathfrak{c}'$ using $V_{r}$, giving us a surface of genus zero. Moreover this implies that $\mathrm{ Fix}_{\Gamma}(\sigma,\mathfrak{F}) =  \mathfrak{a}\cup \mathfrak{b} \cup \mathfrak{c}$. Since $V_{r}(i0)=i\infty$ and  $U^{2}(s)= s'$, we see that the imaginary axis and $l_{\mathfrak{b}}$ are closed geodesics. Next, putting  $P:= (V_{r}U)^{-3}= \left( \begin{smallmatrix} 1&-\sqrt{2}r\\ \frac{\sqrt{2}}{r}&-1 \end{smallmatrix}\right )$, we see that the same is true for $l_{\mathfrak{c}}$ since $P(r)=-r$. We now show that all these geodesics are mapped modulo $\Gamma$ piecewise onto $\mathfrak{a}$, $\mathfrak{b}$ and $\mathfrak{c}$ respectively.

For any number $\tau>0$ we can define $V_{\tau}$ by replacing $r$ with $\tau$ in \eqref{e640}, giving an involution of $\PSLR$. We have $S=U^{2}\in \Gamma$, $M_{\tau,l}:=(U^{2}V_{\tau})^{l}$ satisfies $M_{\tau,l}(z)=\frac{1}{\tau^{2l}}z$, and $N_{\tau,l}:=V_{\tau}(U^{2}V_{\tau})^{l}$ satisfies $N_{\tau,l}(z)= \tau^{2l+2}Sz$. Consider the ray $\{iy: \ \  y\geq 1\}$. We decompose it into subintervals $[1,\tau]\cup[\tau,\tau^{2}]\cup[\tau^{2},\tau^{3}]\cup \ldots$\ . Then if $y\in[\tau^{k},\tau^{k+1}]$, we have
\begin{itemize}
\item[(i)] for even $k$, $M_{\tau,\frac{k}{2}}(iy)\in \mathfrak{a} $, and
\item[]
\item[(ii)]  for odd $k$, $N_{\tau,\frac{k-1}{2}}(iy) \in \mathfrak{a}$,
\end{itemize}
with a similar calculation holding for $y<1$. Since $V_{r}$, $M_{r,l}$ and $N_{r,l}$ are in $\Gamma$ for any integer $l$, applying this argument with $\tau = r$ shows that the closed geodesic $\Re(z)=0$ is mapped piecewise onto the segment $\mathfrak{a}$.

We now consider the two geodesics that make up the segments $\mathfrak{b}$ and $\mathfrak{c}$ of $\mathrm{ Fix}_{\Gamma}(\sigma,\mathfrak{F})$:
\begin{itemize}
\item[(1)] For the arc $\mathfrak{c}$, the geodesic $l_{\mathfrak{c}}: x^{2}+y^{2}= r^{2}$ is mapped to the imaginary axis by $\eta_{\mathfrak{c}} = \left( \begin{smallmatrix} \frac{1}{r}&1\\ -\frac{1}{r}&1 \end{smallmatrix}\right )$, sending $-r \mapsto 0$, $ir \mapsto i$ and $r \mapsto i\infty$. The other endpoint of $\mathfrak{c}$, denoted by $w=re^{i\frac{\pi}{4}}$ is mapped to $is$. We then partition the imaginary axis as done above with $\tau$ replaced with $s$. Applying the maps $M_{s,l}$ and $N_{s,l}$ as above for the appropriate subintervals (with endpoints powers of $s$), and then applying the map $\eta_{\mathfrak{c}}^{-1}$, we see that subintervals of $l_{\mathfrak{c}}$ are mapped onto $\mathfrak{c}$ using the composite maps $\eta_{\mathfrak{c}}^{-1}M_{s,l}\eta_{\mathfrak{c}}$ and $\eta_{\mathfrak{c}}^{-1}N_{s,l}\eta_{\mathfrak{c}}$. To show these latter maps are in $\Gamma$, we write $\eta_{\mathfrak{c}}^{-1}M_{s,l}\eta_{\mathfrak{c}}= \left((\eta_{\mathfrak{c}}^{-1}U^{2}\eta_{\mathfrak{c}})(\eta_{\mathfrak{c}}^{-1}V_{s}\eta_{\mathfrak{c}})\right)^{l}$, and $\eta_{\mathfrak{c}}^{-1}N_{s,l}\eta_{\mathfrak{c}}=(\eta_{\mathfrak{c}}^{-1}V_{s}\eta_{\mathfrak{c}})\left((\eta_{\mathfrak{c}}^{-1}U^{2}\eta_{\mathfrak{c}})(\eta_{\mathfrak{c}}^{-1}V_{s}\eta_{\mathfrak{c}})\right)^{l}$. Since $\eta_{\mathfrak{c}}^{-1}U^{2}\eta_{\mathfrak{c}} = V_{r}\in \Gamma$ and $(\eta_{\mathfrak{c}}^{-1}V_{s}\eta_{\mathfrak{c}}) = (V_{r}U)^{-3}\in \Gamma$, we see that $l_{\mathfrak{c}}$ is a closed geodesic, being  mapped piecewise onto the arc $\mathfrak{c}$.
\end{itemize}

\begin{figure}[!ht]
  \begin{center}
\includegraphics[width=0.7\linewidth]{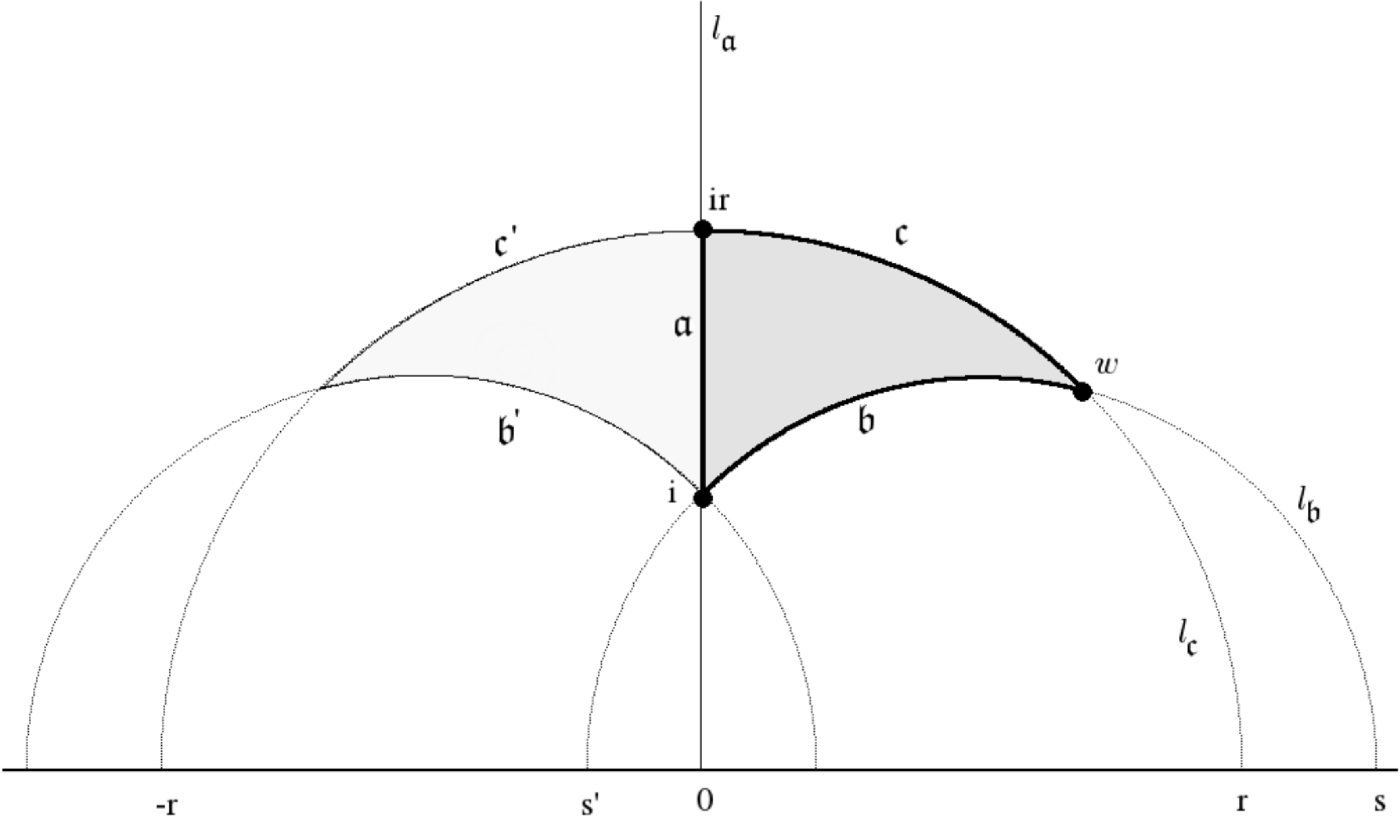}
  \end{center}
  \caption{}
\end{figure}

\begin{itemize}
\item[(2)] For the arc $\mathfrak{b}$, the argument is similar to that of $\mathfrak{c}$. The geodesic $l_{\mathfrak{b}}: (x-1)^{2}+y^{2}=2$ is mapped to the imaginary axis by $\eta_{\mathfrak{b}} = \left( \begin{smallmatrix} s&1\\ -1&s \end{smallmatrix}\right )$, sending $s' \mapsto 0$, $i \mapsto i$ and $s \mapsto i\infty$. The point $w$ is sent to a point $iq$ with $q>1$. Then as above, we need to evaluate $\eta_{\mathfrak{b}}^{-1}M_{q,l}\eta_{\mathfrak{b}}$ and $\eta_{\mathfrak{b}}^{-1}N_{q,l}\eta_{\mathfrak{b}}$, for which it suffices to consider $\eta_{\mathfrak{b}}^{-1}U^{2}\eta_{\mathfrak{b}}$ and $\eta_{\mathfrak{b}}^{-1}V_{q}\eta_{\mathfrak{b}}$. One has $\eta_{\mathfrak{b}}^{-1}U^{2}\eta_{\mathfrak{b}}= S=U^{2}\in \Gamma$. Next, $\eta_{\mathfrak{b}}^{-1}V_{q}\eta_{\mathfrak{b}} = \left( \begin{smallmatrix} 1&-q_{1}\\ q_{2}&-1 \end{smallmatrix}\right )$, where $q_{1}= \frac{s^{2}q+q^{-1}}{s(q-q^{-1})}$ and $q_{2}= \frac{s^{2}q^{-1}+q}{s(q-q^{-1})}$. Writing $q= -i\eta_{\mathfrak{b}}w$, we get $q_{1}= 1+\sqrt{3}=\sqrt{2}r$ and $q_{2}= \sqrt{3}-1= \sqrt{2}r^{-1}$, so that $\eta_{\mathfrak{b}}^{-1}V_{q}\eta_{\mathfrak{b}}= (V_{r}U)^{-3}\in \Gamma$ as in part $(1)$ above. This shows that $l_{\mathfrak{b}}$ is a closed geodesic, being  mapped piecewise onto the arc $\mathfrak{b}$. 
\end{itemize}


\subsubsection{The form ${\bf f(x_{1},x_{2},x_{3})= -7x_{1}^{2}+x_{2}^{2}+x_{3}^{2}\ }$}\

Vinberg's algorithm gives four root vectors $\vec{v}_{1}= <0,0,-1>$, $\vec{v}_{2}= <0,-1,1>$, $\vec{v}_{3}= <1,3,0>$ and $\vec{v}_{4}= <1,2,2>$ giving rise to the four hyperplanes
\begin{align*} 
\begin{split}
&H_{1}= \left\{\vec{x}: x_{3}=0\right\}, \qquad \qquad  H_{2}= \left\{\vec{x}: x_{2}-x_{3}=0\right\},\\ &H_{3}= \left\{\vec{x}: 7x_{1}-3x_{2}=0\right\},\quad H_{4}=\left\{\vec{x}: 7x_{1}-2x_{2}-2x_{3}=0\right\}.
\end{split}
\end{align*}
 Mapping to the Klein disc using $x=\frac{x_{2}}{\sqrt{7}x_{1}}$ and $y=\frac{x_{3}}{\sqrt{7}x_{1}}$ transforms the hyperplanes to the four lines
\[
(1)\ x=0\quad; \quad (2)\ x=y\quad; \quad(3)\ y=\frac{\sqrt{7}}{3}\quad \mathrm{ and}\quad (4)\ x+y=\frac{\sqrt{7}}{2}.
\]
Finally, mapping to $\mathbb{H}$ gives us the following four geodesics (Figure 9):
\begin{align}\label{n50}
\begin{split}
(\mathrm{ i})&\quad l_{\mathfrak{a}}:\ x=0\quad ;\quad\quad\quad\quad (\mathrm{ ii})\quad l_{\mathfrak{b}}:\ (x-1)^{2}+y^{2}=2\ ;\\
 (\mathrm{ iii})&\quad l_{\mathfrak{c}}:\ x^{2}+y^{2}=\vartheta ^{2}\quad ;\quad (\mathrm{ iv})\quad l_{\mathfrak{d}}:\ (x-\frac{2}{3}\eta)^{2}+y^{2}=\frac{1}{9}\eta^{2},
 \end{split}
\end{align}
where $\vartheta = \frac{3+\sqrt{7}}{\sqrt{2}}$ and  $\eta=(2+\sqrt{7})$. 

The reflection maps corresponding to the reflection on the four geodesics above are
\begin{align}
\begin{split}
(\mathrm{ i})\quad R_{\mathfrak{a}}(z)&= -\overline{z}\quad ;\quad\quad  \ (\mathrm{ ii})\quad R_{\mathfrak{b}}(z)= \frac{\overline{z}+1}{\overline{z}-1}\ ;\\
(\mathrm{iii})\quad R_{\mathfrak{c}}(z)&= \frac{\vartheta^{2}}{\overline{z}}\quad ;\quad\quad(\mathrm{iv})\quad R_{\mathfrak{d}}(z)= \frac{2\overline{z}-3\eta}{3\eta '\overline{z}-2}\ ,
\end{split}
\end{align}
with $\eta'= 2-\sqrt{7}$.

\begin{figure}[!ht]
  \begin{center}
\includegraphics[width=0.7\linewidth]{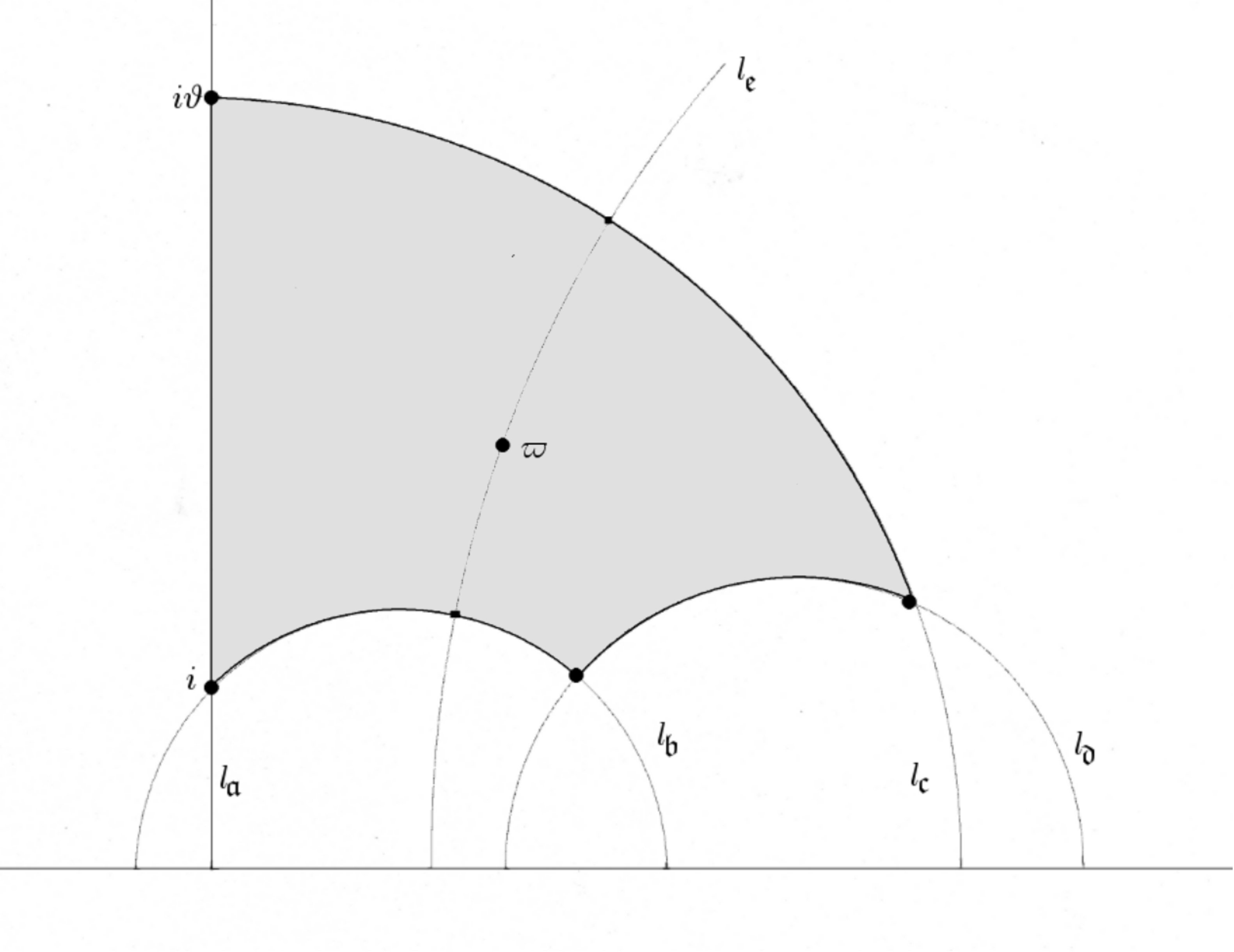}
  \end{center}
  \caption{}
\end{figure}

Now let $\vec{w}=<2,1,1>$ so that $f(\vec{w})=-2$, which implies that $\vec{w}$ lies on the chosen sheet of the hyperboloid (as opposed to the vectors $\vec{v}_{i}$ above). Reflection on $\vec{w}$ is reflection about a point (a $180^{\mathrm{ o}}$ rotation) and leads to a Cartan involution. If $R_{\vec{w}}$ denotes this reflection, then one has $R_{\vec{w}}(\vec{v}_{3})=-\vec{v}_{2}$ and $R_{\vec{w}}(\vec{v}_{4})=-\vec{v}_{1}$. In the upper-half plane model, $\vec{w}$ corresponds to the point $\varpi:=\frac{\eta}{3}(1+i\sqrt{2})$. The reflection about this point $\varpi$ we denote by $X$ and is given by the matrix in $\PSLR$
\[
X= \left( \begin{smallmatrix} \frac{1}{\sqrt{2}}&-\frac{\eta}{\sqrt{2}}\\ -\frac{\eta'}{\sqrt{2}}&-\frac{1}{\sqrt{2}} \end{smallmatrix}\right ).
\]

The geodesic 
\begin{equation}\label{e715}
l_{\mathfrak{e}}:\ \left(x- \sqrt{\frac{7}{2}}\vartheta \right)^{2}+y^{2}= \frac{5}{2}\vartheta ^{2}\ ,
\end{equation}
is orthogonal to both $l_{\mathfrak{b}}$ and $l_{\mathfrak{c}}$ and it contains the point $\varpi$.   It is mapped onto itself by $X$, sending $\varpi$ to itself and the intersections $l_{\mathfrak{b}}\cap l_{\mathfrak{e}} \mapsto l_{\mathfrak{c}}\cap l_{\mathfrak{e}}$ (we  will denote the intersection of two of the geodesics, say $l_{i}\cap l_{j}$ by $\tau_{ij}$) so that $X: \tau_{\mathfrak{be}} \mapsto \tau_{\mathfrak{ce}}$.

Let $\hat{\Gamma}$ be the group generated by the reflections $R_{\mathfrak{a}}, \ldots, R_{\mathfrak{d}}$. As in the previous example, we consider the subgroup of $\hat{\Gamma}$  (in $\PSLR$) of even length words, denoted by $\Gamma^{\text{*}}$. It is generated by $R_{\mathfrak{a}}R_{\mathfrak{b}}$, $R_{\mathfrak{a}}R_{\mathfrak{c}}$ and $R_{\mathfrak{a}}R_{\mathfrak{d}}$ and is a subgroup of $\hat{\Gamma}$  of index two, namely $\hat{\Gamma}=\Gamma^{\text{*}}\sqcup R_{\mathfrak{a}}\Gamma^{\text{*}}$. A fundamental domain $\mathfrak{F}^{\text{*}}$  of $\Gamma^{\text{*}}$ is the shaded region in Figure 10, together with its $\sigma$-reflection; it contains  $l_{\mathfrak{e}}\cap\mathfrak{F}^{\text{*}}$  (and so  $\varpi$) as  interior points and is mapped onto itself  by $X\notin \Gamma^{\text{*}}$. A similar statement holds for the (reflected) arc of $\sigma l_{\mathfrak{e}}$ with $\tilde{X}$ taking the place of $X$.

Now define $\Gamma = \Gamma^{\text{*}}\sqcup X\Gamma^{\text{*}} \neq \Gamma^{\text{*}}$. Using the relations $R_{\mathfrak{d}}=XR_{\mathfrak{a}}X$ and $R_{\mathfrak{c}}=XR_{\mathfrak{b}}X$, we conclude that $X\hat{\Gamma}X=\hat{\Gamma}$, so that $X\Gamma X=\Gamma$. These imply that $\Gamma$ is a subgroup  of $\PSLR$,  and contains $\Gamma^{\text{*}}$ as a subgroup of index two. A set of generators for $\Gamma$ is $X$,  $U$ and $V_{\vartheta}$. Here we have used the fact that $UV_{\vartheta}X\tilde{X}=I$, $R_{\mathfrak{a}}R_{\mathfrak{d}}= (R_{\mathfrak{a}}X)^{2} = \tilde{X}X$, $R_{\mathfrak{a}}R_{\mathfrak{b}} = U$ and $R_{\mathfrak{a}}R_{\mathfrak{c}}=V_{\vartheta}$, with $U$ and $V_{*}$ as given after \eqref{e640}. Finally, as in the previous example, $\Gamma_{f}$ in \eqref{e613} is generated by $X$,  $U$ and $V_{\vartheta}$ (see \cite{Magnus}, page 128), $\Gamma$ is cocompact and arithmetic.

The fundamental domain $\mathfrak{F}$ of $\Gamma$, obtained from that of $\Gamma^{\text{*}}$ is bounded by  the curves $l_{\mathfrak{b}}$, $l_{\mathfrak{c}}$, their reflections under $\sigma$, together with $l_{\mathfrak{e}}$ and its reflection under $\sigma$.  As in the previous subsection, we denote the arcs of the polygon by the subscripts of the corresponding geodesics (see Figure 11).  We have the equivalence of sides $\mathfrak{b}\sim \mathfrak{b}' = \sigma \mathfrak{b}$ using $U$, and $\mathfrak{c}\sim \mathfrak{c}' = \sigma \mathfrak{c}$ using $V_{\vartheta}$.  Since the arc $\mathfrak{e}$ (and $\mathfrak{e}'$) is in the interior of $\mathfrak{F}^{\text{*}}$, and since $X$  (and $\tilde{X}$)  is in the normaliser of $\Gamma$, preserving the arc, one sees that $\mathfrak{e}$ is not $\Gamma$-equivalent to $e'$. Thus, $\mathrm{ Fix}_{\Gamma}(\sigma;\mathfrak{F})= \mathfrak{a}\cup \mathfrak{b}\cup \mathfrak{c}$ is a closed curve with common endpoints  $\tau_{\mathfrak{be}}\sim \tau_{\mathfrak{ce}}$. The geodesics $l_{\mathfrak{a}},\ l_{\mathfrak{b}},\ l_{\mathfrak{c}}$ and $l_{\mathfrak{e}}$ are closed, with their corresponding endpoints mapped to one  another by $V_{\vartheta}$, $S=U^{2}$,  $V_{\vartheta}$ and $X$, respectively. 

As in the previous subsection, the geodesic $l_{\mathfrak{a}}$ is mapped onto $\mathfrak{a}$ using the  map $M_{\vartheta,\frac{k}{2}}$ and $N_{\vartheta,\frac{k-1}{2}}$. We now focus on $l_{\mathfrak{b}}$ and $l_{\mathfrak{c}}$, for which we have the interplay between three maps: $S: l_{\mathfrak{b}} \mapsto l_{\mathfrak{b}}\ $, $V_{\vartheta}: l_{\mathfrak{c}} \mapsto l_{\mathfrak{c}}\ $ and $X: l_{\mathfrak{b}} \mapsto l_{\mathfrak{c}}\ $. One sees that $l_{\mathfrak{c}}$ can be partitioned by starting with $\mathfrak{c}$ and going clockwise; one has four  consecutive arcs $\mathfrak{c}$, $X\mathfrak{b}$, $WX\mathfrak{b}$ and $W\mathfrak{c}$, which together (in order) we denote by $l^{\text{*}}_{\mathfrak{c}}$, where $W=XSX$. Then all of $l_{\mathfrak{c}}$ is obtained from $l^{\text{*}}_{\mathfrak{c}}$ as a sum of $(WV_{\vartheta})^{k}l^{\text{*}}_{\mathfrak{c}}$, with $k\in \mathbb{Z}$. Similarly, for $l_{\mathfrak{b}}$ we start with the arcs $\mathfrak{b}$, $X\mathfrak{c}$, $W'X\mathfrak{c}$ and $W'\mathfrak{b}$, denoted collectively by $l^{\text{*}}_{\mathfrak{b}}$, with $W'=XV_{\vartheta}X$. Then $l_{\mathfrak{b}}$ is obtained from $l^{\text{*}}_{\mathfrak{b}}$ as a sum of $(W'S)^{k}l^{\text{*}}_{\mathfrak{b}}$, with $k\in \mathbb{Z}$. The arc $l_{\mathfrak{c}}^{\text{*}}$ geometrically can  be viewed as traversing $\mathfrak{c}$ to $\tau_{\mathfrak{ce}}$, then to $\mathfrak{b}$ from $\tau_{\mathfrak{b}}$, then on $\mathfrak{b}$ backwards and finally $\mathfrak{c}$ backward (similarly for $l_{\mathfrak{b}}^{\text{*}}$). Thus, $l_{\mathfrak{b}}$ and also $l_{\mathfrak{c}}$ are piecewise equivalent to the pair of (joined) arcs $\mathfrak{b}$ and $\mathfrak{c}$.

\begin{figure}[!ht]
  \begin{center}
\includegraphics[width=0.7\linewidth]{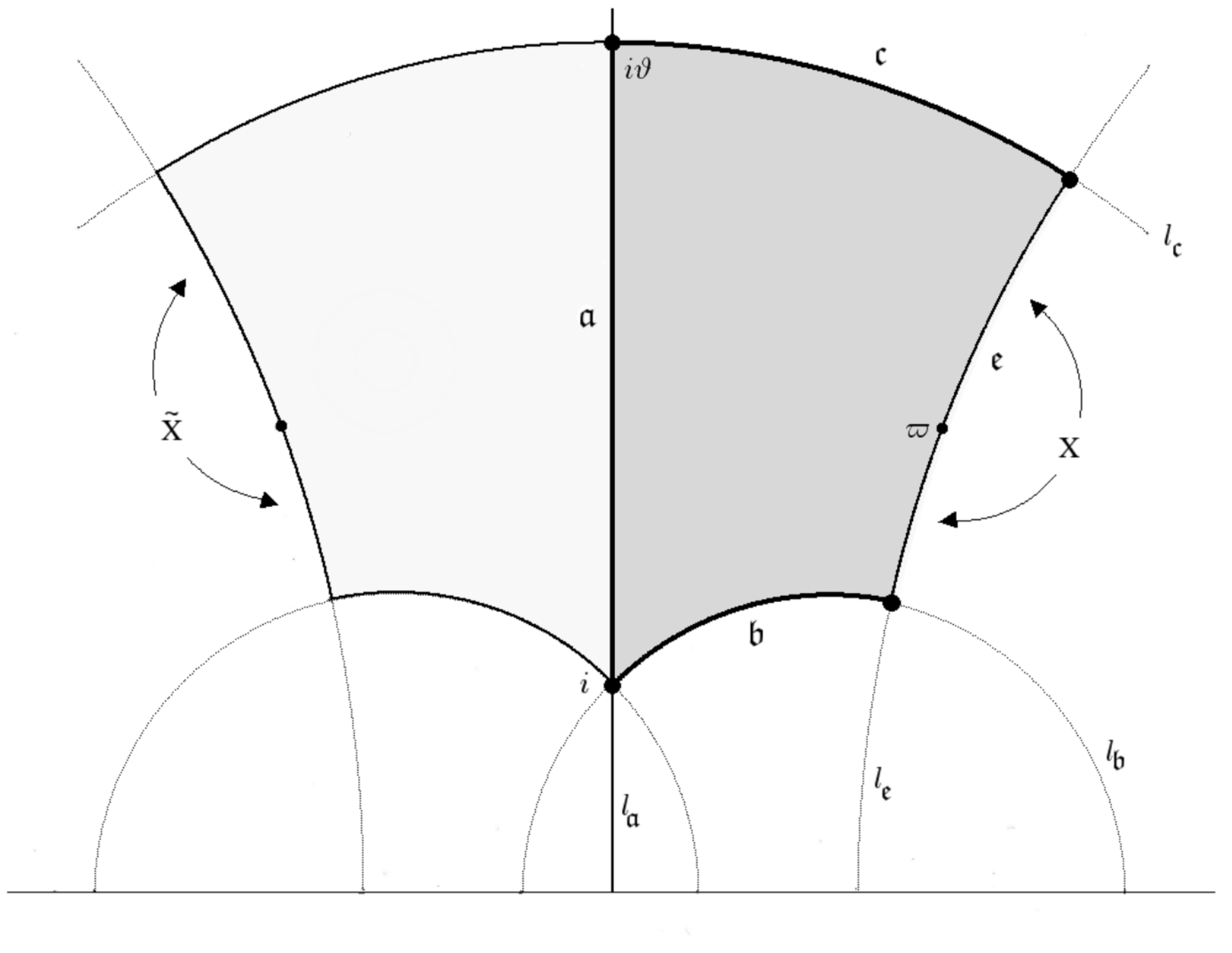}
  \end{center}
  \caption{}
\end{figure}

{\small
\vspace{30pt}
\noindent AMIT GHOSH, Department of Mathematics, Oklahoma State University, Stillwater, OK 74078, USA \hfill {\itshape E-mail address}: ghosh@okstate.edu 
\vspace{12pt}

\noindent ANDRE REZNIKOV, Department of Mathematics, Bar-Ilan University, Ramat Gan, 52900, ISRAEL \hfill {\itshape E-mail address}: reznikov@math.biu.ac.il
\vspace{12pt}

\noindent PETER SARNAK, School of Mathematics, Institute for Advanced Study; Department of Mathematics, Princeton University, Princeton, NJ 08540, USA \\ {\itshape E-mail address}: sarnak@math.ias.edu}

\end{document}